\documentclass[12pt]{amsart}
\pdfoutput=1
\usepackage{amsmath}
\usepackage{amsthm}
\usepackage{amssymb}
\usepackage{mathtools}
\usepackage{amscd}
\usepackage{color}
\usepackage{microtype}
\usepackage{xspace}

\usepackage{tikz} 
\usepackage{tikz-cd}

\usepackage{nicematrix}

\usepackage{hyperref}
\usepackage{cleveref}
\usepackage{enumitem}
\usepackage{mathrsfs}

\hypersetup{
    colorlinks=true,
    linkcolor=black,
    filecolor=black,
    citecolor= black,
    urlcolor=black}

\newcommand{\m}{\mathfrak{m}}
\newcommand{\p}{\mathfrak{p}}
\newcommand{\q}{\mathfrak{q}}
\newcommand{\NN}{\mathbb{N}}
\newcommand{\ZZ}{\mathbb{Z}}
\newcommand{\RR}{\mathbb{R}}
\newcommand{\QQ}{\mathbb{Q}}

\newcommand{\cP}{\mathcal{P}}
\newcommand{\cS}{\mathcal{S}}
\newcommand{\cC}{\mathcal{C}}

\newcommand{\ces}{Ces{\`a}ro\xspace}

\DeclareMathOperator{\lcm}{lcm}

\DeclareMathOperator{\depth}{depth}
\DeclareMathOperator{\rank}{rank}
\DeclareMathOperator{\sdim}{sdim}
\DeclareMathOperator{\fsig}{s}
\DeclareMathOperator{\Fratio}{r_F}
\DeclareMathOperator{\ehk}{e_{HK}}
\DeclareMathOperator{\length}{len}

\DeclareMathOperator{\Ext}{Ext}

\DeclareMathOperator{\Hom}{Hom}
\DeclareMathOperator{\coker}{coker}
\DeclareMathOperator{\asn}{asn}
\DeclareMathOperator{\Spec}{Spec}

\DeclareMathOperator{\Supp}{Supp}
\DeclareMathOperator{\eh}{e}

\DeclareMathOperator{\Ann}{Ann}
\DeclareMathOperator{\Ass}{Ass}
\DeclareMathOperator{\Minh}{Minh}
\DeclareMathOperator{\Min}{Min}

\DeclareMathOperator{\FL}{FL}

\newcommand{\Fmult}[3]{\operatorname{m}_{#1}(#2 \leadsto #3)}
\newcommand{\FRmult}[4]{\operatorname{m}_{#1}^{#4}(#2 \leadsto #3)}
\newcommand{\gsig}[2]{\operatorname{s}(#1 \leadsto #2)}
\newcommand{\csig}[2]{\operatorname{\overline{s}}(#1 \leadsto #2)}
\newcommand{\surjsig}[2]{\operatorname{s_{surj}}(#1 \twoheadrightarrow #2)}
\newcommand*{\Mat}[1][]{\operatorname{\mathrm{Sum}^{#1}_F}}
\DeclareMathOperator{\find}{Ind_F}
\newcommand{\grk}[2]{\mathrm{h}^0_{#1}(#2)}
\newcommand{\fep}[1]{f^{(#1)}}
\newcommand{\ep}[1]{\mathfrak{f}^*_{#1}}

\newcommand{\tr}{\intercal}
\newcommand{\Hh}{\mathrm{H}}
\newcommand{\Id}{\mathrm{Id}}

\newtheorem{theorem}{Theorem}
\numberwithin{theorem}{section}
\newtheorem{lemma}[theorem]{Lemma}
\newtheorem{proposition}[theorem]{Proposition}
\newtheorem*{proposition*}{Proposition}
\newtheorem{question}[theorem]{Question}
\newtheorem{corollary}[theorem]{Corollary}
\newtheorem{conjecture}[theorem]{Conjecture}
\newtheorem{claim}[theorem]{Claim}
\newtheorem*{statement*}{Statement}
\newtheorem*{theorem*}{Theorem}
\newtheorem*{lemma*}{Lemma}
\newtheorem*{fact*}{Fact}
\newtheorem*{claim*}{Claim}

\newtheorem{theoremx}{Theorem}
 
\newtheorem{conjecturex}[theoremx]{Conjecture}

\theoremstyle{definition}
\newtheorem{definition}[theorem]{Definition}
\newtheorem{remark}[theorem]{Remark}
\newtheorem*{definition*}{Definition}
\newtheorem{example}[theorem]{Example}
\newtheorem{notation}[theorem]{Notation}
\newtheorem*{example*}{Example}
\numberwithin{equation}{theorem}

\title{A dynamical approach to Frobenius splitting}
\date{\today}
\author{Alapan Mukhopadhyay}
\address{Institute of Mathematics, CAG,
EPFL SB MATH
MA A2 383 (Bâtiment MA), 
Station 8,
CH-1015 Lausanne,
Switzerland}
\email{alapan.mathematics@gmail.com}
\author{Ilya Smirnov}
\address{BCAM -- Basque Center for Applied Mathematics, Bilbao, Spain \quad and \quad
IKERBASQUE, Basque Foundation for Science, Bilbao, Spain}
\email{ismirnov@bcamath.org}

\begin{document}

\maketitle
\begin{abstract} 
A dynamical system can be associated to the process of decomposing iterated Frobenius pushforwards of modules into indecomposable direct summands. Via this connection, we apply the theory of Markov chains to study asymptotic counts of multiplicities of indecomposable summands in the iterated Frobenius pushforwards of the given module, as the number of iterations grows. Our results recover and extend the existing results on the existence of the (generalized) F-signature due to Hashimoto--Symonds, Hashimoto--Kobayashi, Smith--Van den Bergh, Tucker, and Yao.  We also develop further the theory of F-contributors, the modules for which the multiplicity grows at the maximal possible rate. 

As a key algebraic input, we define and develop strongly F-regular modules, extending Hochster--Huneke's notion of strongly F-regular rings. We use strongly F-regular modules to prove convergence of the relevant sequences and algebraically characterize F-contributors, when the underlying ring admits a module-finite strongly F-regular extension or has dimension at most two.
\end{abstract}


\tableofcontents

\section{Introduction}
The Frobenius endomorphism of a ring $R$ of prime characteristic $p>0$ is one of the central tools in positive-characteristic commutative algebra and algebraic geometry. 
The Frobenius endomorphism, denoted by $F$, induces the \textit{Frobenius pushforward} endofunctor, denoted by $F_*$, on the associated category of finitely generated modules. The study of the indecomposable summands of $F^e_*R$, as $e$ varies, embodies the so-called ``Frobenius splitting methods'' in the study of singularities; see \cite{Smith_F_rational_rings_have_rational_singularities}, \cite{Hara_Watanabe_F_regular_and_F_pur_rings_vs_log_terminal_and_log_canonical_singularities}, the books \cite{Brion_Kumar_Frobenius_splitting_methods_in_geometry_and_representation_theory}, \cite{Schwede_Smith_} and the references therein. 
For example, two fundamental classes of singularities are defined in this way: \emph{F-split} rings are those for which the ring $R$ is a summand of $F^e_*R$ for some (equivalently, all) $e\geq 1$, and \emph{strongly F-regular} rings are characterized in \cite{HY22} by the property that $R$ is a direct summand of $F_*^e M$ for any nonzero torsion-free $R$-module $M$ and all sufficiently large $e$. 
These classes, which originate in \cite{HochsterRoberts, HochsterHuneke}, play an important role in commutative algebra and birational algebraic geometry.
Via \emph{reduction mod $p$} \cite{Hara_Watanabe_F_regular_and_F_pur_rings_vs_log_terminal_and_log_canonical_singularities, HaraYoshida, Takagi_Multiplier, MustataSrinivas}, they correspond (conjecturally, in the case of F-pure singularities) to the \emph{log canonical} and \emph{log terminal} singularities arising in the Minimal Model Program for algebraic varieties of characteristic zero (\cite{HaconXu}).

In order to recall some related results on indecomposable summands of $F^e_*M$ more precisely, let $(R, \m)$ denote a complete noetherian local domain in characteristic $p>0$ with maximal ideal $\m$ and perfect residue field. For finitely generated $R$-modules $M,N$, where $N$ is indecomposable, denote the multiplicity of $N$ in the direct sum decomposition of $F^e_*M$ by $\Fmult{e}{M}{N}$. Such decompositions were first studied by Smith and Van den Bergh (\cite{SmithVandenBergh}).

The maximal rank of a free module summand of $F^e_*M$, i.e.,  $\Fmult{e}{M}{R}$, has been studied within the theory of F-signature originating from \cite{SmithVandenBergh, HunekeLeuschke}. Tucker (\cite{Tucker}) shows that the sequence $(\Fmult{e}{M}{R}/p^{e\dim(R)})_e$ converges; the limit is called the \emph{F-signature} of $M$. The F-signature theory has found wide-ranging applications in singularity theory. For example, the F-signature of the ring characterizes its regularity or strong F-regularity; see \cite{AberbachLeuschke}, \cite{HunekeLeuschke}, \cite{BlickleSchwedeTucker}, \cite{JavierSchwedeTucker}, \cite{Ma_Polstra_Schwede_Tucker_F_signature_under_birational_morphism}.

Free direct summands are special since 
any surjection\footnote{\cite{SmirnovTucker2}
shows that the normalized number of surjections always converges.} onto a free module splits. In comparison to the F-signature theory,  very little is known about the multiplicities $\Fmult{e}{M}{N}$ for general $N$ and existing works are mainly directed towards proving convergence of the sequence $(\Fmult{e}{M}{N}/p^{e\dim(R)})_e$. Notable highlights are Yao's work \cite{Yao05} proving the convergence of  $(\Fmult{e}{M}{N}/p^{e\dim(R)})_e$, when additionally the isomorphism classes of indecomposable summands of $F^e_*M$, as $e$ varies, form a finite set, i.e., when $M$ has the so-called finite F-representation type (FFRT), a notion introduced by Smith--Van den Bergh (\cite{SmithVandenBergh}).
In another context, when $R$ is the invariant ring of a linear action of a finite group on a power series ring, Hashimoto--Symonds have used representation-theoretic tools to study the indecomposable summands of $F^e_*R$ and prove convergence of $(\Fmult{e}{R}{N}/p^{e\dim(R)})_e$ (\cite{HashimotoSymonds}).  However, the methods involved in these or the theory of F-signature do not extend towards studying $\Fmult{e}{M}{N}$ in full generality.

In this article, we bring tools and perspectives from the theory of Markov chains to analyze the indecomposable summands $N$ of $F^e_*M$. This allows us to translate relevant questions in commutative algebra into questions about Markov chains. Using this interplay, we not only derive asymptotic properties of $(\Fmult{e}{M}{N}/p^{e\dim(R)})_e$, but also provide a more algebraic characterization of those $N$ leading to a positive limit in cases including when $\dim(R) \leq 2$ or $R$ admits a strongly F-regular module-finite extension.

\begin{definition}
    Let $M, N$ be finitely generated $R$-modules where $N$ is indecomposable.   Following \cite{Yao05}, we say that $N$ is an \emph{F-contributor}\footnote{Differences between Yao's terminology and ours are addressed in \Cref{subsubsec: F contributors}. The indecomposability assumption on $N$ is made in the introduction for simplicity.} to $M$ if 
    \[ \limsup_{n \to \infty} p^{- n\dim (M)} \Fmult{n}{M}{N} > 0.\]
\end{definition}

The following are our main results regarding the convergence of 
$(p^{-e\dim(M)}\Fmult{e}{M}{N})_e$. 

\begin{theoremx}\label{intro first theorem}
    Let $M, N$ be finitely generated $R$-modules and $N$ be indecomposable. Then
    \begin{enumerate}
        \item The sequence $(p^{-e\dim(M)}\Fmult{e}{M}{N})_e$
        is \ces summable, i.e., the sequence
        \[
        n \mapsto \frac{1}{n}\sum_{e=1}^{n} p^{-e\dim(M)} \Fmult{e}{M}{N}
        \] converges. This follows from \Cref{th: existence of signatures in general} applied to $\p \in \Minh(M)$. 

        \item If there is an $e_0 \geq 1$ such that $N$ is a summand of both $F^{e_0}_*N$ and $F^{e_0+1}_*N$, then the sequence $e \mapsto p^{-e\dim(M)} \Fmult{e}{M}{N}$ converges; see \Cref{cor: aperiodicity implies existence of limit}.

        \item (Inherence) If $N$ is an F-contributor to $M$, then $N$ is an F-contributor to $N$. This follows from \Cref{th: existence of signatures in general}. 
\end{enumerate}
Moreover, we prove that the sequence $(p^{-e\dim(R)}\Fmult{e}{M}{N})_e$ converges in the following cases:
        \begin{enumerate}[label=(\alph*)]
            \item \label{intro part: module-finite SFR extension} $R$ is a module-finite subring of a strongly F-regular local ring; see \Cref{cor: F contributors are SFR when there is a module-finite SFR extension}.
            \item \label{intro part: dimension at most two} $\dim (R) \leq 2$; see \Cref{cor: two dimensional F-contributors}.
            
            \item $M$ has finitely many F-contributors (e.g., $M$ satisfies the FFRT property). This follows from \Cref{thm: F-signature for eFFRT}.
        \end{enumerate}    
\end{theoremx}

Theorem~\ref{intro first theorem} generalizes all previous results \cite{Yao05, HashimotoSymonds, HashimotoKobayashi, HashimotoNakajima, Tucker, SmithVandenBergh} on the convergence of $(p^{-e\dim(M)}\Fmult{e}{M}{N})_e$ in specific cases. 
In particular, it provides a new proof of the existence of the F-signature which uses only the F-contributor and F-split properties of strongly F-regular rings and avoids Tucker's use of the Hilbert--Kunz theory in \cite{Tucker}.

In our dynamical approach, roughly speaking, the iterates $F^e_*$ of the endofunctor $F_*$ induce a dynamical system on the category of finitely generated $R$-modules, where $e$ keeps track of the `time'. Our main idea for systematically harnessing tools from dynamics is the use of a possibly infinite matrix $\Mat(M)$, called the F-decomposition matrix of $M$, that encodes the indecomposable summands of $F^e_*M$ and the numbers $\Fmult{e}{M}{N}$ as $e$ varies. This idea originates in the case where $\Mat(M)$ is finite: independently in \cite{SeibertFFRT} for Cohen--Macaulay modules over rings of finite Cohen--Macaulay type and in \cite{SmithVandenBergh}, in the case where $R$ is additionally strongly F-regular and has the FFRT property.
Let $\find(M)$ be the union of the isomorphism classes of indecomposable summands of $F^e_*M$ as $e$ varies over $\mathbb N_{>0}$.
For simplicity, we fix an order on the countable set $\find(M)$ and write $\find(M)= \{N_i\}_{i \in \mathbb N_{>0}}$. The $(i,j)$-th entry of $\Mat(M)$ is $\Fmult{1}{N_i}{N_j}$. 
Since the $(i,j)$-th entry of $\Mat(M)^e$ is $\Fmult{e}{N_i}{N_j}$, we may analyze the sequence $(\Fmult{e}{M}{N})_e$
by studying the powers $\Mat(M)^e$.
Ostensibly, handling all these $\Fmult{e}{N_i}{N_j}$, where $i, j \in \mathbb N_{>0}$, might appear more difficult than handling a single sequence $(\Fmult{e}{M}{N})_e$. But it is this additional information that, on the one hand, makes results from matrices with nonnegative entries applicable for analyzing summands of $F^e_*M$, and, on the other hand, allows one to exchange information on $(\Fmult{e}{M}{N_i})_e$ across different choices of $N_i \in \find(M)$.

In more detail, for every associated prime ideal $\p$ of $M$, we associate a Markov chain whose state space --- denoted by $\find(\p,M)$ --- consists of all modules $N \in \find(M)$ that have $\p$ as an associated prime; see \Cref{re: Markov chain associated to an associated prime}. The transition probabilities of this Markov chain are encoded by the $\find(\p,M) \times \find(\p,M)$ block --- denoted by $\Mat(\p,M)$ --- of $\Mat(M)$, after a suitable normalization.
The Markov chain framework leads to notions such as connected classes, which are subsets of $\find(\p,M)$, the convergence radius of $\Mat(\p,M)$ and of the connected classes, and the period of a module. 
Recognizing the relevance of these notions to understanding the sequences $(\Fmult{e}{M}{N})_e$ is a key component of the results in the following sections. 

The period of an indecomposable module is a positive integer or infinity. Its relevance is illustrated by its role in the following generalization of \Cref{intro first theorem}:

\begin{theoremx}\label{intro: periodic signature} \textup{(cf. \Cref{th: existence of signatures in general})}
  Suppose $N \in \find(M)$ has period $\omega< \infty$. Let $\p$ be a prime ideal of $R$ associated to $N$ and let $N' \in \find(\p,M)$ be in the connected class of $\Mat(\p,M)$ containing $N$; $N'$ could be isomorphic to $N$.
\begin{enumerate}
\item For any $t \in \mathbb Z$, the sequence $(\Fmult{t +e\omega}{M}{N'}/p^{(t+ e \omega)\dim(R/\p)})_e$ converges to some real number.

    In particular, the sequence $(\Fmult{e}{M}{N'}/p^{e\dim(R/\p)})_e$ has at most $\omega$-many limit points and it is \ces summable. When $\omega=1$, this sequence is convergent. 

    \item \label{part: intro positivity implies positivity of self signature} $\limsup_e\Fmult{e}{M}{N'}/p^{e\dim(R/\p)}$ is positive if and only if $\lim_e \frac{\Fmult{e\omega}{N}{N}}{p^{e\omega\dim(R/\p)}}$ is positive. This follows by combining \Cref{th: existence of signatures in general}(\ref{part: cesaro signature is scaled by entrance probability}) with \Cref{le: first entrance probability for any module}(\ref{part: class property of entrance proability module case}).
\end{enumerate}   
\end{theoremx}

The content and the implications of \Cref{intro: periodic signature} and \Cref{th: existence of signatures in general} are more general than, and in certain aspects different from, the existing results in the literature:
\begin{enumerate}
    \item[(i)]  \Cref{intro: periodic signature} yields a handle on the sequence $(\Fmult{e}{M}{N}/p^{e\dim(R/\p)})_e$ as compared to $(\Fmult{e}{M}{N}/p^{e\dim(M)})_e$, which is the only case considered in the literature. Whenever $\dim(M)> \dim(R/\p)$, consideration of the latter is not interesting since it is bound to converge to zero; see \Cref{eg: contributors at a non minimal prime} for an illustration.

    \item[(ii)]\label{part: intro emphasis on the contributor} For questions on the convergence and positivity of $\lim_e\Fmult{e}{M}{N}/p^{e\dim(R/\p)}$, \Cref{intro: periodic signature}  shifts the entire emphasis onto $N$ in contrast to the earlier emphasis on working with the parent module $M$.

    \item[(iii)] \label{part: intro moving withing a connected class} For two modules $N_1,N_2$ in the same connected class of $\Mat(\p,M)$, the  convergence of $(\Fmult{e}{M}{N_1}/p^{e\dim(R/\p)})_e$ and the positivity of the limsup imply the convergence of $(\Fmult{e}{M}{N_2}/p^{e\dim(R/\p)})_e$ and the positivity of the limsup, respectively; see \Cref{th: existence of signatures in general}(\ref{part: convergence depends only on class}).
\end{enumerate}

As an illustration of (ii), we point to: 

\begin{theoremx}\label{intro: FFRT contributor}\textup{(cf. \Cref{se: FFRT case})}
Let $N$ be an indecomposable $R$-module with the FFRT property, i.e., $\find(N)$ is a finite set. Then, for any finitely generated $R$-module $M$, the sequence $(\Fmult{e}{M}{N}/p^{e\dim(M)})_e$ is convergent.
\end{theoremx}

\Cref{intro: FFRT contributor} follows because,
by \Cref{intro: periodic signature}, the limit is zero unless $\dim M = \dim N$ and $N$ is an F-contributor to itself, in which case $N$ has period one by \Cref{prop: FFRT contributors have period one}. This generalizes Yao's work \cite{Yao05}, which only dealt with the case where $M$ has the FFRT property and $N \in \find(M)$.

Next, we address the question of positivity of $\limsup_e\Fmult{e}{M}{N}/p^{e\dim(M)}$. For simplicity, unless otherwise said, in the rest of the introduction, we restrict ourselves to the case where both $M,N$ are torsion-free modules. In particular,  $\dim(M)= \dim(N)= \dim(R)$. We refer the reader to the following sections for discussion outside of this simplified setting.

In \Cref{intro: periodic signature}, (\ref{part: intro positivity implies positivity of self signature}) states a necessary condition for the positivity of the above limsup, namely the positivity of  $\limsup_e \Fmult{e}{N}{N}/p^{e\dim(R)}$. 
We conjecture that this necessary condition is, in fact, sufficient and provide evidence towards the conjecture. We call an indecomposable finitely generated torsion-free $R$-module $N$ an F-contributor\footnote{The indecomposability and torsion-freeness of F-contributors are assumed in the introduction for simplicity.} if $\limsup_e\Fmult{e}{N}{N}/p^{e\dim(N)}>0$.

\begin{conjecturex}[Universality of F-contributors] \label{intro conj: universality} Let $N$ be an F-contributor. Then $N \in \find(M)$ for any nonzero finitely generated torsion-free $R$-module $M$. 
 
 See \Cref{conj: universality conjecture} for a stronger version of this conjecture.   
\end{conjecturex}

 Note that, once an F-contributor $N$ is in $\find(M)$, $\limsup_e \Fmult{e}{M}{N}/p^{e\dim(R)}$ is positive. \Cref{intro conj: universality} can also be formulated in a more dynamical language:

\medskip 

\textit{If the category of finitely generated $R$-modules contains an indecomposable F-contributor $N$, then for any torsion-free $R$-module $M \neq 0$, $\Mat(M)$ has convergence radius $1/p^{\dim(R)}$ and exactly one final connected class, namely $\find(N)$.}

In general, a matrix with at most countably many rows and columns and nonnegative entries can have multiple final classes. So, \Cref{intro conj: universality} and the evidence towards it highlight a way in which matrices $\Mat(M)$, owing to their additional module-theoretic origin, are special among all matrices with nonnegative entries.

Indeed, our resolution of \Cref{intro conj: universality} in cases including those $R$ which admit a module-finite strongly F-regular ring extension or those with $\dim(R) \leq 2$ relies on algebraic tools. In this direction, our main contribution is the introduction and development of the notion of strongly F-regular modules in \Cref{se: strongly F regular modules}. Strongly F-regular modules provide a module-theoretic generalization\footnote{Blickle's notion of \emph{F-regular} modules \cite{Blickle} is a different generalization unrelated to this work.} of strongly F-regular rings introduced and developed by Hochster--Huneke (\cite{Hochster_Huneke_tight_closure_and_strong_F_regularity}, \cite{HochsterHuneke}). 
We prove \textit{universality of indecomposable strongly F-regular modules:} an indecomposable torsion-free strongly F-regular module $N$ splits from $F^e_*M$ for every nonzero torsion-free $M$ and some $e$ depending on $M$; see \Cref{th: universality of SFR modules}.
Then we resolve \Cref{intro conj: universality} in the aforementioned cases by proving strong F-regularity of indecomposable F-contributors; see \Cref{cor: F contributors are SFR when there is a module-finite SFR extension}, \Cref{cor: two dimensional F-contributors}. In fact, we show that if $R$ admits a module-finite strongly F-regular ring extension, the notions of strongly F-regular modules and F-contributors coincide.
The equivalence between these two notions in the aforementioned cases can be thought of as an analogue of Aberbach--Leuschke's characterization of strongly F-regular rings as F-contributors (\cite{AberbachLeuschke}); 
however, our method is different. 

In all cases where we prove \Cref{intro conj: universality}, we additionally prove a stronger result, namely a \textit{\ces rank formula}. For $R$-modules $M, X$, where $X$ is indecomposable, denote
\[
\csig{M}{X}_{(0)} = \lim_{n \to \infty} \frac{1}{n}\sum_{e=1}^{n} p^{-e\dim R} \Fmult{e}{M}{X},
\]
and recall that $\csig{M}{X}_{(0)} = \lim_{e \to \infty} p^{-e\dim R} \Fmult{e}{M}{X}$ whenever the latter exists.

\begin{theoremx}\label{intro additivity theorem}  Let $M, X$ be finitely generated $R$-modules, not necessarily torsion-free. Suppose $X$ is indecomposable and let $\rank_R(\bullet)$ denote the generic rank.
Then the \textit{\ces rank formula},
\[ \rank_R(X)\csig{M}{X}_{(0)}= \rank_R(M) \csig{X}{X}_{(0)},\]
holds if one of the following conditions is satisfied:
\begin{enumerate}
    \item $R$ is strongly F-regular; see \Cref{th: rank formula over sfr rings}.
    \item $\dim (R) \leq 2$; see \Cref{th: rank formula in dimension two}.
    \item $M$ has the FFRT property and $X$ is an indecomposable direct summand of $F_*^e M$ for some $e$; see \Cref{cor: FFRT rank formulas}.
    \item $M$ is indecomposable, $X \in \find(M)$ and $M \in \find(X)$, i.e., $X$ and $M$ are in the same connected class of $\Mat(M)$; see \Cref{lem: rank formula within connected classes}.
\end{enumerate}
Note that for a fixed indecomposable $X$, the validity of the \ces rank formula for all $M$ is equivalent to the additivity of the map $M \mapsto \csig{M}{X}_{(0)}$ over all short exact sequences of finitely generated modules; see \Cref{prop: characterize rank formula} and the cautionary \Cref{re: rank formula does not imply additivity}.
\end{theoremx}

We conjecture that the \ces rank formula always holds; see \Cref{conj: rank formula}. Our rank formula should remind the reader of similar rank formulas for other asymptotic invariants of singularities such as the Hilbert--Samuel multiplicity, Hilbert--Kunz multiplicity (\cite{WatanabeYoshida}) or the F-signature (\cite{Tucker}). For each of these three invariants, the corresponding rank formula largely follows from subadditivity of the length functions over short exact sequences. In contrast,
for $\csig{M}{X}_{(0)}$, the question of the \ces rank formula could not be turned into such a subadditivity statement due to a missing analogue of Frobenius degeneracy ideals.
Instead, we utilize the associated Markov chain formalism detailed in \Cref{sse: rank formula}. In general, the \ces rank formula for $\csig{M}{\bullet}_{(0)}$ corresponds to properties of the Markov chain induced by $\Mat((0), M)$ and hence its validity displays another distinctive feature of this Markov chain due to its module-theoretic origin; see \Cref{re: probabilistic interpretation of rank formula and total rank formula}.

The \ces rank formula has notable consequences. For example, if for some indecomposable $X$, $\csig{M}{X}_{(0)}$ is positive and $\rank_R(X)\csig{M}{X}_{(0)}= \rank_R(M) \csig{X}{X}_{(0)}$, then a \textit{total rank formula for $M$} holds (\Cref{prop: total rank formula vs rank formula}), i.e.,
\[\rank_R(M)= \sum_{\csig{M}{X'}_{(0)}>0} \csig{M}{X'}_{(0)} \rank_R(X').\]

When $R$ is strongly F-regular, this implication applied to $M \cong R$ reads
\[1= \sum_{X' \in \find(R)} \rank_R(X') \csig{R}{X'}_{(0)}.\]
This provides a way of estimating the classical F-signature $\csig{R}{R}_{(0)}$ from those $X' \in \find(R)$, where $X' \ncong R$. For further consequences of the rank formula and the total rank formula, derived using our Markov chain formalism, we refer to \Cref{pr: comparison of notions of rank formulas and consequences}, \Cref{re: interpretation of total rank formula} and \Cref{prop: consequences of the total rank formula}. 

Next, we turn our attention to the geometric ergodicity of the Markov chain induced by $\Mat(M)$. In our context, geometric ergodicity is equivalent to finding a geometric convergence rate of $(\Fmult{e}{M}{N}/p^{e\dim(R)})_e$, i.e., finding $0 \leq \lambda<1$ and a constant $C_{M,N}$ such that
\[|\lim_n\frac{\Fmult{n}{M}{N}}{p^{n\dim(R)}}-\frac{\Fmult{e}{M}{N}}{p^{e\dim(R)}}| \leq  C_{M,N}\lambda^e; \ \text{assuming convergence for simplicity}.\]
In the F-signature theory (i.e., $N \cong R$), $\lambda= 1/p$ works for the relevant convergence rate problem. 
Moreover, results proving certain uniformity for the accompanying constant, i.e., the analogue of $C_{M,N}$, have been crucial for proving semicontinuity of the limiting invariant; see \cite{PolstraTucker, Smirnov_Semicontinuity, Polstra}.
We prove geometric ergodicity towards F-contributors when $\dim(R) \leq 2$, or when $R$ is strongly F-regular and $N \in \find(R)$, the latter using the formalism developed in \Cref{subsec: convergence rates}. 
When $M$ has the FFRT property, geometric ergodicity with $\lambda= 1/p$ follows from \cite{Yao05} and this restricts possible eigenvalues of $\Mat(M)$; see \Cref{cor: Yao}.

\subsection*{Organization of the material} After briefly reviewing infinite-dimensional versions of the Perron--Frobenius theory and Markov chains in \Cref{se: recap of Perron-Frobenius theory}, we introduce and develop a dynamical system associated to the direct sum decomposition of the Frobenius pushforward of a module in \Cref{se: matrix associated to Frobenius decomposition} in the context of the Perron--Frobenius theory. In the subsections of \Cref{se: matrix associated to Frobenius decomposition}, we develop our general approach towards understanding F-contributors (\Cref{subsubsec: F contributors}), the rank formula (\Cref{sse: rank formula}) and the total rank formula (\Cref{sse: Total rank formula}). \Cref{se: FFRT case}, \Cref{se: strongly F regular modules}, and \Cref{se: dimension two} contain our results in commutative algebra showcasing applications of our dynamical theory developed in \Cref{se: matrix associated to Frobenius decomposition} to the FFRT setting, to rings admitting a module-finite strongly F-regular extension, and to rings of dimension at most two. \Cref{se: strongly F regular modules} briefly discusses the notion of a strongly F-regular module and its properties, which are also used in \Cref{se: dimension two}. \Cref{subsec: convergence rates} develops a general approach to proving convergence rates of the sequences defining F-splitting signatures, which is used in \Cref{se: strongly F regular modules}. Some conjectures on F-decomposition matrices and F-contributors are collected in \Cref{se: questions and conjectures}.

\subsection*{Acknowledgements}
We thank Devlin Mallory, Kevin Tucker and Yuji Yoshino for helpful exchanges and BCAM and EPFL for their hospitality during the authors' visits. 
The second-named author was supported by the Spanish Ministry of Science, Innovation and Universities by grants PCI2024-155055-2 funded by MICIU/AEI/10.13039/501100011033 and UE, PID2024-156181NB-C31 funded by MICIU/AEI/10.13039/501100011033 and FEDER, UE, RYC2020-028976-I funded by MICIU/AEI/10.13039/501100011033 and EI ESF ``ESF Investing in your future''. 

\subsection*{AI disclaimer}
The original draft of this article was completed without any direct assistance from any LLM. We prompted Claude for proofreading and checking mathematical correctness at the final stage of the preparation. As a byproduct, Claude suggested that we state \Cref{cor: F-contributor matrix is irreducible} and \Cref{rem: eigenvectors for F-split} and helped us simplify several arguments in \Cref{se: dimension two}. We take full responsibility for the mathematical content and final form of the article.

\section{Perron--Frobenius theory}\label{se: recap of Perron-Frobenius theory}

We briefly review the Perron--Frobenius theory for infinite matrices at a level of generality suitable for this article. Our recollection slightly exceeds the strict requirement of this article to provide algebraically minded readers with a more comprehensive understanding and for future referencing. The formulation below largely follows Vere-Jones \cite{VereJones} and draws inspiration from earlier works on discrete-time Markov chains; see \cite{Chung_Kai_Lai_Markov_Chains_with_stationary_transition_probabilities}, \cite{KolmogorovPaper} and the references therein. Another reference is Chapter 6 of \cite{SenetaBook}. 

The terminology and results introduced and recalled in this section are heavily used throughout the article.

\subsection{Infinite matrices}\label{sse: terminologies and notation for infinite matrices}

We develop a language of infinite matrices that is suitable for our purpose.

\begin{definition}\label{de: infinite matrix}
Let $\cS_1, \cS_2$ be two sets and $A$ be a ring. An $\cS_1 \times \cS_2$ matrix with entries in $A$ is a set map $\cS_1 \times \cS_2 \rightarrow A$.

For a set $\cS$ and a ring $A$, $\Id_\cS\colon \cS \times \cS \rightarrow A$ denotes the \textit{identity matrix} defined by $\Id_\cS(s,s')= 1$ if $s=s'$ and $\Id_\cS(s,s')$ is zero otherwise.
\end{definition}

For an $\cS_1 \times \cS_2$ matrix $T$, the notation
$T= (t_{ss'})_{ss'}$ means that $T(s, s') = t_{ss'}$ for $s \in \cS_1$ and $s' \in \cS_2$. For $s_1 \in \cS_1$, by the $s_1$-th row of $T$, we address the $\{s_1\} \times \cS_2$-matrix $(s_1,s') \mapsto t_{s_1s'}$. A column of $T$ is defined similarly. 

For $\cS_1' \subseteq \cS_1$, $\cS_2' \subseteq \cS_2$, the $\cS_1' \times \cS_2'$ matrix corresponding to the restriction map $T_{\mid \cS_1' \times \cS_2'}\colon \cS_1' \times \cS_2' \rightarrow A$ is called the $\cS_1' \times \cS_2'$ block of $T$, and is denoted by $T[\cS_1' \times \cS_2']$, or when $\cS_1'= \cS_2'$ by $T[\cS_1']$.

\begin{definition}
 An $\cS_1 \times \cS_2$ matrix $T$ is said to have \emph{finite rows} if any row of $T$ has only finitely many nonzero entries, i.e., for every $s \in \cS_1$, $T(s,s')$ is nonzero only for finitely many $s' \in \cS_2$. 
\end{definition}

Most of the infinite matrices in this work have finite rows. This property allows us to define products and powers. 

\begin{definition}\label{matrix product}
Let $A$ be a topological ring, $\cS_1, \cS_2, \cS_3$ be countable sets. Let $M_1$ and $M_2$ be $\cS_1 \times \cS_2$ and $\cS_2 \times \cS_3$ matrices with entries in $A$, respectively.  Then the \textit{product matrix} --- denoted $M_1M_2$ --- is given by the map $\cS_1\times \cS_3 \rightarrow A$, defined by $(s_1,s_3) \mapsto \sum_{s_2 \in \cS_2} M_1(s_1, s_2)M_2(s_2, s_3)$, whenever the sum is convergent for all $(s_1, s_3)$.
In our applications, $A \subseteq \mathbb C$ has the subspace topology and $M_1$ has finite rows, so the sum is finite and  $M_1M_2$ is defined.

Whenever we mention the product $M_1M_2$, we assume that the product is defined. For an $\cS \times \cS$ matrix $M$ and $n \in \mathbb{N}$, we define $M^n$ inductively as $M^0= \Id_{\cS}$, $M^1= M$ and $M^n= MM^{n-1}$.
\end{definition}

\subsection{Notations for matrices}

\begin{notation}
    Let $f\colon \NN \to \RR_{\geq 0}, g \colon \NN \to \RR_{\geq 0}$ 
    be two sequences. 
    We say that $g = O(f)$ if there exists a positive constant 
    $C$ such that $|g(k)| \leq Cf(k)$ for all $k \in \NN$. 
    We say that $g = \Theta(f)$ if $g = O(f)$ and 
    $\limsup_{k \to \infty} g(k)/f(k) > 0$.
\end{notation}

\begin{notation}
Let $\cS$ be a countable set. 
Let $T= (t_{ss'})_{ss'}$ be an $\cS \times \cS$ matrix and 
$f\colon \NN \to \RR_{\geq 0}$ be a sequence. Then
\begin{enumerate}
    \item Both $T^{(n)}_{ss'}$ and $t^{(n)}_{ss'}$ denote $T^n(s,s')$.
    
    \item A left (resp. right) eigenvector of $T$ with eigenvalue $\lambda \in \mathbb C$ is a nonzero $\{1\} \times \cS$-row (resp. $\cS \times \{1\}$ column) matrix $x$ with complex entries\footnote{Here we do not assume that $x$ has finite rows.}, such that $xT= \lambda x$ (resp. $Tx= \lambda x$).
    \item $T$ is \emph{nonnegative} (resp. \emph{positive}) if all its entries are nonnegative (resp. positive). 
    \item We write $T^k = O(f(k))$ (resp., $T^k = \Theta(f(k))$) to indicate 
    that $t_{ij}^{(k)} = O(f(k))$ (resp., $t_{ij}^{(k)} = \Theta(f(k))$) for all $i, j \in \cS$. Note that the bound is not uniform, i.e., we only require $|t_{ij}^{(k)}| \leq C(i,j) f(k)$ for a constant $C(i,j)$ depending on $i, j$. 
    \item Similarly, the convergence of a sequence of matrices means the entry-wise convergence, unless otherwise mentioned.
\end{enumerate}
\end{notation}

\subsection{Irreducible matrices}\label{sse: generalities on infinite irreducible matrices}
For the rest of \Cref{se: recap of Perron-Frobenius theory}, $\cS$ stands for a countable set  and $T$ denotes an $\cS \times \cS$ matrix.

\begin{definition}\label{de: irreducible matrix}
An $\cS \times \cS$ nonnegative matrix $T$ is called \emph{irreducible} if for any $i,j \in \cS$, there exist positive integers $n_1, n_2$,
such that $T^{n_1} ({i, j}) \neq 0$ and $T^{n_2} (j, i)  \neq 0$.   
\end{definition}

\begin{definition}\label{de: period of an irreducible matrix}
   The \emph{period} of an $\cS \times \cS$ irreducible matrix $T$ is  $\gcd \{k>0 \mid T^{(k)}_{ss} > 0\}$. It is independent of $s \in \cS$ thanks to the irreducibility of $T$. 
We call $T$ \emph{aperiodic} if its period is one.
\end{definition}

If $\omega \geq 1$ is the period of $T$, then we can divide $\cS$ into $\omega$ many disjoint subsets $C_\alpha$, where $0 \leq \alpha \leq \omega-1$, called 
\emph{periodic subclasses} of $T$. The classes are characterized by the property that if $i \in C_\alpha$
and $j \in C_\beta$ then  $T_{ij}^{(n)} \neq 0$ implies that $n \equiv \beta -\alpha \pmod \omega$. Alternatively, for a fixed $i \in \cS$ and an integer $\beta$, define $C_{\beta}(i)$ to be the set of indices $j$, such that $T^{(n)}_{ij} \neq 0$ implies $n \equiv \beta \pmod \omega$. Then the classes $\{C_{\beta}(i)\}_{\beta}$ and $\{C_{\alpha}\}_{\alpha}$ for $0 \leq \alpha, \beta \leq \omega-1$ coincide up to a permutation.

\subsubsection{Finite irreducible matrices}
The results on finite matrices mentioned below are used to conclude further structural properties of F-decomposition matrices of modules with FFRT in \Cref{se: FFRT case}.

\begin{theorem}[{\cite[Theorem~1.5, Theorem~1.7, Corollary to Theorem~1.7]{SenetaBook}}]\label{thm: finite Perron-Frobenius}
Let $n$ be a positive integer and suppose that $T$ is an $n\times n$ irreducible nonnegative matrix with period $\omega$. Then there exists an eigenvalue $r$, such that 
    \begin{enumerate}
        \item $r \in {\mathbb R}_{> 0}$;
        \item $r$ is a simple root of the characteristic polynomial;
        \item $r$ admits unique (up to multiplication) positive left and right eigenvectors;
        \item $r \geq |\lambda|$ for any other eigenvalue $\lambda$.
    \end{enumerate}

    It follows that $r$ is the spectral radius $\rho(T) \coloneqq \max \{|\lambda|\ |\ \lambda \in \mathbb C \text{ is an eigenvalue of $T$}\}$.
    
    Moreover, the spectrum is invariant upon multiplication by $e^{2\pi i/\omega}$. In particular, $\rho(T) e^{2k\pi i/\omega}$ is a simple eigenvalue of $T$ for each $k$ and these are the only eigenvalues of modulus $\rho(T)$.
\end{theorem}

By \cite[Theorem~1.4]{SenetaBook} 
irreducible aperiodic finite matrices are exactly the finite \textit{primitive} matrices, i.e., those which have a power with positive entries.

The following result is a structure theorem for periodic matrices.

\begin{theorem}[{\cite[Theorem~1.3, Lemma~1.4]{SenetaBook}, \cite[9.2.4]{Handbook}}]\label{thm: periodic irreducible}
    Let $T$ be an irreducible $n \times n$ nonnegative matrix with period $\omega > 1$.
    Then
    \begin{enumerate}
        \item there is a partition of $\{1, \ldots, n\}$
        into disjoint subsets $C_0, \ldots, C_{\omega-1}$ such that for any index $i \in \{1, \ldots, n\}$ and $0 \leq a \leq \omega-1$   
        \[
        C_a(i) \coloneqq \{j \mid t^{(m\omega + a)}_{i,j} \neq 0, \forall m \gg 0\} \in \{C_0, \ldots, C_{\omega - 1}\}
        \]
        and $C_a(i) = C_a$ if and only if $i \in C_0$;
        \item $T^\omega[C_a]$ is primitive;
        \item for any $0 \leq c < \omega$ the limit
        $\lim_{k \to \infty} \frac{t^{(k\omega + c)}_{i,j}}{\rho(T)^{k\omega + c}}$ exists.
    \end{enumerate}
\end{theorem}

It follows from \Cref{thm: finite Perron-Frobenius} and \Cref{thm: periodic irreducible} that for an irreducible nonnegative finite matrix $T$,  $T^k/\rho(T)^k$
converges if and only if there is no eigenvalue $\lambda \neq \rho(T)$
with $|\lambda| = \rho(T)$. We extend the convergence result for $(T^k/\rho(T)^k)_k$ to any finite matrix, not necessarily with nonnegative entries or irreducible, via a more elementary argument. In addition, our approach also yields a convergence rate. 

The assumption $\rho(T)=1$ in the next result does not cause any loss in generality. For a nonnilpotent matrix $T$, one can divide by $\rho(T)$ to pass to this case. 

\begin{proposition}\label{prop: bound growth rate}
For a positive integer $n$, suppose that $T$ is an $n\times n$ complex matrix 
    such that $\rho(T) = 1$, $1$ is the unique eigenvalue of modulus $1$, and $1$ is a semisimple\footnote{That is, the algebraic and geometric multiplicities are equal.} eigenvalue. Set $r = \max \{|\lambda| \mid 1 \neq \lambda \in \Spec (T)\}$ and let 
    $s$ be the maximum of the algebraic multiplicities of eigenvalues $\lambda$ with $|\lambda| = r$. Then $T_\infty \coloneqq \lim_{k \to \infty} T^k$ exists and 
    \begin{enumerate}
        \item If $\Spec (T) = \{1\}$, then $T = T_\infty$.
        \item If $r = 0$, then $T^k = T_\infty$ for $k \geq n$.
        \item In the remaining cases, $T^k = T_\infty + O (k^{s-1} r^k)$.
    \end{enumerate}
\end{proposition}
\begin{proof}
    Let $J$ be the Jordan normal form of $T$ and let $P$ be such that $T = P^{-1} J P$. Since $T^k = P^{-1} J^k P$, it suffices to prove the assertion for individual Jordan blocks. A Jordan block $\Lambda$ of size $m$ can be written as  $\lambda I + N$, where $I$ is the identity matrix and $N$ is a nilpotent matrix. Then for any $\lambda$ and
    $k \geq m$ 
    \begin{equation}\label{eq: powers of Jordan blocks}
        (\lambda I + N)^k = \sum_{i = 0}^{m-1} \binom{k}{i} \lambda^{k - i} N^i = O(|\lambda|^k k^{m-1}).
    \end{equation}
If $\lambda=1$, by our assumption $N=0$; if $|\lambda|<1$, $(\Lambda^k)_k$ converges to zero. Thus convergence of $(T^k)_k$ follows. The remaining assertions on convergence rate also follow from (\ref{eq: powers of Jordan blocks}).
    
\end{proof}

The next result concludes structural properties of $T$ from the convergence and convergence rate of $(T^k)_k$, and thus provides a partial converse to the above result.

\begin{corollary}\label{cor: matrix growth rate}
    Let $T$ be an $n\times n$ complex matrix 
    such that $T_\infty \coloneqq \lim_{k \to \infty} T^k$ exists, and is nonzero. Suppose $| T^k - T_\infty|\leq D r^k$ for some constants $D$ and $r$. Then 
    \begin{enumerate}
        \item $\rho(T) = 1$; $1$ is the unique eigenvalue of modulus $1$ and it is semisimple;
        \item $r \geq \mu \coloneqq \max \{|\lambda| \mid 1 \neq \lambda \in \Spec (T)\}$;
        \item any eigenvalue of $T$ with modulus $r$ must be semisimple.
    \end{enumerate}

\end{corollary}

\begin{proof}
Since $(T^k)_k$ is bounded and nonzero, $\rho(T)= 1$. The Jordan blocks of $T$ satisfy the hypothesis made about $T$. Since $T^k$ is convergent, \eqref{eq: powers of Jordan blocks} implies that any eigenvalue of $T$ with modulus one is one and one must be a semisimple eigenvalue. As for (2), take $v$ to be an eigenvector for an eigenvalue $\lambda \neq 1$. Since $T^k v = \lambda^k v$ converges, we must have $|\lambda| < 1$, so $T_\infty v = 0$.
Thus $|\lambda|^k |v| \leq |T_\infty - T^k| |v| \leq D r^k |v|$, so $|\lambda| \leq r$ as needed. 
Semisimplicity in the last claim follows from \eqref{eq: powers of Jordan blocks}.
\end{proof}

\begin{corollary}[{\cite[Theorem~1.2]{SenetaBook}}]\label{cor: limit for primitive}
Let $T$ be a primitive nonnegative matrix and 
take $v, u$ to be left and right $\rho(T)$-eigenvectors with positive entries such that $||u|| = ||v|| = 1$.
Then 
\[
(\rho (T))^{-n}T^n = u^\tau v/(v \cdot u) + O(n^{s-1} r^n\rho(T)^{-n})
\]
where $r = \max \{|\lambda| \mid \rho(T) \neq \lambda \in \Spec T\}$
and $s$ is the maximum of the algebraic multiplicities 
of eigenvalues $\lambda$ with $|\lambda| = r$.    
\end{corollary}

\subsubsection{Infinite irreducible matrices}
We return to the more general situation of countable $\cS \times \cS$ matrices (i.e., $\cS$ is a countable set) $T$ with nonnegative entries.

\begin{definition}
    Let $T$ be an $\cS \times \cS$ nonnegative matrix. The \textit{convergence radius} of $T$, denoted by $R(T)$, is
    \[\inf \left\{ \frac{1}{\limsup_{n \to \infty}(T^{(n)}_{ss'})^{1/n}} \mathrel{\Big|} s,s' \in \cS \right\},\]
    or, in other words, the infimum of the radius of convergence of the power series $\sum_{n=0}^{\infty} T^{(n)}_{ss'}z^n$, as $s,s'$ varies over $\cS$.

    When $T$ is additionally irreducible of period $\omega$, then for any $s \in \cS$, $0 \leq \alpha \leq \omega-1$ and $s' \in C_{\alpha}(s)$, 
    \[
        R(T) = 1/\lim_{n\to \infty} (T^{(\alpha+ n\omega)}_{ss'})^{1/(\alpha+n\omega)}.
    \]
\end{definition}

The insight of \cite{VereJones} is that 
a variant of the Perron--Frobenius theory can be developed for infinite nonnegative irreducible matrices where $R(T)$ takes the role of the reciprocal of spectral radius, whose use is crucial in the finite dimensional Perron--Frobenius theory.

\begin{definition}\label{de: transient and recurrent matrices}
    Let $T= (t_{ij})_{ij}$ be a nonnegative $\cS \times \cS$ matrix and $r$ be a nonnegative real number. We say that 
    $T$ is $r$-transient if $\sum_{n=0}^{\infty} T^{(n)}_{ss}r^n$
    converges for all $s$ (e.g., whenever $r < R(T)$) and that $T$
    is $r$-recurrent if $\sum_{n=0}^{\infty} T^{(n)}_{ss}r^n$ diverges for some $s$ (then $r \geq R(T)$). When $T$ is $r$-recurrent, we isolate two further cases:  
    $T$ is $r$-null recurrent if $\limsup_{n \to \infty} t_{ij}^{(n)} r^n = 0$ for all $i, j \in \cS$; $T$ is called $r$-positive recurrent if $\limsup_{n \to \infty} t_{ij}^{(n)} r^n \neq 0$ for some $i, j \in \cS$, including the case where the limits do not exist.

    When $T$ is irreducible, it suffices to check $r$-transience, null or positive recurrence on some pair $i, j \in \cS$. In particular, if $\limsup_{n \to \infty} t_{ij}^{(n)} r^n$ is nonzero for one choice of $i,j \in \cS$, then it is nonzero for all such choices.
\end{definition}

\begin{remark}
If $T$ is irreducible, then $T$ is positive recurrent for any $r>R(T)$. 
For this reason, we usually suppress the notation and just say that $T$ is \emph{positive or null recurrent} to refer to $R(T)$-positive or null recurrence.     
\end{remark}

\begin{theorem}[{\cite[Theorem~D, Criterion~III (p.375), Lemma 4.1]{VereJones}}] \label{thm: vere-jones eigenvectors}
        Let $T$ be an irreducible matrix. 
        Then $T$ is positive recurrent if and only if 
        there exist unique (up to scaling) positive left and right eigenvectors $a, b$ with the eigenvalue $1/R(T)$
        such that $a \cdot b = \sum a_i b_i$ is finite. 

        For some $\lambda>0$, if there are nonnegative nonzero left and right $\lambda$-eigenvectors $a, b$ with $a \cdot b < \infty$ then $R(T) = 1/\lambda$, $T$ is positive recurrent, and $a, b$ are positive.
\end{theorem}

We record an infinite-dimensional analogue of parts of \Cref{thm: periodic irreducible} and \Cref{cor: limit for primitive}.

\begin{theorem}[{\cite[Theorem~D]{VereJones}, \cite{Moy}}]\label{thm: Vere-Jones limit exists}
    Let $T$ be an irreducible, positive recurrent $\cS\times \cS$-matrix of period $\omega$. Let $a = (a_i)$ and $b = (b_i)$ be left and right positive eigenvectors of $T$ with eigenvalue $1/R(T)$ respectively.
    Then for any $i, j \in \cS$, $0 \leq \alpha \leq \omega-1$,
    \[
    \lim_{n \to \infty} t_{ij}^{(\alpha + \omega n)}R(T)^{\alpha + \omega n} =
    \begin{cases}
    \omega a_j b_i/(a \cdot b) &\text{ if } j \in C_\alpha (i), \\
    0   &\text{ otherwise}.
    \end{cases}
    \]
    In particular, 
    \[
    \frac 1\omega \sum_{\alpha = 0}^{\omega-1} \lim_{n \to \infty} T^{\alpha + \omega n}R(T)^{\alpha + \omega n} = \frac{b a^\tr}{\sum_{k\in \cS} a_kb_k}.
    \]
\end{theorem}

The converse to the theorem holds in the aperiodic case. 

\begin{lemma}\label{lem: limit and periodicity}
    Suppose that $T$ is an irreducible nonnegative $\cS \times \cS$-matrix. Then for some (equivalently all) $i \in \cS$ there is $r \in \RR_{> 0}$ such that $t_{ii}^{(n)}r^n$ converges to a positive real number if and only if $T$ is positive recurrent and aperiodic.
\end{lemma}

\begin{proof} 
    For the forward direction, we note that 
    $1/R(T) \coloneqq \lim_{n \to \infty} (t_{ii}^{(n)})^{1/n} = 1/r$
    and that $T$ is then positive recurrent by the definition. Since $t_{ii}^{(n)}r^n \to c > 0$, $t_{ii}^{(n)} > 0$ for $n \gg 0$, so $T$ is aperiodic.
    For the backward direction, one takes $r= R(T)$.
\end{proof}

\subsection{Connected classes}\label{sse: connected classes}
For a nonnegative $\cS \times \cS$ matrix $T$, which is not necessarily irreducible, we recall the notion of a \emph{connected class} of $T$. These subsets of $\cS$ correspond to strongly connected components of a certain directed graph $G(T)$ associated to $T$: the set of vertices of $G(T)$ 
is $\cS$ and there is a directed edge from $s$ to $t$ if and only if $T_{s, t}$ is nonzero. In other words, replacing nonzero entries of $T$ with $1$ gives the adjacency matrix of $G(T)$. Two vertices of $G(T)$ are called strongly connected if there are directed paths of positive length in $G(T)$ from $s$ to $t$ and from $t$ to $s$.

In the sequel, we freely use the following basic observation:

\begin{lemma}\label{le: connected class vs powers of matrix}
 Two vertices $s, t$ are strongly connected in $G(T)$ if and only if there are positive integers $m,n$ such that both $T^n(s, t)$ and $T^m(t, s)$ are nonzero.
\end{lemma}

\begin{definition}[Connected classes]\label{def: connected class}
We say that $s \in \cS$ forms an \textit{exceptional connected class}
if $s$ is not strongly connected to any vertex, including $s$, in $\cS$.

A subset $\cC \subseteq \cS$ is a \emph{connected class} (of $T$) 
if either $\cC = \{s\}$ is an exceptional connected class, or, for any $s \in \cC$, the set of vertices strongly connected to $s$ in $G(T)$ is $\cC$.

Thus for any $s \in \cS$, the \emph{connected class} of $T$ containing $s$  is either the set of vertices strongly connected to $s$ in $G(T)$ or the singleton set $\{s\}$, when there are no vertices strongly connected to $s$.
\end{definition}

When the connected class is singleton, our terminology does not allow us to infer whether the class is exceptional. But this does not cause any problem in practice.

\begin{definition}\label{de: transience recurrence of a connected class}
Let $\cC \subseteq \cS$ be a connected class.
If $\cC$ is exceptional, then $T[\cC]=0$. In this case, the \emph{period} of $\cC$ is defined to be infinity. Otherwise, $T[\cC]$ is an irreducible matrix and  the \emph{period} of $\cC$ or an element of $\cC$ is defined to be the period of $T[\cC]$.  

Similarly, for some $r \geq 0$, a connected class $\cC$ of $T$ is said to be \textit{$r$-transient}, or \textit{$r$-positive (or null) recurrent} if the matrix $T[\cC]$ is $r$-transient or $r$-positive (or null) recurrent, respectively. 

The \textit{convergence radius} of a connected class $\cC$ of $T$ is by definition the convergence radius of the matrix $T[\cC]$.
\end{definition}

The collection of connected classes imposes a partition on $\cS$.
The structure of $G(T)$ induces a partial order on the set of connected classes.  

\begin{definition}
    We say that $\cC_2 \preceq \cC_1$ if 
    there exist $i \in \cC_1, j \in \cC_2$ and a directed path (potentially of length zero) from $i$ to $j$.
    We call a connected class $\cC \subseteq \cS$ \emph{final} if it is a minimal element in the partial order $\preceq$. 
\end{definition}

By definition, if $\cC_1, \cC_2$ are distinct connected classes, 
then $t_{ij}^{(n)} = 0$ for all $n$ either for all $i \in \cC_1, j \in \cC_2$ or for all $i \in \cC_2, j \in \cC_1$. Thus, if $\cC$ is final and $i \in \cC$ then $t^{(n)}_{ij} \neq 0$ for some $n$ if and only if $j \in \cC$.

\begin{example}\label{ex: matrix d^1_5}
We illustrate the connected classes and the induced partial order on these for the following $\cS \times \cS$ matrix $T$, where $\cS= \{1, 2, \ldots, 6\}$,
\[
    T = \begin{bmatrix}
        1 & 0 & 0 & 1 & 0 & 1\\
        0 & 0 & 0 & 0 & 1 & 1\\
        0 & 0 & 0 & 0 & 1 & 1\\
        0 & 1 & 1 & 0 & 1 & 0\\
        0 & 0 & 0 & 0 & 2 & 2\\
        0 & 0 & 0 & 0 & 2 & 2
    \end{bmatrix}.
\]

Its only final class is $\{5, 6\}$. The resulting graph of the connected classes, where there is a directed path from $\cC_1$ to $\cC_2$ if $\cC_2 \preceq \cC_1$, is

\begin{center}
\begin{tikzpicture}[node distance={20mm}, thick, main/.style = {draw, circle}] 
\node[main] (1) {$1$}; 
\node[main] (2) [right of=1] {$4$}; 
\node[main] (3) [below right of=2] {$2$}; 
\node[main] (4) [above right of=2] {$3$}; 
\node[main] (5) [below right of=4] {$5, 6$}; 
\draw[->] (1) -- (2);
\draw[->] (2) -- (3);
\draw[->] (2) -- (4);
\draw[->] (3) -- (5);
\draw[->] (4) -- (5);
\draw[->] (1) to [out=325, in=245, looseness=1.8] (5); 
\draw[->] (2) -- (5); 
\end{tikzpicture} 
\end{center}
\end{example}

Drawing motivation from the finite matrix case, we introduce the notion of a basic recurrent class of a countably sized nonnegative matrix.

\begin{definition}\label{de: basic recurrent class}
Let $T$ be an $\cS \times \cS$ matrix with convergence radius $R(T)< \infty$. A connected class $\cC$ is called \emph{basic} or \emph{basic recurrent} if $T[\cC]$ is $R(T)$-recurrent. A basic class
$\cC$ is called \emph{basic positive (or null) recurrent} if $T[\cC]$ is positive (or null) recurrent. 
\end{definition}

The following result is a generalization of its known finite matrix version {\cite[Chapter 2, Theorem~3.10]{BermanPlemmons}}. This result is used repeatedly in the sequel. In practice, the hypothesis guaranteeing the existence of a positive eigenvector of the matrix involved is secured using inputs from algebra; see \Cref{pr: positive eigenvector of the block corresponding to a prime}. 

\begin{theorem}\label{thm: infinite basic is final}
Let $\cS$ be a countable set and $T$ be a nonnegative $\cS \times \cS$-matrix.
If there is a positive eigenvector $x$ of $T$ corresponding to the eigenvalue 
$\rho = 1/R(T)$, then any basic recurrent class $\cC$ is a final class. 
\end{theorem}
\begin{proof}
Since the entries of each row of $T$ are nonnegative, the restriction $x[\cC]$ of $x$ to $\cC$ is a $\rho^{-1}$-subinvariant vector of the minor $T[\cC]$ in the sense of \cite{VereJones}, 
\emph{i.e.}, $\rho^{-1} T[\cC] x[\cC] \leq x[\cC]$. 
If $\cC$ is basic recurrent, then 
$R(T[\cC]) = 1/\rho$ and $T[\cC]$ is recurrent, so
by \cite[Theorem~4.1, p.372]{VereJones} 
$x[\cC]$ must be a $\rho$-eigenvector. 

It remains to note that $\cC$ is final if and only if the equality holds, i.e., $x[\cC]$ is an eigenvector of $T[\cC]$ with eigenvalue $\rho$. For that, take $i \in \cC$ and write the $i$th row of $Tx$ as 
\[
(Tx)_i = (T[\cC]x[\cC])_i + \sum_{j \notin \cC} t_{ij}x_j
\]
and note that $\cC$ is final if and only if $t_{ij} = 0$
for all $i \in \cC, j \notin \cC$. 
\end{proof}

Existence of a positive eigenvector bounds the convergence radius of the matrix:

\begin{lemma}\label{lem: convergence radius from positive eigenvector}
Suppose that $T$ is a nonnegative matrix that has a positive right eigenvector $v= (v_s)_{s \in \cS}$ with 
an eigenvalue $\lambda > 0$. Then for any $i,j \in \cS$, $t_{i,j}^{(n)}= O(\lambda^n)$. Moreover, $R(T) \geq 1/\lambda$.
\end{lemma}
\begin{proof}
We have $\lambda^n v_i = (T^n v)_i \geq t_{i,j}^{(n)} v_j$. Positivity implies $t_{i,j}^{(n)}= O(\lambda^n)$, so $R(T) \geq 1/\lambda$.
\end{proof}

\subsection{Limit theorems for reducible matrices}\label{ssse: limit theorems for reducible matrices}
In this subsection, we assume that an $\cS \times \cS$ matrix $T= (t_{ss'})_{ss'}$ is nonnegative and has a positive right eigenvector $x= (x_s)_{s \in \cS}$ with eigenvalue $1/R$, for some $R > 0$.  We warn that $R$ is not necessarily the convergence radius of $T$. Under this assumption, we study the asymptotic behavior of the sequences $(R^nt^{(n)}_{ss'})_n$. The main advantage of the following approach through stochastic matrices lies in its capacity of handling matrices which are not necessarily irreducible, unlike some results in \Cref{sse: generalities on infinite irreducible matrices}.

The normalization
\[
p_{ij} \coloneqq R\frac{x_j}{x_i} t_{ij} 
\]
yields a row-stochastic matrix $P \coloneqq (p_{ij})_{ij}$, i.e., $\sum_{j \in \cS} p_{ij} = 1$ for all $i$. One can associate a Markov chain to the stochastic matrix $P$ \cite{SenetaBook} by thinking of the elements of $\cS$ as the \emph{states} and the matrix $P$ as the matrix encoding \textit{transition probabilities}: the entry $p^{(n)}_{i, j} = R^n t_{i,j}^{(n)} \frac{x_j}{x_i}$ is the probability of obtaining the state $j$ from the state $i$ at the $n$th step. This is a valid probabilistic interpretation since $\sum_jp^{(n)}_{i, j}=1$ for all $n$.

We also think of $(\cS, T)$ as a dynamical system, where $\cS$ consists of the states of the dynamical system, and $P$ dictates the evolution of the states with time: i.e., $P^n$ contains the information of the distribution of the states at time $n$.

Note that $P= R D^{-1}TD$, where $D$ is the diagonal matrix with $j$th diagonal entry $x_j$.
Hence, we may apply the theory of Markov chains (with a discrete parameter and countably many states) to study matrices such as $T$ and interpret its powers $T^e$, as $e$ (`the time') increases, in terms of `the evolution of this Markov chain as time grows'. For the ease of further referencing, below we quote the results about Markov chains after adapting directly in this matrix setting where $R$ is not necessarily one. Our presentation follows Chung's book \cite[Part I, Sections 6--7]{Chung_Kai_Lai_Markov_Chains_with_stationary_transition_probabilities}
which expands the foundational article of Kolmogorov \cite{KolmogorovPaper} (English translation is available in \cite{KolmogorovSelectedWorks})
where infinite Markov chains were first introduced.

\begin{definition}[First-entrance probabilities]\label{de: f}
For $i, j \in \cS$, define $(f^{(n)}_{ij})_{n \in \mathbb Z}$ to be the unique sequence of numbers satisfying: $f^{(n)}_{ij}=0$ for $n \leq 0$
and 
\[
t^{(n)}_{ij}= \sum_{e=1}^{n} f^{(e)}_{ij}t^{(n-e)}_{jj} \text{ for all } n \geq 1.
\]
We define 
$\mathfrak f ^*_{ij}(R) \coloneqq \sum_{n=1}^\infty R^n f^{(n)}_{ij}$
and for integers $\omega > 0, r$,
we further set
\[\mathfrak f^*_{ij}(r, \omega, R) \coloneqq \sum_{e=-\infty}^{\infty} R^{r+e\omega}f^{(r+ e\omega)}_{ij}.
\]
\end{definition}

\begin{remark}\label{re: first entrance probability}
When $T$ is stochastic, one can take $R=1$ and $x_j=1$ for all $j$. In this case, $T=P$ encodes the transition probabilities of a Markov chain. In this situation, $f^{(n)}_{ij}$ should be interpreted as the probability that the state $i$ arrives at $j$ at step $n$ for the first time, i.e., by avoiding $j$ except possibly at the beginning. Likewise, 
$\mathfrak f ^*_{ij}$ should be interpreted as the probability that the state $i$ arrives at $j$ at some step. Therefore, for $T$ as above, $x_j\ep{ij}(R)/x_i$ should be interpreted as the probability that the state $i$ arrives at state $j$ at some step in the Markov chain induced by $(\cS, T)$.
\end{remark}

The following result directly follows from \Cref{lem: convergence radius from positive eigenvector} and definitions.

\begin{lemma}\label{le: R recurrent implies basic recurrent}
Let $T= (t_{ij})_{ij}$ be as before.
\begin{enumerate}
    \item  Suppose that for some $i, j \in \cS$, $\sum_{n=0}^{\infty} t^{(n)}_{ij}R^n$ is divergent. Then $R(T)=R$.

    \item Every $R$-recurrent class of $T$ is a basic recurrent class of $T$.
\end{enumerate}

\end{lemma}

The following observations on $\mathfrak f^*_{ij}(R)$ are
largely immediate from the probabilistic interpretation. Their incarnations for F-decomposition matrices in \Cref{se: matrix associated to Frobenius decomposition} are crucial in our approach towards rank formula and total rank formula.

\begin{proposition}[{\cite[I.6]{Chung_Kai_Lai_Markov_Chains_with_stationary_transition_probabilities}}]\label{prop: properties of f}
Let $\cC$ be an $R$-recurrent\footnote{$T[\cC]$ then has convergence radius $R$, and hence is basic recurrent by \Cref{le: R recurrent implies basic recurrent}.} connected class of $T$ of period $\omega$ 
and $i, j\in \cS$. Then, for any $r \in \mathbb Z$, 
 \begin{enumerate}
     \item \label{part: periodic f is bounded above by subvariant vectors}  $\mathfrak f^*_{ij}(r, \omega, R) \leq \mathfrak f^*_{ij}(R) \leq x_i/x_j$. In particular, the two defining series converge.

    \item \label{part: recurrence in terms of first entrance} The connected class of $T$ containing $i \in \cS$  is $R$-recurrent if and only if $\mathfrak f^*_{ii}(R)=1$.
     
     \item \label{part: nonvanishing of eventual probability} $\mathfrak f^*_{ij}(R)$ is nonzero if and only if $t^{(n)}_{ij}$ is nonzero for some $n>0$.
     \item $\mathfrak f^*_{ij}(R) = \sum_{r=0}^{\omega-1}\mathfrak f^*_{ij}(r, \omega,R)$.
     
     \item \label{part: evetual entrance probability depends only on class} If $k, j \in \cC$, then $\frac{x_j}{x_i}\mathfrak f^*_{ij}(R) = \frac{x_k}{x_i}\mathfrak f^*_{ik}(R)$.  Thus $\mathfrak f^*_{kj}(R)= \frac{x_k}{x_j}\mathfrak f^*_{kk}(R)= \frac{x_k}{x_j}$.

     \item \label{part: eventual periodic entrance probability depends only on the periodic subclass} Suppose $k, j \in \cC$ and let $r' \in \mathbb Z$ be such that  $t^{(n)}_{jk}$ nonzero implies $n \equiv r'$ mod $\omega$. 
     Then $\frac{x_j}{x_i}\mathfrak f^*_{ij} (r, \omega, R) = \frac{x_k}{x_i}\mathfrak f^*_{ik}(r+r', \omega, R)$. 

     \item \label{part: sum of entrance probabilities to recurrent classes is at most one} \textup{(\cite[page 30, I.6.3]{Chung_Kai_Lai_Markov_Chains_with_stationary_transition_probabilities})} For every $R$-recurrent class $\cC'$, define $\tilde {\mathfrak f}^*_{i\mathcal C'}(R)= \frac{x_l}{x_i}\mathfrak f^*_{il}(R)$, for some $l \in \mathcal C'$. Then $\tilde {\mathfrak f}^*_{i\mathcal C'}(R)$ is independent of the choice of $l$ in $\mathcal C'$. Moreover,
     \[\sum_{\cC' \ \text{recurrent}}\tilde {\mathfrak f}^*_{i\mathcal C'}(R) \leq 1.\]
 \end{enumerate}    
\end{proposition}

The next few results record asymptotic behaviour of powers of possibly reducible matrices. These provide an abstract framework for proving existence of F-splitting signatures and to analyze those for rank formula or total rank formula in \Cref{se: matrix associated to Frobenius decomposition}.

\begin{theorem}[{\cite[Theorem I.6.4]{Chung_Kai_Lai_Markov_Chains_with_stationary_transition_probabilities}}]\label{th: results from Chung 's book on stochastic matrices}
Let $\cC \subseteq \cS$ be a connected class of $T= (t_{ss'})_{ss'}$, $j \in \cC$, and $i \in \cS$. 
If $\cC$ is $R$-transient or null $R$-recurrent, 
then $\lim_{n \to \infty} R^n t^{(n)}_{ij}=0$.
Likewise, the limit is zero when $i$ and $j$ belong to different $R$-recurrent classes of $T$.

If $\cC$ is positive $R$-recurrent 
with period $\omega$, then 
for any integer $r$,
    \[\lim_{n \to \infty} R^{r+ n\omega}t^{(r+ n\omega)}_{ij}= \mathfrak f^*_{ij}(r, \omega, R)\lim_{n \to \infty} R^{n \omega} t^{(n\omega)}_{jj} \leq \frac{x_i}{x_j}\lim_{n \to \infty}R^{n \omega}t^{(n\omega)}_{jj}.\]
\end{theorem}

Recall that the \ces limit of a sequence of complex numbers $(a_n)_n$ is the limit of the sequence
$N \mapsto \sum_{n=1}^{N}a_n/N$, whenever it exists. When $(a_n)_n$ is convergent, the \ces limit and the limit of $(a_n)_n$ coincide.

The next result directly follows from \Cref{th: results from Chung 's book on stochastic matrices}.

\begin{proposition}\label{pr: existence of cesaro limit for a stochastic matrix}
For any $i, j \in \cS$ the \ces limit of $(R^nt^{(n)}_{i j})_n$ exists.
If $j$ has period $\omega < \infty$,
then the \ces limit is
\[
\lim_{N \to \infty}
\frac{1}{N}\sum_{n=1}^{N} R^{n} t^{(n)}_{i j}
= 
\frac{1}{\omega} \sum_{r=0}^{\omega-1}\lim_{n \to \infty} R^{r+n\omega}t^{(r+ n\omega)}_{ij}= \frac {\mathfrak f^*_{ij}(R)}{\omega} \lim_{n \to \infty} R^{n\omega}t^{( n\omega)}_{jj}
\leq \frac{x_i}{x_j}
\frac{1}{\omega} 
\lim_{n \to \infty} R^{n\omega}t^{( n\omega)}_{jj}.
\]
If $\omega = \infty$, then the limit is zero. 
\end{proposition}

\begin{theorem}[{\cite[Theorem I.7.1]{Chung_Kai_Lai_Markov_Chains_with_stationary_transition_probabilities}}]\label{th: matrix version of total rank formula}
If the $\cS \times \cS$-matrix $T$ is irreducible of period $\omega$ and positive $R$-recurrent, then
\begin{enumerate}
    \item For any periodic subclass $\mathcal C_r$ of $\cS$,
    \(\displaystyle \sum_{s \in \mathcal C_r} \lim_{N \to \infty}
    \frac{1}{N}\sum_{n=1}^{N} R^{n} t^{(n)}_{s s} = \frac{1}{\omega}.
    \)

    \item \(\displaystyle \sum_{s \in \mathcal S} \lim_{N \to \infty} \frac{1}{N}\sum_{n=1}^{N} R^{n} t^{(n)}_{s s}
    = 1.\)
\end{enumerate}

\end{theorem}

\section{F-decomposition matrices}\label{se: matrix associated to Frobenius decomposition}

We set up some notation and conventions which are used throughout the article.

\noindent\textbf{Notation and Conventions:} Throughout the article, $p>0$ denotes a prime integer. In this article all rings are assumed to be commutative and unital and, unless otherwise said, noetherian and containing $\mathbb F_p$. 

Let $S$ be a ring. For any prime ideal $\p$ in the prime spectrum $\Spec(S)$ of $S$, $\kappa(\p)$ denotes the residue field $S_\p/\p S_\p$. We denote a local ring $R$ with the unique maximal ideal $\m$ by $(R, \m)$ or $(R,\m, k)$ where $k = \kappa(\m)$.

For an $S$-module $N$, the support of $N$, denoted by $\Supp (N)$, is the set of prime ideals $\p$ of $S$, 
such that $N_\p \ncong 0$ and $\Min(N)$ (resp., $\Minh(N)$) denotes the set of minimal primes (resp., minimal primes of maximal dimension) in the support of $N$. The set of associated primes of the $S$-module $N$ is denoted by $\Ass(N)$.

Let $N$ be a finitely generated $S$-module. When $N$ is Artinian, 
$\length_S(N)$ or $\length(N)$, when $S$ is clear, denotes the length of $N$. For $\p \in \Spec (S)$, we denote 
\[
\grk{\p}{N}\coloneqq \length_{S_{\mathfrak p}} (\Hh_{\p S_{\p}}^0 (N_{\p}))= \length_{S_{\mathfrak p}} (\{x \in N_\p \mid \p^nx=0, \ \text{for some}\ n \in \mathbb N\}).\]  
Note that, $\grk{\p}{N} \neq 0$ if and only if $\p \in \Ass (N)$. 
If $S$ is reduced and $\p \in \Min(S)$, then 
\[\grk{\p}{N} = \length_{S_\p} (N_\p) = \dim_{\kappa(\p)} N_{\p}.\] 
If $S$ is a domain, we use the more common notation $\rank_S(N)$ instead of $\grk{(0)}{N}$.

\begin{definition}
    A module $N$ is a direct summand of a module $M$, 
    if there exists an injection $\phi\colon N \to M$ and  
    a surjection $\pi \colon M \to N$ such that $\pi \circ \phi = \Id_N$.
    We call $\phi$ a split injection and $\pi$ a split surjection, without explicitly mentioning a corresponding split surjection or injection, respectively.

    We say that $M \neq 0$ is indecomposable if it has no proper nonzero direct summand. 
\end{definition}

Unless mentioned, we assume that any ring $S$ satisfies the Krull--Remak--Schmidt (KRS for short) condition, i.e., for any finitely generated $S$-module $M$ and two direct sum decompositions into indecomposable $S$-modules, $M \cong \oplus_{i=1}^n M_i$ and $M \cong \oplus_{j=1}^m M'_j$, $m=n$, and after a permutation of $\{1, 2, \ldots, n\}$, $M_i \cong M'_i$ for all $1 \leq i \leq n$.  Recall that examples of rings satisfying KRS include any Henselian local ring, in particular any complete local ring, or a principal ideal domain.

\subsection{Frobenius basics}
A characteristic $p$ ring $S$ is called \textit{F-finite} if the Frobenius endomorphism $F_S\colon S \rightarrow S$, where $F_S(s)= s^p$, is a finite ring map.
For an $S$-module $M$ and $e \in \mathbb N$, its \textit{$e$-th Frobenius pushforward} $F_{S*}^eM$, denotes the $S$-module whose underlying abelian group is $M$; and the $S$-module structure is obtained from the given one by restriction of scalars through the $e$-th iteration of the Frobenius endomorphism of $S$.
By convention, the zeroth iteration is the identity morphism. Most often, we write $F$ and $F^e_*$ because the choice of $S$ is unambiguous. 
For an $m \in M$, we write $F^e_*m$ to denote the element $m$, but now considered inside the module $F^e_*M$. So for $s \in S$, $sF^e_*m= F^e_*s^{p^e}m$.

In this article, we will work with F-finite rings unless otherwise mentioned. This ensures finiteness of the number of indecomposable direct summands of all $F_*^e M$, when $M$ is finitely generated.

The following notation appears throughout the article while relating the length of a module and its iterated Frobenius pushforwards (\Cref{lem: rank pushforward}).

\begin{definition}\label{de: notation for alpha and delta}
    For an F-finite local ring $(R, \m)$, $\alpha(R):= \log_p[\kappa(\m): \kappa(\m)^p]$. For a finitely generated $S$-module $M$, we define
    \(\delta_S(M)= \max\{\alpha(S_{\p})|\ \p \in \Min(M)\}.\)
    We write $\delta(M)$ for $\delta_S(M)$ when $S$ is clear and note that $\delta_S(S/\p) = \alpha(S_\p)$.
\end{definition}

\begin{theorem}[Kunz, {\cite[Proposition~2.3]{Kunz2}}]\label{thm: Kunz rank}
    Let $R$ be an F-finite ring and $\p \subsetneq \q$ be prime ideals. Then $\alpha (R_\p) = \alpha (R_\q) + \dim (R_\q/\p R_\q)$. 

    In particular, if $R$ is local, $M$ is a finitely generated $R$-module,  and $\p \in \Minh(M)$, then 
    \[
    \delta_R(M) = \alpha (R_\p) = \dim(M)+ \alpha(R).
    \]
\end{theorem}
    
\begin{lemma}\label{lem: rank pushforward}
    Let $(R, \m)$ be an F-finite local ring of characteristic $p$ and $M$ be a finite length $R$-module.
    Then $\length (F_* M) = p^{\alpha(R)} \length (M)$.
\end{lemma}
\begin{proof}
    Let 
    $0 = M_0 \subsetneq M_1 \subsetneq \cdots \subsetneq M_\ell = M$
    be the composition series of $M$.
    Then $\ell = \length_{R} (M)$ and $M_{i}/M_{i - 1} \cong \kappa(\m)$ for all $i = 1, \ldots, \ell$.
    Since $F_*$ is exact, it follows that $F_* (M_{i}/M_{i - 1}) \cong F_* \kappa(\m) \cong \kappa(\m)^{1/p}$ and the conclusion follows. 
\end{proof}

\subsubsection{F-singularities and F-invariants} We recall some further terminologies and results from the literature which appear in this article.

\begin{definition}\label{de: F split SFR rings}
    Let $R$ be an F-finite ring of positive characteristic $p$.
    We say that $R$ is 
    \begin{enumerate}
        \item \emph{F-split}  if $R$ is a direct summand of $F_* R$;
        \item \emph{strongly F-regular} if, for any $f \in R$, which is not in any minimal prime ideals of $R$, there exists an $e$, such that the following composition is a split injection
    \[R \xrightarrow[]{F^e} F^e_*R \xrightarrow[]{F^e_*f}F^e_*R.\]
    \end{enumerate}
\end{definition}

\begin{definition}[Yao \cite{Yao06}, Aberbach--Enescu \cite{AberbachEnescu}]
Let $(R, \m)$ be an F-finite local ring of characteristic $p$. 
For a positive integer $e$ the $e$-th \emph{F-splitting ideal} of $R$ is defined as 
\[
I_e \coloneqq \{x \mid \phi (F_*^e x) \in \m, \, \forall \phi \in \Hom_R (F_*^e R, R)\}. 
\]
\end{definition}

The F-splitting ideals are connected with F-singularities: $R$ is F-split if and only if $I_e \neq R$ for all (equivalently some) $e$ and $R$ is strongly F-regular if and only if $\cap_{e \geq 1} I_e = 0$ \cite[Theorem~4.8]{AberbachEnescu}.

\begin{definition}[Aberbach--Enescu \cite{AberbachEnescu}]
If $R$ is F-split, the ideal $\cP = \cap_{e \geq 1} I_e$ is a prime ideal and 
is called the \emph{splitting prime} of $R$.
The \emph{splitting dimension} of $R$ is $\sdim (R) = \dim R/\cP$. 
\end{definition}

\begin{theorem}[\cite{AberbachEnescu, Yao06, BlickleSchwedeTucker}]
    Let $(R, \m)$ be an F-finite local ring of characteristic $p$. Then
    \begin{enumerate}
        \item $a_e (R) \coloneqq \max \{N \geq 0 \mid F_*^e R = \oplus^N R \oplus M\} = p^{e \alpha(R)}\length (R/I_e)$;
        \item $p^{-e (\sdim (R) + \alpha(R))}a_e (R)$ converges and its limit is positive if $R$ is F-split.
    \end{enumerate}    
\end{theorem}

\begin{definition}[\cite{HunekeLeuschke}, \cite{AberbachEnescu}]
Let $(R, \m)$ be an F-split F-finite local ring of characteristic $p$. The limits
\[\Fratio (R) \coloneqq \lim_{e \to \infty} p^{-e (\sdim (R) + \alpha(R))}a_e (R)
\text{ and }
\fsig(R) \coloneqq \lim_{e \to \infty} p^{-e \delta_R(R)}a_e (R)
\]
are called the \emph{F-splitting ratio} of $R$ and 
the \emph{F-signature} of $R$ respectively.
\end{definition}

By \cite{AberbachLeuschke} (see also \cite{PolstraTucker})
$\fsig(R) > 0$ if and only if $R$ is strongly F-regular. 
By \cite{BlickleSchwedeTucker} we have that $\Fratio (R) = 1$ implies that $R/\cP$ is regular (and is equivalent to the stronger condition of \Cref{lem: ratio is 1}).

\subsection{Matrices arising from Frobenius pushforwards} \label{sse: F decomposition matrix}

From now on, let $R$ be an F-finite ring of characteristic $p > 0$ that satisfies the KRS property, unless otherwise specified. For a finitely generated $R$-module $M$, we introduce the notion of the 
\textit{F-decomposition matrix}, and prove properties of this matrix. This possibly infinite matrix encodes information on the direct sum decompositions of $F^e_*M$, for all $e>0$. The F-decomposition matrix is our main tool for harnessing dynamical results of \Cref{se: recap of Perron-Frobenius theory} for understanding indecomposable summands of $F^e_*M$, in the subsequent sections. In what follows, we adopt our convention on infinite matrices set in \Cref{sse: terminologies and notation for infinite matrices}.

\begin{definition}\label{de: collection of Frobenius summands}
   Let $M$ be an $R$-module. The set of \textit{indecomposable F-summands of $M$}, denoted $\find(M)$, is the union of isomorphism classes of nonzero indecomposable $R$-summands of $F^e_*M$ as $e$ varies over the set of positive integers.

   By abusing terminology, we talk of elements of $\find(M)$ as modules.
\end{definition}

\begin{definition}\label{de: the notation for multiplicites of direct summands}
In this definition, we do not assume the KRS property. For two $R$-modules $M,N$ and $e \in \mathbb N$, we denote 
 \[
 \Fmult{e}{M}{N} = \sup \{n \in \NN \mid N^{\oplus n} \text{ is a direct summand of } F_*^e M\}. 
 \]
If the set on the right is unbounded, the supremum is taken to be infinity. We will use $\FRmult{e}{M}{N}{R}$ to highlight the base ring. 
\end{definition}

\begin{definition}\label{def: Summand matrix}
The \emph{F-decomposition} matrix of a finitely generated $R$-module $M$ 
is an $\find(M) \times \find(M) $ matrix $\Mat (M)$
defined by 
\[
\Mat (M)_{N, N'} = \Fmult{1}{N}{N'} \text{ for any } N, N' \in \find(M).  
\]

Since $\find(M)$ is countable, we sometimes, especially when  it is finite, order $\find(M)$ and write $\find(M)= \{M_i\}_{i = 1}^n$, where $n$ could be infinite. In such a context, we think of $\Mat (M) = (t_{ij})_{ij}$ as an $n \times n$ matrix where $F_* M_i = \bigoplus_{j = 1}^n M_j^{\oplus t_{ij}}$ for all $i$. 
\end{definition}

\begin{example}
    Example~\ref{ex: matrix d^1_5} is the F-decomposition matrix of 
    the $D^1_5$ singularity in characteristic $2$ given by the equation $z^2 + x^2y + y^2z + xyz$. The computation was performed with \cite{MallorySayrafi}. 
\end{example}

\begin{remark}
    Since $R$ is F-finite, $\Mat(M)$ has finite rows. In particular, $\Mat (M)^n$ is defined. 
\end{remark}

The matrix $\Mat (M)^e$ encodes the direct sum decomposition of $F_*^e M$ into indecomposables: 

\begin{proposition}\label{prop: entries of powers of the F decomposition matrix}
Suppose $R$ is an F-finite ring satisfying the KRS property and $M$ is a finitely generated $R$-module. 
Let $N, N' \in  \find(M)$ and view $v= (\Fmult{1}{M}{N})_{N \in \find(M)}$ as a row vector. Then for any $e \geq 0$

\begin{enumerate}
    \item the $(N, N')$th entry of $\Mat(M)^e$ is $\Fmult{e}{N}{N'}$, and 

    \item \label{part: splitting numbers from matrix powers} $\Fmult{e+1}{M}{N}$ is the $N$th entry of the row vector $v\Mat(M)^e$.
\end{enumerate}    
\end{proposition}  

\begin{proof}
    We prove both the assertions by induction on $e$. For (1), recall that by our convention $\Mat(M)^0= \Id_{\find(M)\times \find(M)}$. The $e=1$ case follows from the definition of $\Mat(M)$. Since $F^{e+1}_*N \cong \oplus_{L \in \find(M)}F_*L^{\oplus \Fmult{e}{N}{L}}$,
    \[\Fmult{e+1}{N}{N'}= \sum_{L \in \find(M)}\Fmult{e}{N}{L}\Fmult{1}{L}{N'}.\]
    Since the right side of the last equation is also the $(N,N')$-th entry of $\Mat(M)^{e+1}$, the induction step follows.
    For (2), the $e=0$ case follows from the definition of $v$. The inductive step follows from a similar argument.
\end{proof}

In subsequent sections, in order to understand the indecomposable summands of $F^e_*M$, as $e$ grows, we use properties of $\Mat(M)^e$, as $e$ grows. To analyze the asymptotic properties of such matrix powers, we apply the Perron--Frobenius theory (\Cref{se: recap of Perron-Frobenius theory}) to $\Mat(M)$. The remainder of this section develops further properties of $\Mat(M)$ relevant from the point of view of Perron--Frobenius theory and freely uses the terminology set in \Cref{se: recap of Perron-Frobenius theory}. 
Our purpose requires working with specific blocks of $\Mat(M)$ associated to prime ideals of $R$, which we introduce now.

\begin{definition}
For a prime ideal $\mathfrak p$ of $R$, denote
\[\find(\mathfrak p, M)= \{N \in \find(M)|\ \mathfrak p \in \Ass(N)\}.\]
Denote the $\find(\p,M) \times \find(\p,M)$ minor $\Mat(M)[\find(\mathfrak p, M)]$ of $\Mat(M)$ by $\Mat(\mathfrak p, M)$. That is, $\Mat(\mathfrak p, M)_{NN'}= \Fmult{1}{N}{N'}$, for $N, N' \in \find(\mathfrak p, M)$.    
\end{definition}

For example, when $R$ is a domain and $M$ is torsion-free, $\find(\mathfrak p, M)$ is empty unless $\mathfrak p$ is zero, and $\find((0), M)= \find(M)$.
In general, $\find(\mathfrak p, M)$ is empty unless $\mathfrak p \in \Ass(M)$. In the sequel, separately working with $\Mat(\mathfrak p, M)$ is what allows us to develop the theory of F-splitting signatures when $R$ is not necessarily a domain or $M$ is not necessarily torsion-free.
A crucial input to handling such a general situation is the following observation on the existence of a suitable positive eigenvector of $\Mat(\mathfrak p, M)$, 
which allows us to apply the results of \Cref{ssse: limit theorems for reducible matrices}.

\begin{proposition}\label{pr: positive eigenvector of the block corresponding to a prime}
    The column vector $(\grk{\p}{N})_{N \in \find(\mathfrak p, M)}$ is a positive right eigenvector of $\Mat(\mathfrak p, M)$ with eigenvalue $p^{\alpha (R_\p)}$.
\end{proposition}
\begin{proof}
    Let $h$ denote the column vector in the statement. 
    Recall that for a finitely generated module $M$, $\grk{\p}{M} \neq 0$ if and only if $\p \in \Ass (M)$. 
    Thus, $h$ is positive by the definition of $\find(\mathfrak p, M)$. For any $N \in \find(\mathfrak p, M)$, the $N$th row of 
    $\Mat(\p, M)h$ is 
    \[
    \sum_{N' \in \find(\p,M)} \Fmult{1}{N}{N'}\grk{\p}{N'} 
    = \sum_{N'\in \find(M)} \Fmult{1}{N}{N'}\grk{\p}{N'} = 
    \grk{\p}{F_*N}.
    \]
    By \Cref{lem: rank pushforward}, $ \grk{\p}{F_*N} = p^{\alpha (R_\p)}\grk{\p}{N}$ and the conclusion follows. 
\end{proof}

\begin{remark}\label{re: Markov chain associated to an associated prime}
Since $\Mat(\p,M)$ is a nonnegative matrix with a positive eigenvector $(\grk{\p}{N})_{N \in \find(\p,M)}$ with eigenvalue $p^{\alpha(R_\p)}$, one can associate a Markov chain to $$(\Mat(\p,M), (\grk{\p}{N})_{N \in \find(\p,M)}, p^{\alpha(R_\p)})$$ following the recipe of \Cref{ssse: limit theorems for reducible matrices}. 
\end{remark}

\subsubsection{Connected classes}

\begin{lemma}\label{lem: support of an irreducible class}
    Let $\cC \subseteq \find(M)$ be a connected class of $\Mat(M)$ in the sense of (\ref{sse: connected classes}) and let $N \in \cC$.
    Then 
    \begin{enumerate}
        \item a module $L \in \find(M)$, where $L \neq N$, is contained in $\cC$
        if and only if $L \in \find(N)$ and $N \in \find(L)$,
        \item  $\Mat(N)[\cC] = \Mat (M)[\cC]$, and if $\cC$ is not exceptional, $\cC \subseteq \find (N)$.
        \item $\Supp (N) = \Supp (L)$ for any $L \in \cC$, 
        \item $\Ass (N) = \Ass (L)$  for any $L \in \cC$,
        \item $\Min (N) = \Min (L)$  for any $L \in \cC$,
        \item $\depth_I (N) = \depth_I (L)$ for any $L \in \cC$ and any ideal $I$ of $R$. 
    \end{enumerate}
    In particular, $\Supp (\cC)$, $\Min(\cC)$, $\Ass(\cC)$, and $\depth_I (\cC)$ are well-defined by any representative.
\end{lemma}
\begin{proof}
    We may assume that $\cC$ is not exceptional (in the sense of (\ref{sse: connected classes})) since otherwise the claims are trivial.
    The first claim is immediate from \Cref{le: connected class vs powers of matrix} and the second follows readily from the first claim and definitions.
    For the remaining claims, take $e$ such that $L$ is a direct summand of $F_*^e N$. 
    Then $\depth_I (L) \geq \depth_I (F_*^e N) = \depth_I (N)$
    and, for $P$ being one of $\{\Supp, \Ass\}$, we have 
    $P (L) \subseteq P (F_*^e N) = P(N)$. The conclusion follows by symmetry. Last, if $\Supp (L) = \Supp (N)$, then $\Min(L) = \Min(N)$. 
\end{proof}

If $N \in \find (\p, M)$, then, using Definition~\ref{def: connected class}, 
the connected class containing $N$ can be defined using either $\Mat(M)$ or $\Mat(\p, M)$. The following result shows that there is no ambiguity.

\begin{proposition}\label{pr: entries of powers of the matrix for a prime}
    Let $N, N' \in \find(\p, M)$. Then for any $n \in \NN$,
    \[
    \Mat(\p, M)^n_{N, N'} = \Mat(M)^n_{N, N'} = \Fmult{n}{N}{N'}.
    \]
    In particular, the connected class $\cC \subseteq \find(M)$ of $\Mat(M)$ containing $N$
    coincides with the connected class of $\Mat(\p,M)$ containing $N$ and $\cC$ is a subset of $\find(\p, M)$.
\end{proposition}

\begin{proof}
    We use induction on $n$. The base case of $n = 1$ holds by the definition. Assume that the assertion holds for some $n \geq 1$. Then for $N,N' \in \find(\p, M)$,
    \begin{equation*}
    \begin{split}
       &\Mat(M)^{n+1}_{NN'}\\
       &= \sum_{N_1 \in \find(M)} \Mat(M)^{n}_{NN_1}\Mat(M)_{N_1N'}\\
       &= \sum_{N_1 \in \find(\mathfrak p, M)} \Mat(M)^{n}_{NN_1}\Mat(M)_{N_1N'}  \,\  \bigl(\text{since}\ \Mat(M)_{N_1N'}=0 \ \text{if}\ N_1 \notin \find(\mathfrak p, M)\bigr)\\
       & = \sum_{N_1 \in \find(\mathfrak p, M)} \Mat(\p,M)^{n}_{NN_1}\Mat(\p,M)_{N_1N'}\,\,\,\, \bigl(\text{by induction hypothesis}\bigr)\\
       & = \Mat(\p,M)^{n+1}_{NN'}\ .
    \end{split}
    \end{equation*}
    The second claim follows from the first part and Lemma~\ref{lem: support of an irreducible class}. 
\end{proof}

We have the following bounds on the convergence radius. 

\begin{corollary}\label{cor: convergence radius is at least the power of dimension}
Let $M, N$ be finitely generated $R$-modules. Assume that $\p \in \Ass(N)$. Then
$\Fmult{e}{M}{N}= O(p^{e\alpha(R_\p)})$\textemdash this holds even when $R$ does not satisfy KRS. Moreover, the convergence radius of $\Mat(\p, M)$ is at least $p^{-\alpha(R_\p)}$ and the convergence radius of $\Mat(M)$ is at least $p^{-\delta (M)}$.  
\end{corollary}
\begin{proof}
Applying $\grk{\p}{\bullet}$ to the direct sum decomposition of $F^e_*M$, we conclude that
\[\grk{\p}{M}p^{e \alpha(R_\p)} \geq \grk{\p}{N}\Fmult{e}{M}{N}.\]
In particular for $N_1, N_2 \in \find(\p,M)$, $\Fmult{e}{N_1}{N_2}= O(p^{e\alpha(R_\p)})$. Therefore, the convergence radius of $\Mat(\p,M)$ is at least $p^{-\alpha(R_\p)}$. For the rest, it suffices to show that 
for any $M_1, M_2 \in \find(M)$, $ \Fmult{e}{M_1}{M_2}= O(p^{e\delta(M)})$. This follows from above since, for any $\p \in \Ass(M_2)$, $\alpha(R_\p) \leq \delta(M)$.
\end{proof}

We now study the properties of connected classes with the minimal 
radius of convergence allowed by Corollary~\ref{cor: convergence radius is at least the power of dimension}. The following lemma, essentially from \cite{Yao05}, is used to infer upon the depth of such classes.

\begin{lemma}\label{lemma: depth vs frequency of appearance}
    Let $M,N$ be finitely generated modules over an F-finite local ring $(R, \mathfrak m)$ (not necessarily KRS). 
    Then $\Fmult{e}{M}{N}= O(p^{e(\depth(N)+\alpha(R))}).$
\end{lemma}

\begin{proof}
Let $\mathrm K_\bullet$ be the Koszul complex of $R$ with respect to a  system of parameters $(\underline{x})$. The Hochster--Huneke--Roberts theorem (\cite{Roberts}, \cite[Theorem 6.2]{HochsterHunekePhantom}) implies that, for any $t\geq 0$,
\[\length_R (\Hh_{\dim(R)-t}(\underline x, F^e_*M )) = \length_R (\Hh_{\dim(R)-t}(F^e_*M \otimes \mathrm K_\bullet))= O(p^{e(t+\alpha(R))}). \]
Set $\depth(N) = d$.
Since $\Hh_{\dim(R)-d}(\underline x, N)$ is nonzero and 
\[\Fmult{e}{M}{N}\length_R (\Hh_{\dim(R)-d}(\underline x, N)) \leq \length_R (\Hh_{\dim(R)-d}(\underline x, F^e_*M )) = O(p^{e(d+\alpha(R))}),\]
the conclusion follows. 
\end{proof}

\begin{proposition}\label{prop: properties of maximal recurrent classes for a prime}
    Suppose that there is a  $p^{-\alpha(R_\p)}$-recurrent class $\cC \subseteq \find (\p, M)$ for $\Mat(\p, M)$.
    Then $R(\Mat(\p, M)) = p^{-\alpha(R_\p)}$ and $\cC$ is a final class of $\Mat(\p,M)$.  For any $N \in \cC$ and any prime ideal $\q \supseteq \p$, we have 
    $\depth_\q (N_\q) = \dim (R_\q/\p R_\q)$. 
    In particular, $\Ass (\cC) \cap V(\p) = \{\p\}.$
\end{proposition}
\begin{proof}
    The convergence radius $R(\Mat (\p, M)) = p^{-\alpha(R_\p)}$  by Corollary~\ref{cor: convergence radius is at least the power of dimension}.
    Thus, $\Mat (\p, M)$ has a positive eigenvector with eigenvalue  $1/R(\Mat (\p, M))$; see \Cref{pr: positive eigenvector of the block corresponding to a prime}. Hence, $\cC$ is final by Theorem~\ref{thm: infinite basic is final}.
    
    Let $\omega$ be the period of $\cC$. By the definition of the convergence radius, for any $\varepsilon > 0$, any $\q \in V(\p)$, and any $N \in \cC$, we have
    \[
    \limsup_{e \to \infty} (p^{\varepsilon -\alpha(R_\p)})^{e \omega} \Fmult{e\omega}{N_\q}{N_\q} \geq 
    \limsup_{e \to \infty} (p^{\varepsilon -\alpha(R_\p)})^{e \omega} \Fmult{e\omega}{N}{N} > 0.
    \]
    Therefore,
    $\alpha(R_\p) - \varepsilon \leq \depth_\q (N_\q) + \alpha(R_\q)$  by Lemma~\ref{lemma: depth vs frequency of appearance}. 
    Since $\alpha(R_\p) - \alpha(R_\q) = \dim (R_\q/\p R_\q)$ by Theorem~\ref{thm: Kunz rank}, 
    letting $\varepsilon \to 0$ shows that $\depth_\q (N_\q) \geq \dim (R_\q/\p R_\q)$. The opposite inequality follows since $\p$ is an associated prime of $N_\q$ \cite[Proposition~7, p.61]{Serre_Local_Algebra_ed1}. 
\end{proof}

\begin{corollary}\label{cor: properties of maximal recurrent classes in the whole matrix}
Suppose that $\Mat(M)$ has a $p^{-\delta(M)}$-recurrent class $\cC \subseteq \find (M)$.
Then $R(\Mat(M)) = p^{-\delta(M)}$, $\cC$ is a final class of $\Mat(M)$, and for any 
$\p \in \Ass (\cC)$, $\alpha(R_\p) = \delta(M)$.
Moreover, any $N \in \cC$ is Cohen--Macaulay. In addition, if $R$ is local, then for any $N \in \cC$, $\dim(N)= \dim(M)$.
\end{corollary}

\begin{proof}
    It is immediate from Corollary~\ref{cor: convergence radius is at least the power of dimension} that $R(\Mat(M)) = p^{-\delta(M)}$, and for any $ \p \in \Ass (\cC)$, $\alpha(R_\p) = \delta(M)$. If there were a connected class $\cC'$ such that $\cC' \preceq \cC$ (see \Cref{sse: connected classes}), then, as $\Ass (\cC') \subseteq \Ass (\cC)$ by Lemma~\ref{lem: support of an irreducible class}, $\cC \cup \cC' \subseteq \find (\p, M)$ for some $\p \in \Ass (\cC)$. Since $\cC$ is a final class of $\Mat (\p, M)$, $\cC'= \cC$; see Proposition~\ref{prop: properties of maximal recurrent classes for a prime}.

    For $N \in \cC$, we argue that $N_\m$ is Cohen--Macaulay for every maximal ideal $\m \in \Supp(N)$. Fix $\m \in \Supp(N)$. Let $\p$ be a prime ideal minimal in the set $\Spec(R_\m) \cap \Supp(N)$; then $\p$ is a minimal prime of $N$, and hence an associated prime of $N$. We have already shown $\alpha(R_\p)= \delta(M)$, so from \Cref{thm: Kunz rank} and the definition
    \begin{equation*}
    \begin{split}
      \alpha(R_\m)+ \dim(R_\m/\p R_\m) = \alpha(R_\p) = \delta(M) &\geq \sup \{\delta_{R_\q}(M_\q)\mid \q \subseteq \m, \q \in \Supp(M)\} 
      \\ &\geq 
      \delta_{R_\m} (M_\m) = \dim_{R_\m}(M_\m)+ \alpha(R_\m).
    \end{split}
    \end{equation*}
Since $\dim(M_\m) \geq \dim(R_\m/\p R_\m)$, the above inequality shows $\dim(M_\m) = \dim(R_\m/\p R_\m)$. Hence, \Cref{prop: properties of maximal recurrent classes for a prime}
    implies $ \depth_\m (N_\m) \geq \delta(M) - \alpha(R_\m) \geq \dim_{R_\m}(M_\m) \geq \dim_{R_\m}(N_\m)$. Hence, $N_\m$ is Cohen--Macaulay.

If $R$ is local, $\delta(M)= \dim(M)+ \alpha(R)$. Thus, we conclude from above that for any $\p$ as before, $\dim(R/\p)= \dim(M)= \dim(N)$.
\end{proof}

Following the notion of period of an irreducible matrix in \Cref{de: period of an irreducible matrix}, we introduce the notion of period of an indecomposable module.

\begin{definition}\label{de: period of indecomposable module}
    Let $N$ be an $R$-module. The period of $N$ is the positive integer which is the greatest common divisor of  $\{e \geq 1 |\ \Fmult{e}{N}{N} \neq 0\}$. If the latter set is empty,  we set the period to be infinity as a convention. 
    
    If $N$ is indecomposable and $N \in \find(N)$, the \emph{period} of $N$ is indeed the period of the connected class of $\Mat(N)$ containing $N$.
\end{definition}

The following result is straightforward:

\begin{lemma}\label{le: period one and splitting from large enough pushforward}
An $R$-module $N$ has period one if and only if $\Fmult{e}{N}{N}$ is nonzero for all large enough $e$.
\end{lemma}

\subsection{Existence and positivity of F-splitting signatures}\label{subsection: signatures and F-contributors}

In this section, we apply general properties of matrices with nonnegative entries and Markov chains to an F-decomposition matrix $\Mat(M)$, to analyze the asymptotic properties of its powers $(\Mat(M)^e)_e$, as $e$ approaches infinity. Consequently, we derive asymptotic properties of the sequences $(\Fmult{e}{M}{N})_e$ for any $N \in \find(M)$ (\Cref{th: existence of signatures in general}). Our work naturally leads to a theory of F-splitting signatures and its variants, which generalizes the classical theory of F-signatures --- the latter dealing with the $N=R$ case. 

\begin{definition}\label{de: generalized signature with a phase}
Let $M$ be a finitely generated $R$-module and $N \in \find(M)$ of period $\omega$. For any $\p \in \Ass(N)$ and integer $t$, we denote
\[\gsig{M}{N}_{t, \omega, \p} \coloneqq \lim_{e \to \infty}\frac{\Fmult{t+ e\omega}{M}{N}}{p^{(t+ e\omega)\alpha(R_\p)}},\]
and refer to $\gsig{M}{N}_{t, \omega, \p}$ as the \emph{F-splitting signature of $M$ towards $N$ with period $\omega$ and phase $t$ at $\p$}, provided the limit exists. Whenever we omit $\omega$ from the terminology and notation, it should be understood that $\omega$ is the period of $N$.

The \emph{\ces F-splitting signature of $M$ towards $N$ at $\p$} is defined to be the following limit, whenever it exists,
\[
\csig{M}{N}_{\p} = \lim_{n \to \infty} \frac{1}{n}\sum_{e=1}^{n}\Fmult{e}{M}{N}/p^{e\alpha(R_\p)}.
\]

We call the limit $\lim_{e \to \infty}\Fmult{e}{M}{N}/p^{e\alpha(R_\p)}$ (whenever it exists) the \emph{F-splitting signature of $M$ towards $N$ at $\p$} and denote it by $\gsig{M}{N}_\p$. Note that  $\gsig{M}{N}_\p = \csig{M}{N}_{\p}$ whenever the F-splitting signature exists.

Whenever $\lim_{e \to \infty}\Fmult{e}{M}{N}/p^{e\delta(M)}$ exists, we call the limit the \emph{F-splitting signature of $M$ towards $N$}, and denote it by $\gsig{M}{N}$. 
Likewise we use $\csig{M}{N}$ to denote the \ces limit of $\Fmult{e}{M}{N}/p^{e\delta(M)}$.

Whenever $\omega= \infty$, we define the quantities $p^{-(e\omega+t) \alpha(R_\p)}\Fmult{e\omega+t}{M}{N}$ to be zero. Consequently, $\gsig{M}{N}_{t, \omega, \p}= 0$.
\end{definition}

One of the main contents of \Cref{subsubsec: Existence of variants of signatures} is that the limits defining both $\gsig{M}{N}_{t, \omega, \p}$ and $\csig{M}{N}_{\p}$ exist.

\subsubsection{Existence}\label{subsubsec: Existence of variants of signatures}
As a first step, in \Cref{de: probability of entering a state for an indecomposable module} and \Cref{le: first entrance probability for any module}, we introduce suitable analogues of `first entrance probabilities' of  \Cref{de: f}. These play a crucial role in the existence of F-splitting signatures with phase (\Cref{th: existence of signatures in general}), the rank formula (\Cref{sse: rank formula}), and the total rank formula (\Cref{sse: Total rank formula}).

\begin{definition}\label{de: probability of entering a state for an indecomposable module}
Let $M$ and $N$ be two finitely generated $R$-modules. Assume that $N$ is indecomposable. Denote the period of $N$ by $\omega$. Define the sequence of numbers $(f^{(n)}_{M N})_n$ to be the unique sequence satisfying
\[
f^{(n)}_{M N}=0 \quad \text{if } n \leq 0; \qquad \Fmult{n}{M}{N}= \sum_{e=1}^{n}f^{(e)}_{M N}\Fmult{n-e}{N}{N} \quad \forall n \geq 1.
\]

For a prime ideal $\p$ of $R$ and an integer $t$, define
\[
\mathfrak f^*_{MN}(\p,t) \coloneqq \sum_{e=-\infty}^{\infty}\frac{f^{(t+ e\omega)}_{M N}}{p^{(t+ e\omega)\alpha(R_\p)}} \quad\text{ and }\quad \mathfrak f^*_{MN}(\p) \coloneqq \sum_{e=0}^{\infty} \frac{f^{(e)}_{M N}}{p^{e\alpha(R_\p)}}.
\]
When $\omega= \infty$, for all $t \in \mathbb Z$, $\mathfrak f^*_{MN}(\p,t)$ are zero by definition.
\end{definition}

\begin{remark}\label{re: probabilistic interpretation in the module case}
Suppose $\p \in \Ass(M)$. For any $N, N' \in \find(\p,M)$.
$\fep n_{N'N}$ are the `first-entrance probabilities' defined in \Cref{de: f}. After the suitable renormalization prescribed in \Cref{re: first entrance probability}, $\ep{N'N}(\p)$ is indeed the first entrance probability for the Markov chain associated to $\Mat(\p,M)$; see \Cref{re: Markov chain associated to an associated prime}. The notions of $\fep n_{MN}, \ep{MN}(\p)$ introduced above, provide the suitable analogue of `first entrance probabilities', even when $M$ is not indecomposable. These play a crucial role in the sequel, for example in \Cref{th: existence of signatures in general}, \Cref{pr: comparison of notions of rank formulas and consequences}.
    Moreover, the probabilistic interpretations of $\fep n_{MN}$, $\ep{MN}(\p)$ provide a probabilistic outlook towards the results where these appear. While we avoid probabilistic arguments, such interpretations have played crucial implicit roles in framing many results.
\end{remark}

\begin{lemma}\label{le: recursion on first entrance probability}
Let $M,N$ be as in Definition~\ref{de: probability of entering a state for an indecomposable module} and $\p$ be an associated prime of $N$. Then, for any integer $n \geq 2$,
\[\fep{n}_{MN}= \sum_{N' \in \find(M), N' \ncong N} \Fmult{1}{M}{N'} \fep{n-1}_{N'N}= \sum_{N' \in \find(\p,M), N' \ncong N} \Fmult{1}{M}{N'} \fep{n-1}_{N'N}.\]
\end{lemma}
\begin{proof}
The rightmost equality above follows since $\fep e_{N'N}$ is zero unless $\p \in \Ass(N')$. A straightforward calculation using \Cref{prop: entries of powers of the F decomposition matrix}(\ref{part: splitting numbers from matrix powers}) shows that the middle terms satisfy the recurrence relations defining $\fep{n}_{MN}$ in \Cref{de: probability of entering a state for an indecomposable module}.    
\end{proof}

The next lemma allows us to use \Cref{prop: properties of f} to analyze $\ep{MN}(\p,t)$, even when $M$ is not indecomposable.
\begin{lemma}\label{le: first entrance probability for any module}
Let $M$ be a finitely generated $R$-module and suppose that $N \in \find(\p, M)$ is such that the connected class of $\Mat(M)$ containing $N$ is $p^{-\alpha(R_\p)}$-recurrent. Denote the period of $N$ by $\omega$.
Then 
\begin{enumerate}
\item \label{part: vanishing of the entrance probability}
If for some integer $t$, $\mathfrak f^*_{M N}(\p,t) = 0$,
then $\Fmult{r}{M}{N} = 0$ for all $r \equiv t \pmod \omega$.
\item \label{part: summation for entrance probabilities}
For any positive integer $e_0$ and integer $t$,
\[\mathfrak f^*_{MN}(\p,t) = \sum_{N_1 \in \find(\p, M)} \frac{\Fmult{e_0}{M}{N_1}}{p^{e_0\alpha(R_\p)}}\mathfrak f^*_{N_1 N}(\p, t-e_0).
\]

\item \label{part: upper bound on first entrance module case} $\mathfrak f^*_{MN}(\p) \leq \grk{\p}{M}/\grk{\p}{N}$.
\item \label{part: rank formula for entrance probabilities} For any integer $r'$ and module $N' \in \mathcal C_{r'}(N)$ (see \Cref{sse: generalities on infinite irreducible matrices}), and $t \in \mathbb Z$,
    $\grk{\p}{N} \mathfrak f^*_{MN}(\p,t)= \grk{\p}{N'} \mathfrak f^*_{MN'}(\p,t+r').$ 
\item \label{part: class property of entrance proability module case} For any module $N'$ in the connected class of $\Mat(\p,M)$ containing $N$,  $\grk{\p}{N} \mathfrak f^*_{MN}(\p)= \grk{\p}{N'} \mathfrak f^*_{MN'}(\p).$    
\end{enumerate}
\end{lemma}

\begin{proof}
Note that our assumption implies $\omega< \infty$. Since $\ep{MN}(\p, t)= \ep{MN}(\p, t+ \omega)$, we assume throughout that $t \in (0, \omega]$.

(\ref{part: vanishing of the entrance probability}) We have $\fep {t + e\omega}_{MN} = 0$ for all $e \in \ZZ$. Since $\Fmult{e}{N}{N} = 0$ when $e \not \equiv 0 \pmod \omega$, we have 
\[
\Fmult{t + n\omega}{M}{N} = \sum_{e \leq n}f^{(t + (n - e)\omega)}_{M N}\Fmult{e\omega}{N}{N} = 0.
\]

For (\ref{part: summation for entrance probabilities}),  by 
induction we may  assume that $e_0=1$. If $t=1$, then by \Cref{le: recursion on first entrance probability}
\begin{equation*}
\begin{split}
    \ep{MN}(\p, 1)
    &= \frac{\fep 1_{MN}}{p^{\alpha(R_\p)}}+ \sum_{e=1}^{\infty}\frac{f^{(1+ e\omega)}_{M N}}{p^{(1+ e\omega)\alpha(R_\p)}}\\
    &= \frac{\Fmult{1}{M}{N}}{p^{\alpha(R_\p)}}+ \sum_{e=1}^{\infty} \sum_{N' \in \find(\p,M), N' \ncong N} \frac{\Fmult{1}{M}{N'}}{p^{\alpha(R_\p)}} \frac{\fep{ e\omega}_{N'N}}{p^{e\omega \alpha(R_\p)}}. 
\end{split}   
\end{equation*}
So, by \text{\Cref{prop: properties of f}(\ref{part: recurrence in terms of first entrance})},
\begin{equation*}
\begin{split}
    \ep{MN}(\p, 1)
    & = \frac{\Fmult{1}{M}{N}}{p^{\alpha(R_\p)}}\ep{NN}(\p,0) + \sum_{N' \in \find(\p,M), N' \ncong N} \frac{\Fmult{1}{M}{N'}}{p^{\alpha(R_\p)}} \ep{N'N}(\p,0)\\
    & = \sum_{N' \in \find(\p,M), } \frac{\Fmult{1}{M}{N'}}{p^{\alpha(R_\p)}} \ep{N'N}(\p,0).
\end{split}   
\end{equation*}
In the remaining $\omega \geq t>1$ case, we have
\begin{equation*}
    \begin{split}
       \ep{MN}(\p, t)&=  \sum_{e=0}^{\infty} \sum_{N' \in \find(\p,M), N' \ncong N} \frac{\Fmult{1}{M}{N'}}{p^{\alpha(R_\p)}} \frac{\fep{t-1+ e\omega}_{N'N}}{p^{(t-1+e\omega)\alpha(R_\p)}}\,\, (\text{by \Cref{le: recursion on first entrance probability}}) \\
       & = \sum_{e=0}^{\infty} \sum_{N' \in \find(\p,M)} \frac{\Fmult{1}{M}{N'}}{p^{\alpha(R_\p)}} \frac{\fep{t-1+ e\omega}_{N'N}}{p^{(t-1+e\omega)\alpha(R_\p)}} \,\, (\text{since}\ \fep{t-1+e\omega}_{NN}=0)\\
       &= \sum_{N' \in \find(\p,M)} \frac{\Fmult{1}{M}{N'}}{p^{\alpha(R_\p)}} \ep{N'N}(\p, t-1).
    \end{split}
\end{equation*}

(\ref{part: upper bound on first entrance module case}) By \Cref{prop: properties of f} (\ref{part: periodic f is bounded above by subvariant vectors}),
for any $N_1 \in \find (\p, M)$ we have $\mathfrak f^*_{N_1 N}(\p, t) \leq \mathfrak f^*_{N_1 N}(\p) \leq \grk{\p}{N_1}/\grk{\p}{N}$.
Plugging this in (\ref{part: summation for entrance probabilities}) and summing over $t$, we obtain that
\[
\mathfrak f^*_{M N}(\p) 
= \sum_{N_1 \in \find(\p, M)} \frac{\Fmult{e_0}{M}{N_1}}{p^{e_0\alpha(R_\p)}} \mathfrak f^*_{N_1 N}(\p) 
\leq \sum_{N_1 \in \find(\p, M)} \frac{\Fmult{e_0}{M}{N_1}}{p^{e_0\alpha(R_\p)}} \frac{\grk{\p}{N_1}}{\grk{\p}{N}} = \frac{\grk{\p}{M}}{\grk{\p}{N}}.
\]

(\ref{part: rank formula for entrance probabilities}) follows from (\ref{part: summation for entrance probabilities}) after noting that by \Cref{prop: properties of f}(\ref{part: eventual periodic entrance probability depends only on the periodic subclass}), for any $N_1 \in \find(\p,M)$, 
\[\frac{\grk{\p}{N'}}{\grk{\p}{N_1}}\mathfrak f^{*}_{N_1N'}(\p,t+ r')= \frac{\grk{\p}{N}}{\grk{\p}{N_1}}\mathfrak f^{*}_{N_1N}(\p,t).\]
(\ref{part: class property of entrance proability module case}) follows from \Cref{prop: properties of f}, (\ref{part: evetual entrance probability depends only on class}).
\end{proof}

We now establish existence of the limits defining F-splitting signatures with phase and its \ces version. The following observation allows us to focus only on modules with finite period later on.

\begin{lemma}\label{le: infinite period implies zero limit}
    Let $M$ be a finitely generated $R$-module with $\p \in \Ass(M)$ and $N \in \find(\p, M)$. If the period $\omega$ of $N$ is infinity, then for any integer $t$, $\gsig{M}{N}_{t, \omega, \p}= \gsig{N}{N}_{0, \omega, \p}= \csig{M}{N}_\p=0$.
\end{lemma}

\begin{proof}
The matrix $\Mat(\p,M)$ has a positive eigenvector for the eigenvalue $p^{\alpha(R_\p)}$; see \Cref{pr: positive eigenvector of the block corresponding to a prime}. \Cref{th: results from Chung 's book on stochastic matrices} applied to this setting implies $\csig{N'}{N}_\p=0$ for all $N' \in \find(\p,M)$. Since for $e \geq 1$, $\Fmult{e+1}{M}{N}= \sum_{N' \in \find(\p,M)}\Fmult{1}{M}{N'} \Fmult{e}{N'}{N}$, $\csig{M}{N}_\p$ is also zero. The rest follows from our convention set in \Cref{de: generalized signature with a phase}.
\end{proof}

\begin{theorem}\label{th: existence of signatures in general}
Let $R$ be an F-finite ring satisfying KRS and $M$ be a finitely generated $R$-module. Let $N \in \find(\p, M)$ have period $\omega$. Fix $t \in \ZZ$. 

Then the limits defining $\gsig{M}{N}_{t, \omega, \p}$ and  $\csig{M}{N}_{\p}$ exist.
In particular, the sequence $(\Fmult{e}{M}{N}/p^{e\alpha(R_\p)})_e$ has at most $\omega$ limit points.  

Moreover, the limits have the following properties:
\begin{enumerate}
    \item \label{part: zero self signature implies zero signature} if $\gsig{N}{N}_{0, \omega, \p} = 0$, then $\gsig{M}{N}_\p = 0$; 
    \item \label{part: the limit with period is a multiple of the self signature}  $\gsig{M}{N}_{t, \omega, \p} = \mathfrak f^*_{MN}(\p,t) \gsig{N}{N}_{0, \omega, \p}$;
    \item \label{part: cesaro signature is scaled by entrance probability} $\csig{M}{N}_\p = \frac{1}{\omega}\sum_{t=0}^{\omega-1}\gsig{M}{N}_{t, \omega, \p}= \frac{\mathfrak f^*_{MN}(\p)}{\omega} \gsig{N}{N}_{0, \omega, \p}= \ep{MN}(\p)\csig{N}{N}_\p$,
    where by our convention, any complex number divided by infinity is zero;
    \item \label{part: convergence depends only on class} Let $N' \in \find(\p,M)$ be in the connected class of $\Mat(\p,M)$ containing $N$. Then $(\Fmult{e}{M}{N}/p^{e\alpha(R_\p)})_e$ converges if and only if $(\Fmult{e}{M}{N'}/p^{e\alpha(R_\p)})_e$ converges, and the limsup of one of these sequences is positive if and only if the limsup of the other is.    
\end{enumerate}
\end{theorem}
\begin{proof}
We apply \Cref{prop: properties of f} and \Cref{th: results from Chung 's book on stochastic matrices} to the matrix $\Mat(\p,M)$ for the positive eigenvector $(\grk{\p}{N_1})_{N_1\in \find(\p,M)}$ with eigenvalue $p^{\alpha(R_\p)}$; see \Cref{pr: positive eigenvector of the block corresponding to a prime}. Thanks to \Cref{le: infinite period implies zero limit}, we assume that $\omega< \infty$ without loss of generality, from now on.

(\ref{part: cesaro signature is scaled by entrance probability}) follows from (\ref{part: zero self signature implies zero signature}), (\ref{part: the limit with period is a multiple of the self signature}) once we note that $\Fmult{e}{N}{N}$ is zero unless $\omega$ divides $e$. So, we just prove the existence of limits and (\ref{part: zero self signature implies zero signature}), (\ref{part: the limit with period is a multiple of the self signature}).

 First, we prove the desired statements when $M$ is replaced by any $N' \in \find(\p,M)$. 
When the connected class of $\Mat(\p,M)$ containing $N$ is not $p^{-\alpha(R_\p)}$-positive recurrent, all the limits defining periodic signatures and \ces signature are zero by \Cref{th: results from Chung 's book on stochastic matrices}. Hence, the desired statements follow.
So, while proving the statements for $N'$, assume the connected class of $\Mat(\p,M)$ containing $N$ is $p^{-\alpha(R_\p)}$-positive recurrent.
In this case, (\ref{part: zero self signature implies zero signature}), (\ref{part: the limit with period is a multiple of the self signature}) follow from  \Cref{th: results from Chung 's book on stochastic matrices}.

To get the statements for $M$, let $N'_1, \ldots, N'_r$ be the distinct modules in $\find(\p,M)$ which are summands of $F_*M$.  We use:

\begin{enumerate}
    \item[(i)] $\Fmult{e+1}{M}{N}= \sum_{i=1}^r \Fmult{1}{M}{N'_i} \Fmult{e}{N'_i}{N}$, for any $e \geq 1$.
\end{enumerate}

When the connected class of $\Mat(\p,M)$ containing $N$ is not $p^{-\alpha(R_\p)}$-positive recurrent, $(\Fmult{e}{N'_i}{N}/p^{e\alpha(R_\p)})_e$ converges to zero for any $i$. Thus by (i), all periodic signatures or the \ces signature of $M$ towards $N$ at $\p$ are zero.
Therefore, without loss of generality, we assume that the connected class of $\Mat(\p,M)$ containing $N$ is $p^{-\alpha(R_\p)}$-positive recurrent. 
It remains to check the existence of the limits defining $\gsig{M}{N}_{t,\omega,\p}$, $\csig{M}{N}_\p$ and (\ref{part: the limit with period is a multiple of the self signature}). All these follow, since for any $t \in \mathbb Z$, (i) and the case for $N' \in \find(\p,M)$ imply
\begin{equation*}
\begin{split}
 \lim_{e \to \infty} \frac{\Fmult{t+e \omega}{M}{N}}{p^{(t+e \omega)\alpha(R_\p)}}&= \sum_{i=1}^r \frac{\Fmult{1}{M}{N'_i}}{p^{\alpha(R_\p)}} \gsig{N'_i}{N}_{t-1, \omega,\p}  \\
 &= \sum_{i=1}^r \frac{\Fmult{1}{M}{N'_i}}{p^{\alpha(R_\p)}}\ep{N'_iN}(\p, t-1) \gsig{N}{N}_{0, \omega,\p}\\
 & = \ep{MN}(\p, t)\gsig{N}{N}_{0, \omega,\p}
\end{split} 
\end{equation*}

(\ref{part: convergence depends only on class}) If the connected class containing $N$ and $N'$ is not positive recurrent, then $\gsig{N}{N}_{0,\omega,\p} = \gsig{N'}{N'}_{0,\omega,\p} = 0$, and both sequences converge to zero by (\ref{part: zero self signature implies zero signature}). Otherwise $\omega < \infty$, so let $r$ be such that $\Fmult{e}{N}{N'} \neq 0$ implies $e \equiv r$ modulo $\omega$. By \Cref{le: first entrance probability for any module}, $\ep{MN}(\p,t)= \grk{\p}{N'}\ep{MN'}(\p, t+r)/\grk{\p}{N}$ for every $t$. By (\ref{part: the limit with period is a multiple of the self signature}), each of the two sequences converges if and only if the corresponding function $t \mapsto \ep{M\bullet}(\p,t)$ is constant, and these two conditions are equivalent. Finally, since $N, N' \in \find(\p,M)$, by (\ref{part: the limit with period is a multiple of the self signature}) the limsup of either sequence is positive if and only if $\gsig{N}{N}_{0,\omega,\p}>0$, i.e., if and only if the class is positive recurrent.

\end{proof}

Combining with \Cref{le: first entrance probability for any module} we get an upper bound on splitting signatures. 

\begin{corollary}[Rank inequality]\label{cor: rank inequalities}
Let $R$ be a ring of prime characteristic $p>0$, $\p \in \Spec(R)$, and $M$ be a finitely generated $R$-module. Let $N \in \find(\p, M)$ have period $\omega$. 
Then for any $t \in \ZZ$ $\gsig{M}{N}_{t, \omega, \p} \leq \frac{\grk{\p}{M}}{\grk{\p}{N}} \gsig{N}{N}_{0, \omega, \p}$, 
and
\[
\csig{M}{N}_\p \leq \frac{\grk{\p}{M}}{\omega \grk{\p}{N}} \gsig{N}{N}_{0, \omega, \p} = \frac{\grk{\p}{M}}{\grk{\p}{N}} \csig{N}{N}_{\p}.
\]
\end{corollary}
\begin{proof}
    Since $\gsig{N}{N}_{t, \omega, \p} = 0$
    whenever $t$ is not divisible by $\omega$, $\csig{N}{N}_\p= \gsig{N}{N}_{0, \omega, \p}/\omega$. Since $\ep{MN}(\p, t) \leq \ep{MN}(\p)$, the rest is immediate from \Cref{le: first entrance probability for any module}.
\end{proof}

\begin{corollary}\label{cor: aperiodicity implies existence of limit}
Let $X$ be an indecomposable finitely generated $R$-module of period one and $\p \in \Ass(X)$. Then for any finitely generated module $M$, the F-splitting signature of $M$ towards $X$ at $\p$ exists. In particular, the F-splitting signature of $M$ towards $X$ exists.
\end{corollary}

\subsubsection{F-contributors}\label{subsubsec: F contributors}
For a finitely generated $R$-module $M$, we study those $N \in \find(\p,M)$ such that $(\Fmult{e}{M}{N})_e$ has a maximum possible geometric growth rate, under the name of F-contributors and their variants. 
Our notion of F-contributors at a prime ideal, in retrospect, generalizes Yao's notion of F-contributors introduced in \cite{Yao05}. Our characterization of F-contributors in terms of basic positive recurrent classes of the F-decomposition matrix (\Cref{thm: F-contributor at a prime via matrices}) is a key new feature which allows us to characterize all torsion-free F-contributors over a strongly F-regular ring or their finite group quotient in \Cref{se: strongly F regular modules}.

\begin{definition}\label{de: F-contributor at a prime}
    Let $R$ be an F-finite ring, $\p$ be a prime ideal of $R$, and 
    $M$ be a finitely generated $R$-module. A module $X$ is called an \textit{F-contributor to $M$} at $\p$ if $\p \in \Ass(X)$ and $\limsup_{e \to \infty} \Fmult{e}{M}{X}/p^{e\alpha (R_\p)}$ is positive.
    An $R$-module $X$ is called an \textit{F-contributor to $M$} if $\limsup_{e \to \infty} \Fmult{e}{M}{X}/p^{e\delta (M)}$ is positive.

    If $M = X$, we will simply say that $X$ is an \textit{F-contributor} (resp., an \textit{F-contributor at $\p$}). 
\end{definition}

By \Cref{cor: convergence radius is at least the power of dimension}, for an $X \in \find(\p,M)$, $\Fmult{e}{M}{X}= O(p^{e\alpha(R_\p)})$. Thus an F-contributor $X$ to $M$ (at $\p$) splits from $F^e_*M$ at a maximum possible geometric rate as $e$ varies.

Our notion of an F-contributor to a module is akin to Yao's definition. In contrast, an F-contributor does not refer to any parent module. Moreover, the framework of F-contributor at a prime ideal includes modules $X$, where $\Fmult{e}{X}{X}$ has a geometric growth rate less than $p^{\delta(X)}$, unlike Yao's notion.


The following result equivalently characterizes F-contributors at a prime (and hence F-contributors) via F-decomposition matrices. This implies special properties of F-contributors at a prime thanks to \Cref{prop: properties of maximal recurrent classes for a prime}.

\begin{theorem}\label{thm: F-contributor at a prime via matrices}
    Let $X$ be an indecomposable finitely generated $R$-module, $M$ be a finitely generated $R$-module, and $\p \in \Spec(R)$. 
    Consider the following conditions:
    \begin{enumerate}
            \item[(1a)] $X$ is an F-contributor to $M$ at $\p$. 
            \item[(1b)] $X \in \find(\p, M)$, $R(\Mat(\p, M)) = p^{-\alpha(R_\p)}$, and the connected class of $\Mat(\p, M)$ containing $X$ is basic positive recurrent.
            \item[(2a)] \label{part: 2a} $X$ is an F-contributor at $\p$. 
            \item[(2b)]\label{part: 2b} $X \in \find (\p, X)$, $R(\Mat(\p, X)) = p^{-\alpha(R_\p)}$, and $\Mat(\p, X)$ is irreducible positive recurrent.
            \item[(3)] for any finitely generated $R$-module $N$ such that $X \in \find(N)$, $X$ is an F-contributor to $N$ at $\p$ 
            and the connected class of $\Mat (\p, N)$ containing $X$ 
            is $\find (\p, X)$. 
    \end{enumerate}
    Then (1a) $\Leftrightarrow$ (1b) $\Rightarrow$ 
    (2a) $\Leftrightarrow$ (2b) $\Rightarrow$ (3). 
\end{theorem}
\begin{proof}

    (1a) $\Rightarrow$ (2a). By definition, $\p \in \Ass(X)$. \Cref{th: existence of signatures in general}(\ref{part: zero self signature implies zero signature}) implies $X$ is an F-contributor at $\p$.

    (1a) $\Rightarrow$ (1b): 
    Both conditions imply that $\p \in \Ass (X)$. By \Cref{cor: convergence radius is at least the power of dimension}, $R(\Mat(\p,M))\geq p^{-\alpha(R_\p)}$. By (2a), the radius of convergence of $\sum_{n\geq 0}\Fmult{n}{X}{X}z^n$ is $p^{-\alpha(R_\p)}$. Hence $R(\Mat(\p,M))= p^{-\alpha(R_\p)}$. (2a) implies that the concerned connected class of $\Mat(\p,M)$ is $p^{-\alpha(R_\p)}$-positive recurrent, and hence basic positive recurrent.

    (1a) $\Leftarrow$ (1b): (1b) implies $\limsup_e \Fmult{e}{X}{X}/p^{e\alpha(R_\p)}$ is positive. Since $X \in \find(\p,M)$, $\p \in \Ass(X)$ and $\limsup_e \Fmult{e}{M}{X}/p^{e\alpha(R_\p)}$ is also positive. So (1a) follows.

    (2a) $\Leftrightarrow$ (2b): For the backward direction, note that, since $\Mat(\p,X)$ is irreducible and $X \in \find(\p,X)$, the connected class of $\Mat(\p,X)$ containing $X$ is $\find(\p,X)$, and the latter is positive recurrent by our very assumption.
    Thus (2a) follows by applying the equivalence (1a) $\Leftrightarrow$ (1b). 
    For the forward direction, thanks to (1a) $\Leftrightarrow$ (1b), it remains to check the irreducibility of $\Mat(\p,X)$. To this end, (2a) implies that the connected class of $\Mat(\p,X)$ containing $X$, denoted by $\cC$, is final by \Cref{prop: properties of maximal recurrent classes for a prime}.
    Since $X \in \cC$, $\find(\p,X) \subseteq \cC$. Thus $\find(\p,X) = \cC$. Hence $\Mat(\p,X)= \Mat(\p,X)[\cC]$ is irreducible; cf. 
    \Cref{de: transience recurrence of a connected class}.

    (2b) $\Rightarrow$ (3). By definition, $X \in \find(\p, N)$.
    Let $\cC \subseteq \find(\p, N)$ be the connected class of $\Mat(\p, N)$ containing $X$. By \Cref{lem: support of an irreducible class}, $\cC$ is also the connected class of $\Mat(\p,X)$ containing $X$. The irreducibility assumption in (2b) implies $\cC= \find(\p,X)$. 
    By (2b) $\limsup_e \Fmult{e}{X}{X}/p^{e\alpha(R_\p)}$ is positive. Since $X \in \find(N)$, $\limsup_e \Fmult{e}{N}{X}/p^{e\alpha(R_\p)}$ is also positive, and thus $X$ is an F-contributor to $N$ at $\p$.  
\end{proof}

Combining \Cref{thm: F-contributor at a prime via matrices} and \Cref{th: existence of signatures in general} we immediately get the following characterization of F-contributors to a module as summands which are themselves F-contributors. 

\begin{corollary}\label{cor: positive ces signature -> F-contributor}
    Let $R$ be a ring of prime characteristic $p>0$, $M$ be a finitely generated $R$-module, and $N \in \find(\p, M)$. 
    Then $\csig M N_\p > 0$ if and only if $N$
    is an F-contributor at $\p$. 
\end{corollary}

\begin{lemma}\label{lem: radius of convergence for contributors}
    If $X$ is an indecomposable F-contributor to $M$, then $R(\Mat(M)) = p^{-\delta (M)}$  and $X$ is an F-contributor to $M$ at a prime ideal $\p$ such that $\delta (M) = \alpha (R_\p)$. 
\end{lemma}
\begin{proof}
    Let $\p \in \Ass (X)$ be arbitrary.
    Since $X \in \find(M)$, $\Ass (X) \subseteq \Ass (M)$, 
    so $\alpha(R_\p) \leq \delta(M)$ by \Cref{thm: Kunz rank}.
    On the other hand, by \Cref{cor: convergence radius is at least the power of dimension} we have $\alpha(R_\p) \geq \delta(M)$, 
    so the equality holds. Now, by \Cref{cor: convergence radius is at least the power of dimension} and 
    by \Cref{thm: F-contributor at a prime via matrices}
    we have 
    \[
    p^{- \delta (M)} \leq R(\Mat(M))  \leq R(\Mat(\p, M))
    = p^{-\alpha(R_\p)} \leq p^{-\delta(M)}
    \]
    and the equality must hold throughout.
\end{proof}

\begin{corollary}\label{cor: F-contributor matrix is irreducible}
    If $X$ is an indecomposable F-contributor, then $R(\Mat(X)) = p^{-\delta (X)}$  and $\Mat(X)$ is irreducible positive recurrent.
\end{corollary}
\begin{proof}
    By the proof of \Cref{lem: radius of convergence for contributors}, $\delta(X) = \alpha(R_\p)$ for every $\p \in \Ass(X)$. Thus, for any $\p \in \Ass(X)$,  $X$ is an F-contributor at $\p$ and  
    \Cref{thm: F-contributor at a prime via matrices}, (2a) $\Rightarrow$ (2b), shows that $\Mat(\p, X)$ is irreducible. Take $N \in \find(X)$ and $\p \in \Ass(N)$. Then $N \in \find(\p,X)$ and thus $N$ is in the connected class of $\Mat(\p,X)$ containing $X$.
    Since the connected classes of $\Mat(\p,X)$ and $\Mat(X)$ containing $N$ are equal by \Cref{pr: entries of powers of the matrix for a prime}, $\Mat(X)$ has exactly one connected class, namely the one containing $X$. So irreducibility follows. Positive recurrence holds since $X$ is an F-contributor.
\end{proof}

We now derive special properties of F-contributors from Proposition~\ref{prop: properties of maximal recurrent classes for a prime} and 
Corollary~\ref{cor: properties of maximal recurrent classes in the whole matrix}. In particular, we have:

\begin{corollary}\label{cor: Cohen--Macaulayless of contributors}
    If $X$ is an indecomposable F-contributor to $M$ then $X$ is an F-contributor and is Cohen--Macaulay.
\end{corollary}
\begin{proof}
   By Lemma~\ref{lem: radius of convergence for contributors} and  Theorem~\ref{thm: F-contributor at a prime via matrices}, $X$ is an F-contributor at a prime $\p \in \Ass(X)$ such that $\delta(M)= \alpha(R_\p)$. Since $\delta(M) \geq \delta(X) \geq \alpha(R_\p)$, $\delta(X)= \alpha(R_\p)$, and we are done by Corollary~\ref{cor: properties of maximal recurrent classes in the whole matrix}.
\end{proof}

\begin{proposition}\label{prop: associated primes and splitting}
    Let $R$ be an F-finite ring and $X$ be a finitely generated $R$-module.

    \begin{enumerate}
        \item \label{part: local cohomology modules are killed by the first power} If there is a finitely generated module $M$ such that $\Fmult{e}{M}{X}>0$ for infinitely many positive integers $e$, then $I\Hh^0_I(X)=0$ for any ideal $I$. In particular, $\Ann(X)$ is a radical ideal.

        \item 
        If $X$ is an F-contributor at $\p$ and $R$ is local, then for any ideal $I$ such that $\Hh^0_I(X)$ is nonzero, $\dim(R/I) \geq \dim (R/\p)$. In particular, there are no associated primes of $X$ properly containing $\p$.
    \end{enumerate}
\end{proposition}

\begin{proof}
(\ref{part: local cohomology modules are killed by the first power}) 
Choose $e_0$ large enough so that $IF^e_*\Hh^0_I(M)=0$ for all $e \geq e_0$. Since $\Hh^0_I(X)$ is a summand of $F^e_*\Hh^0_I(M)$ for some $e \geq e_0$, $I\Hh^0_I(X) = 0$.

(2) Applying the $\Hh^0_I(-)$ functor to the decomposition of $F^e_*X$, we conclude that 
$$\limsup_{e \to \infty} \Fmult{e}{\Hh^0_I(X)}{\Hh^0_I(X)}/p^{e\alpha(R_\p)}>0.$$ 
If $\Hh^0_I(X)$ is nonzero, its depth must be at least $\dim(R/\p)$ by \Cref{lemma: depth vs frequency of appearance}. Thus $\dim(\Hh^0_I(X))$ is at least $\dim(R/\p)$. Since $I \subseteq \Ann(\Hh^0_I(X))$, $\dim(\Hh^0_I(X)) \leq \dim(R/I)$.
\end{proof}

We study the behavior of F-contributors under the Grothendieck duality. As a consequence, we produce examples of F-contributors from F-contributors at a prime. We refer to \cite{stacks-project} for dualizing complex and local duality.

\begin{proposition}\label{prop: Contributors under Grothendieck duality}
Let $(R, \m)$ be a complete local domain, $X$ be an F-contributor at a prime ideal $\p$. Let $\omega_R^\bullet$ be a normalized dualizing complex of $R$. Then,

\begin{enumerate}
    \item \label{part: Grothendieck dual of contributors at a prime} $\Ext_R^{- \dim(R/\p)}(X, \omega_R^\bullet)$ is an F-contributor of dimension $\dim(R/\p)$.

    \item \label{part: Grothendieck dual of contributors} For an F-contributor $N$, $\Ext_R^{- \dim(N)}(N, \omega_R^\bullet)$ is an F-contributor.
\end{enumerate}
\end{proposition}

\begin{proof}
(\ref{part: Grothendieck dual of contributors at a prime}) For a finitely generated $R$-module $M$, denote $K^i(M)= \Ext_R^{-i}(M, \omega_R^\bullet)$. The module $K^{\dim(R/\p)}(X)$ has dimension at most $\dim(R/\p)$, thus $\delta(K^{\dim(R/\p)}(X)) \leq \alpha(R_\p)$; cf. \Cref{thm: Kunz rank}. We argue that $\p \in \Ass(K^{\dim(R/\p)}(X))$.
For that, note, $(\omega_R^\bullet)_\p[-\dim(R/\p)]$ is a normalized dualizing complex of $R_\p$, as $R$ is an excellent domain. We combine this observation with local duality for $R_\p$ to conclude 
\[K^{\dim(R/\p)}(X)^{\wedge}_\p \cong 
\Hom_{R_\p} (\Hh^0_{\p R_\p}(X_\p), E_{R_\p}(\kappa(\p))).\]
Since the object on the right has $\p R_\p$ as an associated prime, $\p R_\p \in \Ass(K^{\dim(R/\p)}(X)^{\wedge}_\p)$. Therefore, $\p \in \Ass(K^{\dim(R/\p)}(X))$. Hence, $\dim(K^{\dim(R/\p)}(X))= \dim(R/\p)$ and $\alpha(R_\p)= \delta(K^{\dim(R/\p)}(X))$. Thus, it remains to check
\[\Fmult{e}{K^{\dim(R/\p)}(X)}{K^{\dim(R/\p)}(X)}= \Theta(p^{e\alpha(R_\p)}).\]
To that end, apply $\Hom_R(\bullet, \omega_R^{\bullet})$ to the direct sum decomposition $F^e_*X \cong X^{\oplus \Fmult{e}{X}{X}} \oplus \cdots$, and take cohomology at degree $-\dim(R/\p)$ to conclude that
\[F^e_* K^{\dim(R/\p)}(X)\cong K^{\dim(R/\p)}(X)^{\oplus \Fmult{e}{X}{X}} \oplus \cdots ; \]
here we have used Grothendieck duality along $F^e$ to compute the left side. Since $\Fmult{e}{X}{X}= \Theta(p^{e\alpha(R_\p)})$, we are done by 
\Cref{cor: convergence radius is at least the power of dimension}.

For (\ref{part: Grothendieck dual of contributors}), use part (\ref{part: Grothendieck dual of contributors at a prime}) for a prime ideal in $\Supp(N)$ of dimension $\dim(N)$.
\end{proof}

\subsection{Properties of F-splitting signature}\label{subsection properties of signatures}

\subsubsection{Rank formula}\label{sse: rank formula}
We will now investigate when equality holds in Corollary~\ref{cor: rank inequalities}. 
This is partially prompted 
by a property of the classical F-signature:
\[\gsig{M}{R}= \rank_R(M)\gsig{R}{R}; \quad \text{\cite[Theorem 4.11]{Tucker}}.\]
This rank formula uses the theory of F-splitting ideals, which has no known extension to our setting. Therefore, we use the F-decomposition matrix instead. Using the theory developed here, we prove a rank formula towards any contributor of dimension $\dim(R)$ when $R$ is strongly F-regular (\Cref{th: rank formula over sfr rings}) and in dimension up to two (\Cref{th: rank formula in dimension two}).

\begin{definition}\label{de: rank formula}
Let $N$ be a finitely generated indecomposable $R$-module of period $\omega$. For a finitely generated $R$-module $M$ and $\p \in \Ass(N)$, if
\[\csig{M}{N}_{\p}= \frac{\grk{\p}{M}\csig{N}{N}_{\p}}{\grk{\p}{N}},\]
we say that the \emph{\ces rank formula for $M$ towards $N$ at $\p$} holds. If for some integer $t$,
\[\gsig{M}{N}_{t, \omega, \p}= \frac{\grk{\p}{M}\gsig{N}{N}_{0, \omega, \p}}{\grk{\p}{N}}\]
holds, we say that \emph{the rank formula for $M$ towards $N$ with phase $t$ and period $\omega$ at $\p$} holds.

When the period of $N$ is one, the two notions of rank formula coincide. 
\end{definition}

General properties of $\Mat(M)$ guarantee the rank formula within connected classes.

\begin{lemma}\label{lem: rank formula within connected classes}
Suppose $\p \in \Ass(M)$. For any two modules $N',N$ in the same connected class of $\Mat(\p,M)$, the \ces rank formula for $N'$ towards $N$ at $\p$ holds.
\end{lemma}
\begin{proof}
Without loss of generality, we may assume that the connected class in question is $p^{-\alpha(R_\p)}$-positive recurrent; otherwise, all involved \ces limits are zero. It then follows from \Cref{th: existence of signatures in general} and \Cref{prop: properties of f}(\ref{part: evetual entrance probability depends only on class}) that
\[\csig{N'}{N}_{\p}= \mathfrak f^*_{N'N}(\p)\csig{N}{N}_{\p}= \frac{\grk{\p}{N'}}{\grk{\p}{N}}\csig{N}{N}_{\p}. \qedhere\]
\end{proof}

The next result compares the two notions of rank formula and derives formal consequences of each. Since both the \ces rank formula and the rank formula for a phase towards $N$ hold trivially when $\gsig{N}{N}_{0, \omega, \p} = 0$, we may assume that $\gsig{N}{N}_{0, \omega, \p}$ is positive.

\begin{proposition}\label{pr: comparison of notions of rank formulas and consequences}
Let $M$ be a finitely generated $R$-module and $N \in \find(\p, M)$ have period $\omega$. Assume that $\gsig{N}{N}_{0, \omega, \p}$ is positive. In the following assertions, the rank formulas are evaluated at $\p$.
\begin{enumerate}
    \item \label{part: rank formula for phase implies Cesaro and vanishing of phases} If the rank formula for $M$ towards $N$ for some phase $t$ holds, then the \ces rank formula for $M$ towards $N$ holds and $\Fmult{r}{M}{N}=0$ unless $r \equiv t \pmod \omega$.

    \item \label{part: rank formula at phase implies rank formula for all summands} If the rank formula for $M$ towards $N$ holds for some phase $t$, then for any $N' \in \find(\p, M)$ such that $\Fmult{e}{M}{N'}>0$, the rank formula for $N'$ towards $N$ holds for the phase $t-e$.

    \item \label{part: cesaro rank formula for a module implies that for its summands} The \ces rank formula for $M$ towards $N$ holds if and only if for any $N' \in \find(\p, M)$, the \ces rank formula for $N'$ towards $N$ holds. 

    \item \label{part: Cesaro rank formula implies a unique recurrent class} If the \ces rank formula for $M$ towards $N$ holds, then $\Mat(\p,M)$ has convergence radius $p^{-\alpha(R_\p)}$ and exactly one basic recurrent class, namely the one containing $N$. In particular, there is no null $p^{-\alpha(R_\p)}$-recurrent class. 

    \item \label{part: cesaro rank formula is a class property of the target} The \ces rank formula for $M$ towards $N$ holds if and only if for any $N' \in \find(\p,M)$ with positive $\csig{N'}{N'}_{\p}$, the \ces rank formula for $M$ towards $N'$ holds. 
\end{enumerate}
\end{proposition}

\begin{proof}
By \Cref{le: infinite period implies zero limit}, $\omega< \infty$. By \Cref{th: existence of signatures in general}
    and \Cref{le: first entrance probability for any module}
    the rank formula for $M$ towards $N$ at $\p$  
    at phase $t$ (resp., the \ces rank formula) holds if and only if 
    $\mathfrak f^*_{MN}(\p,t)= \grk{\p}{M}/\grk{\p}{N}$ (resp., $\mathfrak f^*_{MN}(\p)= \grk{\p}{M}/\grk{\p}{N}$). 
    From now on, we will work with these equivalent characterizations. 

    In (\ref{part: rank formula for phase implies Cesaro and vanishing of phases}),  \Cref{le: first entrance probability for any module} shows that 
    $\grk{\p}{M}/\grk{\p}{N} = \mathfrak f^*_{MN}(\p) = \mathfrak f^*_{MN}(\p,t)$, so the claim follows from 
    \Cref{le: first entrance probability for any module}(\ref{part: vanishing of the entrance probability}).

    (\ref{part: rank formula at phase implies rank formula for all summands}) For any $e \geq 1$, rewrite $\grk{\p}{M}/\grk{\p}{N}= \grk{\p}{F^e_*M}/p^{e\alpha(R_\p)}\grk{\p}{N}=\mathfrak f^*_{MN}(\p,t)$ and use \Cref{le: first entrance probability for any module} to conclude 
\begin{equation*}
\sum_{\substack{N_1 \in \find(\p,M)\\ \Fmult{e}{M}{N_1} \neq 0}}\frac{\Fmult{e}{M}{N_1}\grk{\p}{N_1}}{p^{e\alpha(R_\p)}\grk{\p}{N}}
= 
\mathfrak f^*_{MN}(\p,t) = \sum_{N_1 \in \find(\p,M)}\frac{\Fmult{e}{M}{N_1}\mathfrak f^*_{N_1N}(\p,t-e)}{p^{e\alpha(R_\p)}}.
\end{equation*}
By \Cref{le: first entrance probability for any module}, $\mathfrak f^*_{N_1N}(\p,t-e) \leq \mathfrak f^*_{N_1N}(\p)  \leq \grk{\p}{N_1}/\grk{\p}{N}$, so the right-hand side of the above equality is at most the left-hand side. Thus, $\mathfrak f^*_{N_1N}(\p,t-e) = \grk{\p}{N_1}/\grk{\p}{N}$ for all $N_1$ such that $\Fmult{e}{M}{N_1} \neq 0$.  

(\ref{part: cesaro rank formula for a module implies that for its summands}) is analogous to (\ref{part: rank formula at phase implies rank formula for all summands}). 

(\ref{part: Cesaro rank formula implies a unique recurrent class}) By the rank formula, $N$ is an F-contributor to $M$ at $\p$. So, by \Cref{thm: F-contributor at a prime via matrices} 
the connected class $\cC$ of $N$ is basic positive recurrent and 
$R(\Mat(\p,M)) = p^{-\alpha(R_\p)}$.
Likewise, by (\ref{part: cesaro rank formula for a module implies that for its summands}), $\cC$ is the only final class. 
By \Cref{pr: positive eigenvector of the block corresponding to a prime} and \Cref{thm: infinite basic is final}, any basic 
recurrent class $\cS$ of $\Mat(\p,M)$ is final, hence $\cS = \cC$.

(\ref{part: cesaro rank formula is a class property of the target}) 
Assume that $\mathfrak f^*_{MN}(\p)= \grk{\p}{M}/\grk{\p}{N}$. For any $N' \in \find(\p,M)$ such that $\csig{N'}{N'}_{\p} > 0$, $N$ and $N'$ belong to the same connected class by (\ref{part: Cesaro rank formula implies a unique recurrent class}). So,
$\mathfrak f^*_{MN'}(\p)= \grk{\p}{M}/\grk{\p}{N'}$ by \Cref{le: first entrance probability for any module}, (\ref{part: class property of entrance proability module case}). For the converse direction, set $N'=N$.     
\end{proof}

The validity of the rank formula puts constraints on the contributor:

\begin{proposition}\label{pr: rank formula implies ann is prime}
Let $X$ be an indecomposable F-contributor at $\p$. Suppose there is a finitely generated $R$-module $M$ such that $\p \in \Supp(M)$, $\p^nM=0$ for some $n$,  and \ces rank formula for $M$ towards $X$ at $\p$ holds. Then, $\Ann(X)= \p$. 

In particular, if the \ces rank formula at $\p$ towards $X$ holds for any finitely generated $R$-module, then $\Ann(X)=\p$.    
\end{proposition}

\begin{proof}
Choose $e$ large enough such that $\p F^e_*M=0$ and $X$ is a summand of $F^e_*M$. Then $\p X=0$. Hence, any minimal prime in $\Supp(X)$ contains $\p$. So, the minimal prime must be $\p$. Therefore the radical of $\Ann(X)$ is $\p$. Since $\p \subseteq \Ann(X)$, $\Ann (X)= \p$.

If the \ces rank formula holds for any finitely generated module, it holds for the module $\Hh^0_\p(X)$. Thus the previous part applies.    
\end{proof}

\begin{proposition}
  \label{prop: characterize rank formula}
Let $\p$ be a minimal prime ideal of $R$. Let $X$ be an indecomposable F-contributor at $\p$. Then 
the \ces rank formula towards $X$ at $\p$ holds for any finitely generated $R$-module if and only if the function $M \mapsto \csig{M}{X}_\p$ is additive in short exact sequences of finitely generated $R$-modules.   
\end{proposition}

\begin{proof}
    First, assume that the rank formula holds. To prove the additivity, it suffices to argue that $\bullet \mapsto \grk{\p}{\bullet}$ is additive over short exact sequences of finitely generated modules. Since $\p$ is minimal, $R_\p$ is Artinian, so $\Hh^0_{\p R_\p}(M_\p) = M_\p$. Since length is additive, the desired additivity of $\bullet \mapsto \csig{\bullet}{X}_\p$ follows.
    
    Assume additivity. Since $\p$ is not an associated prime of $M/\Hh^0_\p(M)$, and $\p \in \Ass(X)$, $\csig{M/\Hh^0_\p(M)}{X}_\p$ is zero. Additivity implies $\csig{M}{X}_\p= \csig{\Hh^0_\p(M)}{X}_\p$. 
Take a prime cyclic filtration on $\Hh^0_\p(M)$ and apply $\csig{\bullet}{X}_\p$. Since for any prime ideal $\q \neq \p$, $\p \notin \Ass(R/\q)$, $\csig{R/\q}{X}_\p=0$. Additivity implies $\csig{\Hh^0_\p(M)}{X}_\p$ is $\csig{R/\p}{X}_\p$ times the multiplicity of $R/\p$ in the prime cyclic filtration.
Since any prime in $\Supp(\Hh^0_\p(M))$ contains $\p$, by localizing the filtration at $\p$, we conclude that the multiplicity is $\grk{\p}{M}$. 
Thus $\csig{M}{X}_\p= \grk{\p}{M}\csig{R/\p}{X}_\p$. Applying this to $M= X$, we get $\csig{M}{X}_\p= \grk{\p}{M}\csig{X}{X}_\p/\grk{\p}{X}$.
\end{proof}

In contrast with \Cref{prop: characterize rank formula}, when the prime is not minimal, it is possible that \ces rank formula holds, but the additivity over short exact sequences fails; see \Cref{re: rank formula does not imply additivity}.

The following lemma concerns splitting of simple modules from Frobenius pushforwards of any module, and proves the rank formula towards simple modules. 

\begin{lemma}\label{le: dimension zero contributor}
 Let $S$ be an F-finite ring, where the category of finitely generated $S$-modules do not necessarily satisfy the KRS condition. Let $\m$ be a maximal ideal of $S$. Then, for any finitely generated $S$-module $M$ and large enough $e$,
    \[\Fmult{e}{M}{S/\m}= \grk{\m}{M} p^{e\alpha(S_\m)}.\]
\end{lemma}

\begin{proof}
Note that, $\m F^e_*M \cap \Hh^0_\m(F^e_*M)= F^e_*\m^{[p^e]}M \cap F^e_*\Hh^0_\m(M)= F^e_*(\m^{[p^e]}M \cap \Hh^0_\m(M))$. There exists an $r \in \mathbb N$, such that $\m^r\Hh^0_\m(M)=0$.
Thus, using Artin-Rees lemma, we conclude that for large enough $e$,  $\m^{[p^e]}M \cap \Hh^0_\m(M)=0$. Also, for large enough $e$, $\m\Hh^0_\m(F^e_*M)=0$. That is, for any large enough $e$, the $S$-module structure on $\Hh^0_\m(F^e_*M)$ comes from the natural $S/\m$-vector space structure on $\Hh^0_\m(F^e_*M)$. Consider the composition
\[\Hh^0_\m(F^e_*M) \xrightarrow[]{\text{inclusion}} F^e_*M \xrightarrow[]{\text{quotient}} \frac{F^e_*M}{\m F^e_*M}.\]
It follows from the discussion above that for any large enough $e$, the composition is an injective $S$-linear map between two $S/\m$-modules. Thus, the composition admits an $S/\m$-linear and hence an $S$-linear splitting projection. Hence, $\Hh^0_\m(F^e_*M)$ is an $S$-module summand of $F^e_*M$, for large enough $e$.  Since $\dim_{S/\m}\Hh^0_\m(F^e_*M)= \grk{\m}{M}p^{e\alpha(S_\m)}$, for large enough $e$ and any copy of $S/\m$ splitting from $F^e_*M$ must split from the submodule $\Hh^0_\m(F^e_*M)$. Hence, the lemma is established.
\end{proof}

\begin{remark}\label{re: rank formula does not imply additivity}
For a maximal ideal $\m$ in an F-finite ring $R$ which is not minimal, the map $\bullet \mapsto \grk{\m}{\bullet}$ is not additive over short exact sequences of finitely generated modules. Thus, although rank formula at $\m$ holds towards $R/\m$, the additivity may fail.
\end{remark}

\subsubsection{Total rank formula}\label{sse: Total rank formula}
We study another property of the decomposition of iterated Frobenius pushforwards of modules related to F-splitting signatures. For any ring $R$, a finitely generated $R$-module $M$, write
\[F^e_*M= \oplus_{N \in \find(M)}N ^{\oplus \Fmult{e}{M}{N}},\ \text{for any}\ e \geq 1.\]
For a prime ideal $\p$ of $R$, applying the map $\grk{\p}{\bullet}$, one gets
\[\grk{\p}{M}= \frac{\grk{\p}{F^e_*M}}{p^{e\alpha(R_\p)}}= \sum_{N \in \find(M)}\frac{\Fmult{e}{M}{N}}{p^{e\alpha(R_\p)}} \grk{\p}{N}.\]
In the sequel, we investigate the validity of a limiting version of the above formula, under the nomenclature of the total rank formula.

\begin{definition}\label{de: rank quantization}
Let $M$ be a finitely generated $R$-module. For a prime ideal $\p$ of $R$, if 
\[\grk{\p}{M}= \sum_{N \in \find(\p, M)}\csig{M}{N}_{\p} \grk{\p}{N}\]
holds, we say that the \emph{total rank formula} at $\p$ holds for $M$ (or simply that the total rank formula holds for $M$, when $\p$ is clear from context).
\end{definition}

In general, the total rank formula might not hold, although a total rank inequality holds. 

\begin{example}\label{eg: failure of rank quantization}
The total rank formula might fail for some modules. For example, let $R$ be the completion of the coordinate ring of a curve of genus two or higher considered in \Cref{eg: a matrix with no final class}, and take $M=R$. Since $p^{-e\delta(R)}\Fmult{e}{R}{N} \to 0$ for any $N\in \find(R)$, the total rank formula at the zero ideal cannot hold for $R$. 

Note that $\find(M)= \find((0), R)$ contains no F-contributors. In \Cref{prop: total rank formula vs rank formula}, we show that in the presence of an F-contributor, the total rank formula holds for modules in the connected class of the contributor and in \Cref{cor: total rank formula in dimension upto two}, we show that in dimension at most two, presence of a contributor implies the total rank formula. This and other evidence leads up to \Cref{conj: total rank formula}. 
\end{example}

\begin{proposition}[Total rank inequality]\label{cor: rank semiadditive}
Let $M$ be a finitely generated $R$-module. Then for any prime ideal $\p \subset R$
\[
\sum_{N \in \find(\p, M)} \csig{M}{N}_{\p} \grk{\p}{N} \leq \grk{\p}{M}.
\]
\end{proposition}

\begin{proof}
For any $N \in \find(\p,M)$, define the sequence of numbers $(f_e(N))_e$ by
\[f_e(N)\coloneqq \frac 1 e \grk{\p}{N} \sum_{n=1}^{e} p^{-n\alpha(R_\p)}\Fmult{n}{M}{N}.
\]
By \Cref{th: existence of signatures in general}, $\lim_{e \to \infty} f_e(N) = \grk{\p}{N}\csig{M}{N}_{\p}$.
On the other hand, for any $e \geq 1$
\[\sum_{N \in \find(\p,M)} f_e(N)= \grk{\p}{M}.\]
Thus, taking the limit in the equation above and applying Fatou's lemma for series, we conclude 
\[\sum_{N \in \find(\p,M)}\grk{\p}{N}\csig{M}{N}_{\p}= \sum_{N \in \find(\p,M)} \lim_{e \to \infty} f_e(N)\leq \liminf_{e \to \infty} \sum_{N \in \find(\p,M)} f_e(N)= \grk{\p}{M}.\]
\end{proof}

In the rest of this subsubsection, we show that total rank formula has strong consequences. 
The first immediate one, contained in the next remark, roughly says that validity of the total rank formula for $M$ at $\p$ forces low geometric growth of $\Fmult{e}{M}{N}$, if $N$ is not an F-contributor at $\p$. In fact, if the total rank formula holds, $\Mat(\p,M)$ cannot have a $p^{-\alpha(R_\p)}$-null recurrent class; see \Cref{prop: consequences of the total rank formula}.

\begin{remark}\label{re: interpretation of total rank formula}
Let $\cS \subseteq \find (\p, M)$
consist of those modules which are not F-contributors at $\p$.
By the proof of \Cref{cor: rank semiadditive}, if the total rank formula holds then
\[\lim_{n \to \infty}  \sum_{N \in \cS} \grk{\p}{N}\left(\frac{\sum_{e=1}^{n}\Fmult{e}{M}{N}/p^{e\alpha(R_\p)}}{n}\right)=0.\]
That is, not only does the sequence $(\Fmult{n}{M}{N}/p^{n\alpha(R_\p)})_n$ converge to zero, but the rate of convergence is sufficiently fast that the weighted sum on the left of the above equation also converges to zero.

Moreover, if $\lim_{e \to \infty} \Fmult{e}{M}{N} p^{-e\alpha(R_\p)}$ exists for every $N \in \find (\p,M)$ (e.g. when every F-contributor at $\p$ has period one), the same holds for the terms themselves, by the same proof.
\end{remark}

\begin{lemma}\label{lem: total rank formula implies total rank formula for summands}
Let $\p$ be a prime ideal of a ring $R$, and let $M$ be a finitely generated $R$-module. Then the following are equivalent: 
\begin{enumerate}
    \item the total rank formula at $\p$ holds for $M$; 
    \item there is $e \geq 0$ such that the total rank formula at $\p$ holds for all indecomposable summands of $F_*^e M$;
    \item the total rank formula at $\p$ holds for all $L \in \find(\p,M)$.
\end{enumerate}
\end{lemma}
\begin{proof} 
Since $\csig{M}{X}_{\p}= \mathfrak f^*_{MX}(\p)\csig{X}{X}_{\p}$ by \Cref{th: existence of signatures in general}, we obtain that
\begin{equation}\label{eq: sum of total rank inequalities}
        \sum_{X \in \find(\p,M)}\grk{\p}{X}\csig{M}{X}_{\p}= \sum_{X \in \find(\p,M)}\grk{\p}{X}\mathfrak f^*_{MX}(\p)\csig{X}{X}_{\p}.
\end{equation}
 
Substitute the expression for $\mathfrak f^*_{MX}(\p)$ from \Cref{le: first entrance probability for any module} into the right-hand side of (\ref{eq: sum of total rank inequalities}) and then apply (\ref{eq: sum of total rank inequalities}) 
again to get, for any $e \in \mathbb N$, 
    \begin{equation}\label{eq: expaning total rank inequality}
        \begin{split}
            &\sum_{X \in \find(\p,M)}\grk{\p}{X}\csig{M}{X}_{\p}  = \sum_{N \in \find(\p,M)}\frac{\Fmult{e}{M}{N}}{p^{e\alpha(R_\p)}}\sum_{X \in \find(\p,M)}\grk{\p}{X}\csig{N}{X}_{\p}\\
           & \leq \sum_{N \in \find(\p,M)}\frac{\Fmult{e}{M}{N}}{p^{e\alpha(R_\p)}} \grk{\p}{N}= \grk{\p}{M}, \ \text{cf. \Cref{cor: rank semiadditive} for the inequality}.
        \end{split}
    \end{equation} 
It is easy to see from above that the total rank formula for $M$ at $\p$ holds if and only if it holds for all indecomposable $N$ such that 
$\Fmult{e}{M}{N} \neq 0$. All the equivalences now follow. 
\end{proof}

\begin{proposition}\label{prop: consequences of the total rank formula}
    Let $M$ be a finitely generated $R$-module and $\p \in \Ass (M)$.
    If the total rank formula at $\p$ holds for $M$, then
    \begin{enumerate}
        \item $\find(\p,M)$ contains an F-contributor at $\p$. 
        In particular, $\Mat(\p,M)$ has convergence radius $p^{-\alpha(R_\p)}$. \label{part: total rank -> F-contributor}
        \item $\Mat(\p,M)$ does not have any $p^{-\alpha(R_\p)}$-null recurrent classes. \label{part: total rank -> no null basic}
        \item Final classes of $\Mat(\p,M)$ are precisely 
        positive $p^{-\alpha(R_\p)}$-recurrent classes, i.e., 
        consist of F-contributors at $\p$. \label{part: total rank -> final = positive basic}
    \end{enumerate}
\end{proposition}
\begin{proof}
Since $\grk{\p}{M} > 0$, there is $N \in \find(\p, M)$ such that $\csig{M}{N}_{\p} > 0$. Rest of (\ref{part: total rank -> F-contributor}) 
now follows from Corollary~\ref{cor: positive ces signature -> F-contributor} and \Cref{thm: F-contributor at a prime via matrices}. 

By \Cref{pr: positive eigenvector of the block corresponding to a prime} and \Cref{thm: infinite basic is final}, any $p^{-\alpha(R_\p)}$-recurrent class of $\Mat(\p,M)$ is final.
Now, if $N$ belongs to a final class $\cC \subseteq \find(\p, M)$, then the total rank formula holds for $N$ by \Cref{lem: total rank formula implies total rank formula for summands}, hence $\cC$ is positive recurrent by (\ref{part: total rank -> F-contributor}). This shows (\ref{part: total rank -> no null basic}) and (\ref{part: total rank -> final = positive basic}). 
\end{proof}

The next result compares the validity of the total rank formula and the \ces rank formula. This is used later to prove total rank formula when the ring is strongly F-regular (\Cref{cor: total rank formula in the SFR case}) or has dimension at most two (\Cref{cor: total rank formula in dimension upto two}).

\begin{proposition}\label{prop: total rank formula vs rank formula}
Let $M$ be a finitely generated $R$-module and $\p \in \Ass(M)$. Then, there exists an F-contributor $N$ at $\p$,
such that the \ces rank formula for $M$ towards $N$ at $\p$ holds if and only if the total rank formula at $\p$ holds for $M$ and $\Mat(\p,M)$ has exactly one final class.

In particular, if $N$ is an indecomposable F-contributor at $\p$, then the total rank formula at $\p$ holds for every $N' \in \find(\p,N)$.
\end{proposition}

\begin{proof}
First, suppose that for some F-contributor at $\p$, denoted by $N$, the \ces rank formula for $M$ towards $N$ at $\p$ holds. 
Let $\cC \subseteq \find (\p, M)$ be the connected class containing $N$. 
By \Cref{pr: comparison of notions of rank formulas and consequences}(\ref{part: Cesaro rank formula implies a unique recurrent class}), $\Mat(\p,M)$ has convergence radius $p^{-\alpha(R_\p)}$ and $\cC$ is the only $p^{-\alpha(R_\p)}$-recurrent class. Thus, by \Cref{thm: F-contributor at a prime via matrices}, $\cC$ is precisely the set of $L \in \find(\p, M)$
such that $\csig{M}{L}_{\p} > 0$ and, by \Cref{pr: comparison of notions of rank formulas and consequences}(\ref{part: cesaro rank formula is a class property of the target}), \ces rank formula for $M$ towards $L$ at $\p$ holds. Therefore, by applying first 
the rank formula and then \Cref{th: matrix version of total rank formula}, we conclude
\begin{equation*}
        \sum_{L \in \find(\p,M)} \grk{\p}{L}\csig{M}{L}_{\p}
        = \sum_{L \in \cC} \grk{\p}{M}\csig{L}{L}_{\p}
        = \grk{\p}{M}, 
\end{equation*}
i.e., the total rank formula at $\p$ holds for $M$. The uniqueness of the final class of $\Mat(\p, M)$ follows from \Cref{prop: consequences of the total rank formula}.

We now prove the converse. By \Cref{prop: consequences of the total rank formula}, $\Mat(\p,M)$ has convergence radius $p^{-\alpha(R_\p)}$ and has unique basic recurrent class $\cC$. By 
\Cref{prop: consequences of the total rank formula}(\ref{part: total rank -> no null basic}), $\cC$ is positive recurrent. 
Therefore, use rank inequality (\Cref{cor: rank inequalities}) in the total rank formula to obtain
\[
\grk{\p}{M} = \sum_{N \in \find(\p,M)} \grk{\p}{N}\csig{M}{N}_{\p} = \sum_{N \in \cC} \grk{\p}{N}\csig{M}{N}_{\p} \leq \sum_{N\in \cC} \grk{\p}{M}\csig{N}{N}_{\p}.
\]
By \Cref{th: matrix version of total rank formula}, 
the rightmost term equals $\grk{\p}{M}$. Thus, for any $N \in \cC$, the rank inequality must be an equality, i.e.,  
$\grk{\p}{N}\csig{M}{N}_{\p} = \grk{\p}{M}\csig{N}{N}_{\p}$. This implies rank formula.

For the last assertion, $\find(\p,N)$ is the connected class of $\Mat(\p,N)$ containing $N$ by \Cref{thm: F-contributor at a prime via matrices},  $(2a) \Rightarrow (2b)$. Hence, for $N' \in \find(\p,N)$, the \ces rank formula for $N'$ towards the F-contributor $N$ holds by \Cref{lem: rank formula within connected classes}, and the first part applies to $M = N'$.
\end{proof}

\begin{remark}\label{re: probabilistic interpretation of rank formula and total rank formula}
In terms of the Markov chain associated to $\Mat(\p,M)$ (see \Cref{re: Markov chain associated to an associated prime}), the rank formulas have the following meaning.
\begin{enumerate}
    \item The total rank formula for $M$ at $\p$ is a natural analogue of the nondissipativeness of this chain; see \cite{Chung_Kai_Lai_Markov_Chains_with_stationary_transition_probabilities}. The same notion is called \emph{positive Tweedie recurrence} in \cite[Definition 18.8]{Kuntz} and goes back to Kendall \cite{Kendall_on_non-dissipative_Markoff_chains_with_an_enumerable_infinity_of_states}.
    \item Let $N \in \find(\p,M)$ with $\csig{N}{N}_\p>0$, so that the connected class of $N$, namely $\find(\p,N)$, is positive recurrent. Then the \ces rank formula for $M$ towards $N$ at $\p$ holds if and only if every state $M' \in \find(\p,M)$ enters $\find(\p,N)$ with probability one. Indeed, by \Cref{pr: comparison of notions of rank formulas and consequences}(\ref{part: cesaro rank formula for a module implies that for its summands}), it is equivalent to the \ces rank formula for every $M' \in \find(\p,M)$ towards $N$. By \Cref{th: existence of signatures in general}(\ref{part: cesaro signature is scaled by entrance probability}), this means $\ep{M'N}(\p)= \grk{\p}{M'}/\grk{\p}{N}$, and $\grk{\p}{N}\ep{M'N}(\p)/\grk{\p}{M'}$ is the probability that $M'$ enters $\find(\p,N)$; see \Cref{prop: properties of f}(\ref{part: sum of entrance probabilities to recurrent classes is at most one}).
    \item Combining (2) with \Cref{prop: total rank formula vs rank formula}, the rank formula for $M$ at $\p$ towards some F-contributor at $\p$ is equivalent to the \emph{positive Harris recurrence} of the chain, i.e., nondissipativeness with exactly one essential class; see \cite[Definition 18.8]{Kuntz}.
\end{enumerate}
\end{remark}

\subsection{Examples}

We include some examples of F-decomposition matrices, and corresponding F-contributors at different associated primes. The following result is used to produce higher dimensional examples from a given example.

\begin{lemma}\label{le: adjoining variable}
    Let $(R, \m)$ be an F-finite complete local ring, $M$ be a finitely generated $R$-module and $R[z]$ be the polynomial ring in the variable $z$. Let $M[[z]]$ be the $R[[z]]$-module given by the $zR[z]$-adic completion of the $R[z]$-module $M \otimes_RR[z]$. Denote the Frobenius endomorphisms of $R[[z]]$ and $R$ by $F$ and $F_R$ respectively.
    \begin{enumerate}
        \item If $M$ is an indecomposable $R$-module, then $M[[z]]$ is an indecomposable $R[[z]]$-module.

        \item If for some $e \in \mathbb N$, $F_{R*}^eM \cong \oplus_{j=1}^{m}M_j$, where $M_j$ are indecomposable $R$-modules, not necessarily mutually nonisomorphic, then 
        \[F^e_*M[[z]] \cong \oplus_{j=1}^{m}\oplus_{i=0}^{p^{e}-1}M_j[[z]]F^e_*z^i,\]
        where $M_j[[z]]F^e_*z^i$ is isomorphic to $M_j[[z]]$ as an $R[[z]]$-module; the only role of $F^e_*z^i$ is to distinguish different copies of $M_j[[z]]$.
    \end{enumerate}
\end{lemma}

\begin{proof}
    (1) Since the direct summands of $M[[z]]$ correspond to idempotent elements in $\text{End}_{R[[z]]}(M[[z]])$, it suffices to prove that such an idempotent element is either zero or identity. Recall the canonical isomorphisms:
    $$\Hom_{R}(M, M[[z]]) \cong \Hom_{R[z]}(M\otimes_RR[z], M[[z]]) \cong \Hom_{R[[z]]}(M[[z]], M[[z]]),$$
    where the isomorphism sends $\phi \in \Hom_{R}(M, M[[z]])$ to
    $\widehat \phi \in \Hom_{R[[z]]}(M[[z]], M[[z]])$ given by 
    $$\widehat \phi(\sum_{n=0}^{\infty}x_nz^n)= \sum_{n=0}^{\infty} \phi(x_n)z^n, \ \forall (x_n)_{n \in \mathbb N} \in \prod_{\mathbb N}M.$$
  For any $\phi$ as before, let $(\delta_n\phi)_{n \in \mathbb N}$ be the unique collection of maps in $\Hom_R(M,M)$ such that
  \[\phi(x)= \sum_{n=0}^{\infty} \delta_n\phi(x)z^n, \forall x \in M.\]
  Now suppose that $\widehat \phi$ is idempotent. A straightforward calculation shows that $(\widehat \phi)^2= \widehat \phi$ if and only if $\widehat \phi \phi= \phi$. Since $\widehat \phi$ is $(z)R[[z]]$-adically continuous,
  \[\widehat \phi \phi(x)= \sum_{n=0}^\infty \widehat \phi(\delta_n\phi(x))z^n= \sum_{n=0}^\infty (\sum_{\alpha=0}^{n}\delta_\alpha  \phi \circ \delta_{n-\alpha} \phi(x))z^n= \sum_{n=0}^{\infty} \delta_n\phi(x)z^n, \forall x \in M.\]
Hence, for any $n \in \mathbb N$,
\begin{equation*}
    \sum_{\alpha=0}^{n}\delta_\alpha  \phi \circ \delta_{n-\alpha} \phi = \delta_n\phi.
\end{equation*}
  
In particular $(\delta_0\phi)^2= \delta_0\phi$. Since $M$ is indecomposable $\delta_0\phi$ is either zero or identity. In either case, using the last equation and induction on $n$, it follows that $\delta_n\phi$ is zero for all $n \geq 1$. Therefore, $\widehat \phi$ is either zero or the identity.

(2) Consider the map
\[\phi_1:\oplus_{i=0}^{p^{e}-1}F^e_{R*}M[[z]]F^e_*z^i \rightarrow F^e_*M[[z]],\]
where $F^e_*z^i$'s play a similar bookkeeping role as in the statement of the lemma, and,
$$\phi_1((\sum_{n=0}^\infty F_{R*}^e m_nz^n)F^e_*z^i)= F_{*}^e\sum_{n=0}^\infty  m_nz^{np^e+i}, \ \forall \ (\sum_{n=0}^\infty F_{R*}^e m_nz^n)F^e_*z^i \in F^e_{R*}M[[z]]F^e_*z^i, \forall i .$$
It is straightforward to verify that $\phi_1$ is $R[[z]]$-linear and bijective. Now, we use an $R$-linear isomorphism $F_{R*}^eM \cong \oplus_{j=1}^{m}M_j$, to get an $R[[z]]$-linear isomorphism
\[F_{R*}^eM[[z]]F^e_*z^i \cong \oplus_{j=1}^{m}M_j[[z]]F^e_*z^i,\]
for each $0 \leq i \leq p^e-1$. So (2) follows.
\end{proof}

\begin{example}\label{eg: contributors at a non minimal prime}
    For a positive integer $m$ and an F-finite field $k$, consider the quotient of the power series ring:
    \[R'= \frac{k[[x, y_1, \ldots, y_m]]}{(x^2, xy_1, xy_2, \ldots, xy_m)}\]
 One has a primary decomposition $(x^2, xy_1, xy_2, \ldots, xy_m)= (x,y_1, \ldots, y_m)^2 \cap (x)$. Denote the image of $f= f(x, y_1, \ldots, y_m)$ in $R'$ by $\bar f$. As $k$ vector spaces, $R' = k\bar x \oplus k[[\bar y_1, \ldots, \bar y_m]]$.
 Note $k\bar x$ and $k[[\bar y_1, \ldots, \bar y_m]]$ underlie the indecomposable $R'$-modules $R'/(\bar x,\bar y_1, \ldots, \bar y_m)\cong 0:_{R'} (\bar x,\bar y_1, \ldots, \bar y_m)R'$ and $R'/(\bar x)R' \cong k[[y_1, \ldots, y_m]]$, respectively. Moreover, as $R'$-modules,
 \begin{equation*}
 \begin{split}
     F_*R' \cong F_*\frac{R'}{(\bar x, \bar y_1, \ldots, \bar y_m)R'} \oplus F_*\frac{R'}{(\bar x)R'}& \cong F_*k \oplus F_*k[[y_1, \ldots, y_m]]\\ & \cong k^{\oplus p^{\alpha(R')}} \oplus k[[y_1, \ldots, y_m]]^{\oplus p^{(m+ \alpha(R'))}}.
 \end{split}
 \end{equation*}
 Therefore, it follows that $F^e_*R' \cong k^{\oplus p^{e\alpha(R')}} \oplus k[[y_1, \ldots, y_m]]^{\oplus p^{e(m+ \alpha(R'))}}.$
 So we conclude that $\find(R')= \{R'/(\bar x,\bar y_1, \ldots, \bar y_m)R' \cong k, R'/(\bar x)R' \cong k[[y_1, \ldots, y_m]]\}$ and
 \[\Mat(R')= \begin{bmatrix}
     p^{\alpha(R')} & 0\\
     0 & p^{\dim(R')+\alpha(R')} 
 \end{bmatrix} \ .\]
\textbf{Contributors:} $R'/(\bar x,\bar y_1, \ldots, \bar y_m) \cong k$ is a contributor to $R'$ at the zero-dimensional associated prime $(\bar x,\bar y_1, \ldots, \bar y_m)$, and $R'/(\bar x)R'$ is a contributor to $R'$ (at $(\bar x)R'$).
 
\end{example}

\begin{example}\label{eg: contributor at an associated prime of prescribed dimension}
    For a positive integer $n$, take \[R= \frac{k[[x, y_1, \ldots, y_m, z_1, \ldots, z_n]]}{(x^2, xy_1, xy_2, \ldots, xy_m)}.\]
    Using \Cref{le: adjoining variable} and \Cref{eg: contributors at a non minimal prime}, we conclude that
    $$\find (R)= \{\frac{R}{(\bar x,\bar y_1, \ldots, \bar y_m)R} \cong \frac{R'}{(\bar x,\bar y_1, \ldots, \bar y_m)R'}[[z_1, \ldots, z_n]], \frac{R}{(\bar x)R} \cong \frac{R'}{(\bar x)R'}[[z_1, \ldots, z_n]]\}$$
   Therefore, $\find(\p,R)$ is nonempty if and only if $\p= (\bar x)R$ or $\p= (\bar x,\bar y_1, \ldots, \bar y_m)R$. $\find(\p,R)$ for these two primes represent the two connected classes of $\Mat(R)$.  Moreover  \[\Mat(R)= \begin{bmatrix}
     p^{n+ \alpha(R)} & 0\\
     0 & p^{\dim(R)+\alpha(R)} 
 \end{bmatrix} \ .\]
 So, both of the connected classes are not exceptional, $p^{-\alpha(R_\p)}$-positive recurrent and final. The convergence radius of $\Mat(R)$ is $p^{-\dim(R)- \alpha(R)}$ and the connected class $\{R/(\overline{x})\}$ is basic positive recurrent.
\end{example}

\begin{example} \label{eg: a matrix with no final class}
Let $X$ be a smooth projective curve of genus at least two over an algebraically closed field $k$ of characteristic $p>0$. Let $\mathcal L$ be an ample invertible sheaf. Let $R$ be the corresponding $\mathbb N$-graded section ring,
\[R= \underset{j 
\in \mathbb N}{\bigoplus}H^0(X, \mathcal L ^j).\]
Denote the completion of $R$ at the unique homogeneous maximal ideal by $\widehat R$. We describe $\find(\widehat R)$ and $\Mat(\widehat R)$ in \Cref{le: DM for higher genus curves}, and record its further properties below.

To describe $\find(\widehat R)$, we set some convention. For $\beta \in [0, p^n-1]$, denote $\beta+ p^n\mathbb Z 
\in \mathbb Z/ p^n\mathbb Z$ by $(\beta, p^n)$. Endow
$$
S \coloneqq \underset{n \geq 1}{\coprod}\frac{\mathbb Z} {p^n\mathbb Z}$$
with the total order induced by: 
\begin{gather*}
\text{If}\,\, m<n,\,\, (\gamma, p^m) < (\beta, p^n),\, 
\forall \gamma, \beta; \,
\text{if}\,\, \beta_1< \beta_2, (\beta_1, p^n)<(\beta_2, p^n).
\end{gather*}

For $n \geq 1$ and an integer $\beta \in [0, p^n-1]$, set \[M_{\beta,p^n}= F^n_*\underset{m 
\in \mathbb N}{\bigoplus}R_{\beta + mp^n}.\]
Then $M_{\beta,p^n}$ is an  $\mathbb N$-graded module and is an $R$-submodule of $F^n_*R$. 

\begin{lemma}\label{le: DM for higher genus curves}
 The $R$-modules $M_{\beta,p^n}$, for any $n \geq 1$, $\beta \in [0, p^n-1]$, have the following properties:

\begin{enumerate}
    \item The $R$-module $M_{\beta,p^n}$ is the section module of the sheaf $F^n_{\ast}\mathcal L^{\beta}$ on $X$ with respect to $
    \mathcal L$. 

    \item The $\widehat R$-module $M_{\beta,p^n} \otimes_R \widehat R$ is indecomposable. 

    \item For different  $(\beta, p^n)$'s the corresponding $\widehat R$ modules $M_{\beta,p^n} \otimes_R \widehat R$ are not isomorphic.

    \item The decomposition of the $\widehat R$-module $F_*(M_{\beta,p^n}
    \otimes \widehat R)$ into a direct sum of pairwise nonisomorphic indecomposable $\widehat R$-modules is 

    \[F_*(M_{\beta,p^n}
    \otimes_R \widehat R ) \cong \underset{0 \leq j \leq p-1}{\bigoplus}M_{\beta+jp^n,p^{n+1}} \otimes \widehat R.\]

    \item The decomposition of $F^n_\ast \widehat R$ into a direct sum of indecomposable $\widehat R$-modules is given by

    \[F^n_\ast \widehat R \cong \underset{0 \leq \beta \leq p^n-1}{\bigoplus}M_{\beta, p^n} \otimes_R \widehat R.\]

    \item $\find(\widehat R)= \{M_{\beta,p^n} \otimes_R \widehat R\}_{(\beta,p^n) \in S }$. The $((\alpha, p^n), (\beta, p^m))$-th entry of the $S \times S$-matrix $\Mat (\widehat R)$ is nonzero if and only if $m=n+1$ and $\beta= \alpha+ j p^n$ for some $j \in [0,p-1]$, and the nonzero entry is one. 
\end{enumerate}
\end{lemma}

\begin{proof}
    (1) follows from a direct calculation. For (2), recall that the $\mathcal O_X$-module $F^n_\ast \mathcal L ^ \beta$ is a stable vector bundle by applying \cite[Proposition~1.2]{Lange_Pauly_on_Frobenius_destabilized_rank_2_vector_bundles_over_curves} and \cite[Theorem~2.2]{SunDirectImage08} inductively. So, it cannot be a direct sum of two nonzero $
    \mathcal O_X$-modules. Therefore, the $R$-module $M_{\beta,p^n}$, as a graded module, cannot be a direct sum of two nonzero graded $R$-modules. Recall that the $\widehat R$-module $M_{\beta,p^n} \otimes_R \widehat R$ is decomposable if and only if the $R$-module $M_{\beta,p^n}$, as a graded module, is a direct sum of two nonzero graded $R$-modules; see \cite[Theorem 3.8]{HashimotoYang23}. So, $M_{\beta,p^n} \otimes_R \widehat R$ is indecomposable. \\

    \noindent For (3), recall that  $M_{\beta,p^n} \otimes_R \widehat R$ and $M_{\gamma,p^m} \otimes \widehat R$ are isomorphic as $\widehat R$-modules if and only if there is an integer $l$ such that $M_{\beta,p^n}$ and $M_{\gamma,p^m}(l)$ are isomorphic as graded $R$-modules; see \cite[Corollary 3.9]{HashimotoYang23}.  The generic ranks of the $R$-modules $M_{\beta,p^n}$ and $M_{\gamma,p^m}(l)$ are $p^n$ and $p^m$, respectively, by (1). So $M_{\beta,p^n}$ and $M_{\gamma,p^m}(l)$ are isomorphic only if $m=n$. Now, a graded isomorphism between $M_{\beta,p^n}$ and $M_{\gamma,p^n}(l)$ will yield 

    $$\dim_k H^0(X, \mathcal L^{\beta+ jp^n})= \dim_k H^0(X, \mathcal L^{\gamma+ (j+l)p^n}) \,\,\, \text{for all j}.$$

    Large enough powers of $
    \mathcal L$ do not have any cohomologies in positive degrees. So  Riemann--Roch and the above equation imply that for all large enough $j$,

     \[(\beta+ jp^n)\deg(\mathcal L)= (\gamma+ lp^n + jp^n)\deg(\mathcal L).\]

     Since $\deg(\mathcal L)$ is positive and both $
     \beta$ and $\gamma$ are less than $p^n$, $\gamma= \beta$.\\

     \noindent The verification of (4) is straightforward. For (5) note 

     \[F^n_\ast R = \underset{0 \leq \beta \leq p^n-1}{\bigoplus}M_{\beta, p^n}.\]

      We get the desired direct sum decomposition in (5) by tensoring with $\widehat R$. (6) follows from the previous parts.
\end{proof}

We record some further properties of $\Mat(\widehat R)$.

\begin{enumerate}
    
    \item The connected classes of $\Mat(\widehat R)$ are all exceptional: For $(\gamma, p^m)\leq (\beta, p^n)$, \Cref{le: DM for higher genus curves}, (4) shows that there is no path in the graph $G(\Mat(\widehat R))$ from $(\beta, p^n)$ to $(\gamma, p^m)$; see \Cref{def: connected class}.

    \item The multiplication on the left by $\Mat({\widehat{R}})$ on $\underset{\mathbb{N}}{\prod}\mathbb {C}$, where the elements of the latter are thought of as column vectors, has any complex number as an eigenvalue. Moreover, an eigenvalue admits an eigenvector in $l^\infty (\mathbb C)$ (resp. in $l^1(\mathbb C)$) if and only if the eigenvalue has absolute value at most $p$ (resp. strictly less than 1).

    In particular, the spectral radius of $\Mat(\widehat R)$ for its action on $l^\infty(\mathbb C)$ or $l^1(\mathbb C)$ is different from the reciprocal of its convergence radius --- the convergence radius is infinity. This is in contrast with the case of finite size matrices with nonnegative entries.

    \noindent \textit{Proof:} Think of the column vectors as elements of $\underset{S}{\prod} \mathbb C$. For $v=(v_{(\beta, p^n)})_{(\beta, p^n) \in S}$, $$(\Mat(\widehat{R})v)_{(\beta, p^n)}= \sum_{j=0}^{p-1}v_{(\beta+jp^n, p^{n+1})}.$$
Therefore, $v$ is an eigenvector with eigenvalue $\lambda$ if and only if,
\begin{equation}\label{eq: equation for eigenvector}
\lambda v_{(\beta, p^n)}= \sum_{j=0}^{p-1}v_{(\beta+jp^n, p^{n+1})}.
\end{equation}
For any $\lambda \in \mathbb C$, the unique vector $v$ satisfying: $v_{(0,p)}=1$, $v_{(\beta, p)}=0$ otherwise, and $v_{(\beta+jp^n, p^{n+1})}= \lambda v_{(\beta, p^{n})}/p$ for any $n \geq 1$ and $j \in [0,p-1]$, is an eigenvector with eigenvalue $\lambda$. When $|\lambda|\leq p$ (resp. $<1$), this eigenvector is in $l^\infty(\mathbb C)$ (resp. $l^1 (\mathbb C)$). For an eigenvector $v$ with eigenvalue $\lambda$, assume $v_{(\beta,p^n)}$ is nonzero. Let $S'$ be the subset of elements of $\underset{m \geq n}\sqcup \mathbb Z/p^m\mathbb Z \subseteq S$ which map to $\beta+p^n\mathbb Z$ via reduction mod $p^n$. Then for any $m \geq n$
\[\underset{(\alpha,p^m) \in S'}{\sum}v_{(\alpha,p^m)}= \lambda^{m-n}v_{(\beta,p^n)}, \ \text{by} \ \eqref{eq: equation for eigenvector}.\]
When $|\lambda| \geq 1$, the series $\sum_m \lambda^{m-n}v_{(\beta,p^n)}$ is divergent. So there is no $l^1(\mathbb C)$ eigenvector when $|\lambda| \geq 1$. The last equation implies that for any $m \geq n$, $\max \{|v_{(\alpha,p^m)}| \ |\ (\alpha,p^m) \in S'\}$ is at least $(|\lambda|/p)^{m-n}|v_{(\beta,p^n)}|$. Therefore, there is no $l^\infty(\mathbb C)$ eigenvector when $|\lambda|>p.$

    \item The multiplication on the right by $\Mat({\widehat{R}})$ on $\underset{\mathbb{N}}{\prod}\mathbb {C}$, where the elements of the latter are thought of as row vectors, has no complex eigenvalue.

    Indeed, let $v$ be a row eigenvector with eigenvalue $\lambda \neq 0$. Since the first $p+p^2+ \cdots+p^n$ columns of $\Mat(\widehat R)^n$ are zero and $v\Mat(\widehat R)^n= \lambda^n v$, the first $p+p^2+ \cdots+p^n$ columns of $v$ must be zero for any $n$; forcing $v$ to be zero. A direct calculation shows that zero is not an eigenvalue for the right multiplication. 
\end{enumerate}  
\end{example}

\subsection{Beyond the Krull--Remak--Schmidt condition}\label{subsec: beyond KRS}
Most of the results preceding and following this subsection are proven under the KRS assumption on the relevant categories of finitely generated modules. We point out that many of those can be extended to the case when KRS is not satisfied using \Cref{le: passing to completion}. We demonstrate this by extending the existence of F-splitting signatures with periods and phases, i.e., parts of \Cref{th: existence of signatures in general}, in \Cref{prop: limit in the non KRS case}.

In this subsection, $(R, \m)$ denotes an F-finite local ring. We do not assume that the category of finitely generated $R$-modules satisfies the KRS condition. For any finitely generated $R$-module $M$, denote its $\m$-adic completion by $\widehat M$. The following result follows from a well-known result (\cite[Proposition]{Wiegand_Failure_or_Krull_Schmidt_for_direct_sums_of_copies_of_modules}), once we note: For any finitely generated $R$-module $M$ and any $e \in \mathbb N$, the completion of $F^e_*M$ is isomorphic to $F^e_{\widehat R*}\widehat M$.

\begin{lemma}\label{le: passing to completion} (cf. \cite{Wiegand_Failure_or_Krull_Schmidt_for_direct_sums_of_copies_of_modules}, Proposition)
    Let $M$ and $N$ be finitely generated $R$-modules over an F-finite local ring. Then $\FRmult{e}{M}{N}{R}= \FRmult{e}{\widehat M}{\widehat N}{\widehat R}$ for all $e \in \mathbb N$.
\end{lemma}

\begin{proposition}\label{prop: limit in the non KRS case}
Let $M$ and $N$ be a finitely generated $R$-modules. Let $\omega$ be the least common multiple of the periods of the indecomposable $\widehat R$-module summands of $\widehat N$; the least common multiple is taken to be infinity if one of the periods is infinity. Then, the sequence
\[
(\frac{\Fmult{t+ e\omega}{M}{N}}{p^{(t+e\omega)\delta_R(N)}})_e
\]
converges for any $t \in \mathbb Z$. When $\omega= \infty$, each term of the sequence is zero by definition.
\end{proposition}

\begin{proof}
We may assume $\omega< \infty$. Since $R$ is local and $\dim(N)= \dim(\widehat N)$, $\delta_R(N)= \delta_{\widehat R}(\widehat N)$. So, using \Cref{le: passing to completion}, we can replace $(R,M,N)$ by $(\widehat R, \widehat M, \widehat N)$ and assume that $R$ is complete, without loss of generality.
Suppose $N \cong \oplus_{i=1}^n N_i^{\oplus a_i}$, where $N_i$'s are pairwise nonisomorphic and indecomposable, and $a_i$ is positive for any $i$. Note that,
\[\Fmult{e}{M}{N}= \min \left\{ \left\lfloor \frac{\Fmult{e}{M}{N_i}}{a_i} \right \rfloor \mid 1\leq i \leq n \right\}, \ \forall e \in \mathbb N.\]

Assume $\delta_R(N)$ is positive. For any $i$,  $\delta_R(N)\geq \delta_R(N_i)$ and $(\Fmult{t + e\omega}{M}{N_i}/p^{(t+e\omega)\delta_R(N_i)})_e$ converges, so  
\[\lim_{e \to \infty}\frac{\Fmult{t+ e\omega}{M}{N}}{p^{(t+e\omega)\delta_R(N)}}=  \min \left \{ \lim_{e \to \infty} \frac{\Fmult{t+e\omega}{M}{N_i}}{a_ip^{(t+e\omega)\delta_R(N)}}  \mid 1\leq i \leq n \right \}.\]

In the remaining case, where $\delta_R(N)=0$, both $\dim(N)$ and $\alpha(R)$ must be zero. Hence $\m^{r'} N=0$ for some $r'$. We stick to the only interesting case where $\Fmult{t+ e\omega}{M}{N}$ is nonzero for infinitely many $e$. 
In this case $\m N=0$ by \Cref{prop: associated primes and splitting}. Thus, $N$ is isomorphic to a direct sum of copies of the residue field $R/\m$. It follows easily from \Cref{le: dimension zero contributor} that the sequence $\Fmult{e}{M}{N}$ is eventually constant with value $\lfloor \grk{\m}{M}/ \length(N)\rfloor$.
\end{proof}

\section{Revisiting the FFRT case}\label{se: FFRT case}
We now study modules $M$ for which $\find (M)$ is finite and modules with finitely many F-contributors.

\subsection{Finite Frobenius representation type}
We recall that $M$ has the FFRT property if $\find(M)$ is a finite set. Under this assumption, Yao has already established the existence of the limit defining F-splitting signatures. We recall Yao's result after a suitable modification. 

\begin{theorem}[{\cite[Theorems~A and C]{Yao05}}]\label{thm: Yao main}
Let $(R, \m)$ be an F-finite local ring that satisfies the KRS condition and $M$ be a finitely generated $R$-module that has the FFRT property.
Then $\lim_{e \to \infty} p^{-e\delta_R(M)} \Mat(M)^e$ exists 
and is a nonzero matrix with entries in $\QQ_{\geq 0}$. 
In particular for every $\p' \in \Minh(M)$, there is an indecomposable F-contributor to $M$ which has $\p'$ as an associated prime and for any F-contributor $X$ to $M$, the limit
\[
\gsig{M}{X}= \lim_{e \to \infty} \frac{\Fmult{e}{M}{X}}{p^{e \delta_R(M)}} \in \QQ_{> 0}.
\]
\end{theorem}

When $R$ is strongly F-regular with the FFRT property, 
\cite{SmithVandenBergh} deduced 
a similar statement using the Perron--Frobenius theory. This pioneering treatment is simpler than Yao's because strong F-regularity immediately implies aperiodicity; see \Cref{cor: aperiodicity implies existence of limit}. In the general FFRT case, 
matrix theory is augmented with the following result from commutative algebra: 

\begin{theorem}[{proof of \cite[Theorem~3.11]{Yao05}}]\label{thm: Yao span}
    Let $(R, \m)$ be a local ring satisfying KRS and $M_1, \ldots, M_n$ be pairwise nonisomorphic indecomposable modules. Then $\mathbb{Q}^n$ is spanned by the vectors 
    \[
    \{
    [\length (M_1 \otimes_R L), \ldots, \length (M_n \otimes_R L)] \mid \length (L) < \infty
    \}.
    \]
\end{theorem}

We will provide a new proof of Theorem~\ref{thm: Yao main} 
that will avoid Theorem~\ref{thm: Yao span} by using the theory of \emph{F-surjective signatures} built in \cite{SmirnovTucker2}. 

\begin{definition}[\cite{Hashimoto}]
    Let $M, N$ be two finitely generated modules. 
    The asymptotic surjective number of $M$ onto $N$ is the limit
    \[
    \asn (M \twoheadrightarrow N) \coloneqq \lim_{m \to \infty} \frac{
    \max\{k \mid \text{there is an $R$-linear surjection}\ M^{\oplus m} \twoheadrightarrow N^{\oplus k}\}
    }{m}.
    \]    
\end{definition}
Hashimoto (\cite{Hashimoto}) shows that the limit exists and 
\cite{SmirnovTucker2} shows that the limit defining $\asn (M \twoheadrightarrow N)$ is in fact the maximum value of the sequence involved: 

\begin{theorem}[{\cite[Theorem~2.4]{SmirnovTucker2}}]\label{thm: asn achieved}
Let $(R, \m)$ be a local ring and $M, N$ be finitely generated $R$-modules.
Then there exists $m \geq 1$ such that 
    \(
      \asn (M \twoheadrightarrow N) = \max \{n \mid M^{\oplus m} \twoheadrightarrow N^{\oplus n}\}/m.    
    \)
   
\end{theorem}

\begin{proposition}\label{prop: FFRT contributors have period one}
    Let $M$ be an $(R, \m)$-module with FFRT. Then any F-contributor to $M$ has period one.
\end{proposition}

\begin{proof}
Let $X$ be an indecomposable F-contributor to $M$. By \Cref{cor: Cohen--Macaulayless of contributors}, $X$ is an F-contributor, and since $\find(X) \subseteq \find(M)$, it satisfies FFRT. Replacing $M$ by $X$, we may assume that $\Mat(M)$ is irreducible and $M \in \find(M)$; see \Cref{cor: F-contributor matrix is irreducible}.
    For any $\p \in \Ass (M)$, $\Mat (M) = \Mat (\p, M)$ by 
    Lemma~\ref{lem: support of an irreducible class}, so $\Mat(M)$ has a positive eigenvector.  Let $\find (M) = \{M_1, \ldots, M_\ell\}$.
    
    By contradiction, let the period of $\Mat (M)$ be $\omega > 1$.   
    By \Cref{th: existence of signatures in general}, for any 
    integers $0 \leq r < \omega$ and $1 \leq i \leq \ell$
    the limit 
    \[
    s_{i, r} \coloneqq  \lim_{n \to \infty} \frac{\Fmult{n\omega + r}{M}{M_i}}{p^{(n\omega+r)\delta_R(M)}}
    \]
    exists. Moreover, since $\Mat(M)$ has integer entries, 
    $s_{i, r} \in \mathbb{Q}$ by \cite[Lemma~2.4]{SeibertFFRT} applied to $\Mat(M)^{\omega}$.
    Let $b_r$ be the common denominator of all $s_{i, r}$
    and define $N_r = \bigoplus_{i = 1}^\ell M_i^{b_r s_{i,r}}$. 
    
    Since the limit in the definition of the F-surjective signature
    \[
    \surjsig {M}{X} \coloneqq \lim_{e \to \infty} \frac{\max \{n \mid \text{ there exists } F_*^e M \twoheadrightarrow \oplus^n X\}}{p^{e \delta_R(M)}}
    \]
    exists by \cite{SmirnovTucker2}, we may compute it along any subsequence. 
    In particular, using the subsequence $e = n\omega + r$ we compute, 
    following the argument of \cite[Proposition~4.10]{SmirnovTucker2},
    that $\surjsig {M}{X} = \frac{1}{b_r} \asn (N_r \twoheadrightarrow X)$. 
    Taking $X = N_j$, 
    we obtain that $\asn (N_r \twoheadrightarrow N_j) = b_r/b_j$ for all $0 \leq r, j < \omega$.  
    By Theorem~\ref{thm: asn achieved}, we may find $k$ such that 
    there are surjections
    \[
        \oplus^{k b_r} N_0 \twoheadrightarrow
        \oplus^{kb_0} N_r \text{ and } 
        \oplus^{kb_0} N_r \twoheadrightarrow \oplus^{k b_r} N_0
    \] 
    for all $0 < r < \omega$. 
    The composition of the two surjections is a surjective endomorphism of $N_0^{\oplus k b_r}$, which must be an automorphism. Hence, $\oplus^{k b_r} N_0 \cong \oplus^{kb_0} N_r$. 
Thus $\{i \mid s_{i, 0} > 0\} = \{i \mid s_{i, r} > 0\}$. Choose $M_i$ such that $M$ and $M_i$ belong to different periodic subclasses of $\Mat(M)$; see the text following \Cref{de: period of an irreducible matrix}. Let $0<r< \omega$ be such that $\Fmult{e}{M}{M_i}>0$ implies $e \equiv r$ modulo $\omega$.
Since $\Mat(M)$ is positive recurrent $s_{i,r}>0$.
Our earlier observation implies $s_{i,0}>0$.
Thus $M_i$ and $M$ must be in the same periodic subclass, contradicting our $\omega>1$ assumption.
\end{proof}

\begin{proof}[Proof of \Cref{thm: Yao main}]
We argue that all the entries of the sequence of matrices $$(p^{-e\delta(M)}\Mat(M)^e)_e$$ converge. For $M_i, M_j \in \find(M)$, if $M_j$ is not an F-contributor to  $M$, then $(\Fmult{e}{M_i}{M_j}/p^{e\delta(M)})_e$ must converge to zero. When $M_j$ is an F-contributor to $M$, it must have period one by \Cref{prop: FFRT contributors have period one}. Choose $\p \in \Minh(M_j) \cap \Minh(M)$; see \Cref{lem: radius of convergence for contributors}. Thus $(p^{-e\delta(M)}\Fmult{e}{M_i}{M_j})_e$ must be convergent by \Cref{cor: aperiodicity implies existence of limit} applied at $\p$. 
The rationality follows from the existence of the limit by \cite[Lemma~2.4]{SeibertFFRT}.
For any $\p' \in \Minh(M)$,
\[p^{e\alpha(R_{\p'})}\grk{\p'}{M}= p^{e\delta(M)}\grk{\p'}{M}= \grk{\p'}{F^e_*M}= \sum_{N \in \find(\p',M)}\Fmult{e}{M}{N}\grk{\p'}{N}.\]
Now divide both the sides by $p^{e\delta(M)}$ and take limit. Since $\find(\p',M)$ is finite, the limit commutes with sum. Since $\grk{\p'}{M}>0$, there must be an $N \in \find(\p',M)$, such that $\lim_{e}p^{-e\delta(M)}\Fmult{e}{M}{N}$ is positive.
\end{proof}

\subsection{Classes of F-contributors}
We will now proceed to study F-contributors of modules that have FFRT. The following consequence of Yao's proof of Theorem~\ref{thm: Yao main} serves as our tool. We also derive consequences on the eigenvalues of the F-decomposition matrices of modules satisfying FFRT from here. Such a restriction on possible eigenvalues displays a way these F-decomposition matrices are special among all possible finite matrices with nonnegative entries.

\begin{corollary}\label{cor: Yao}
    Let $(R, \m)$ be an F-finite local ring and $M$ be a finitely generated $R$-module that satisfies FFRT. 
    Let $\find (M) = \{M_1, \ldots, M_\ell\}$.
    Then, for $1 \leq i,k \leq \ell$ there exist finite length $R$-modules $L_k$ and rational numbers $c_{i,k}$, such that for all $j \in \{1, \ldots, \ell\}$ and all $e \geq 1$ we have
    \[
    \Fmult{e}{M_j}{M_i} = \sum_{k = 1}^\ell c_{i, k} \length (L_{k} \otimes_R F_*^e M_j) = 
    \gsig{M_j}{M_i} p^{e \delta(M_j)} + O(p^{e(\delta(M_j)- 1)}).
    \]
    In particular, 
    if $\lambda$ is an eigenvalue of $\Mat(M)$ different from $p^{\delta(M)}$ then 
    $|\lambda| \leq p^{\delta(M) - 1}$ and any eigenvalue of $\Mat(M)$ with modulus $p^{\delta(M)-1}$ is semisimple.
\end{corollary}
\begin{proof}
Using Theorem~\ref{thm: Yao span}, for $1 \leq i,k \leq \ell$, choose finite length $R$-modules $L_k$ and rational numbers $c_{i,k}$, such that for $1 \leq i \leq \ell$, the standard basis vector
    \[
    e_i = \sum_{k = 1}^\ell c_{i, k} \begin{bmatrix}
        \length (M_1 \otimes_R L_{k}) & \cdots & \length (M_\ell \otimes_R L_{k}) 
    \end{bmatrix} 
    ^\tr.\]
 For any  $j \in \{1, \ldots, \ell\}$ and $e \geq 1$,  $\Fmult{e}{M_j}{M_i}= e_j^\tr \Mat(M)^e e_i$. Thus the first equality in the desired equation follows by plugging in the above expression for $e_i$. The $O(p^{e(\delta(M_j)-1)})$ term in the second equality is obtained from the first one by using:
    \[
    |p^{e \delta (M_j)}\ehk (L_{k}; M_j)  - \length (L_{k} \otimes F_*^e M_j) | = O (p^{e (\delta (M_j)-1)});\ \text{see \cite[Proposition~2(a)]{Seibert}}.\]
The first term of the second equality is already established in \Cref{thm: Yao main} or \Cref{prop: FFRT contributors have period one}.
The remaining desired conclusion on the eigenvalue of $\Mat(M)$ follows from \Cref{cor: matrix growth rate}.
\end{proof}

The following extends \cite[Corollary~3.9]{HashimotoNakajima}
from irreducible MCM modules over a quotient singularity 
to arbitrary FFRT modules and rings which are not domains. 

\begin{proposition}\label{prop: FFRT contributors in the same class}
    Let $(R, \m)$ be an F-finite local ring.
    Let $M, N$ be two finitely generated $R$-modules that satisfy FFRT and such that $\delta(M) = \delta(N)$.
    If $\length(M_\p) = \length (N_\p)$ for all $\p \in \Spec(R)$ such that $\alpha(R_\p) = \delta(M)$, 
    then $\gsig{M}{X} = \gsig{N}{X}$ for any indecomposable $R$-module $X$.
\end{proposition}
\begin{proof}
    Let $\overline{R} \coloneqq R/(\Ann(M) \cap \Ann (N))$ and note that $\delta_{\overline{R}}(\overline{R}) = \delta(M) = \delta(N)$ and $\alpha(\overline{R}) = \alpha(R)$. In the new ring, $\length(M_\p) = \length (N_\p)$ for all $\p \in \Minh(\overline{R})$. We now replace $R$ with $\overline{R}$.
    
    Let $\{M_i\}_{i = 1}^m = \find (M \oplus N)$. Without loss of generality, it suffices to show that $\gsig{M}{M_1} = \gsig{N}{M_1}$. 
    Taking limits in \Cref{cor: Yao}, we obtain $c_i \in \QQ$ and finite length $R$-modules $L_i$ such that
    \begin{align*}
        \gsig{M_j}{M_1} &= \sum_{i = 1}^m c_i \ehk (L_i; M_j)
    = \sum_{i = 1}^m c_i \sum_{\p \in \Minh(M_j)} \ehk (L_i; R/\p) \length ((M_j)_\p)
    \\ &= \sum_{\p \in \Minh(M_j)}\length ((M_j)_\p)  \sum_{i = 1}^m c_i \ehk (L_i; R/\p).
    \end{align*}

    Write $F_* M = \bigoplus_{j=1}^m M_j^{\oplus a_j}$ and take any $\q \in \Minh(R)$.
    Since $\Fmult{e+1}{M}{M_1}= \sum_{j=1}^m a_j\Fmult{e}{M_j}{M_1}$, it follows from the above that
\begin{equation*}
  \begin{split}
       \gsig{M}{M_1}&= \sum_{j=1}^ma_j \gsig{M_j}{M_1}_\q/p^{\delta_R(R)}\\
       &= \sum_{\p \in \Minh(R)} \frac{\sum_{j=1}^m a_j\length_{R_\p} ((M_j)_\p)}{p^{\delta_R(R)}}\sum_{i = 1}^m c_i \ehk (L_i; R/\p)\\
       &= \sum_{\p \in \Minh(R)} \length_{R_\p} (M_\p) \sum_{i = 1}^m c_i \ehk (L_i; R/\p), \,\, \text{by \Cref{lem: rank pushforward}}.
    \end{split}   
\end{equation*}
A similar formula for $\gsig{N}{M_1}$ shows that $\gsig{M}{M_1}= \gsig{N}{M_1}$.
\end{proof}

\begin{corollary}\label{cor: FFRT additivity}
    Let $(R, \m)$ be an F-finite local ring, $X$ be a finitely generated indecomposable $R$-module, 
    and suppose that $0 \to M_1 \to M_2 \to M_3 \to 0$
    is an exact sequence of finitely generated $R$-modules 
    satisfying FFRT. 
    Then $\gsig{M_2}{X} = \gsig{M_1 \oplus M_3}{X}$. 
    In particular, if $\dim (M_1) = \dim (M_3)$ then     
    $\gsig{M_2}{X} = \gsig{M_1}{X} + \gsig{M_3}{X}$.
\end{corollary}
\begin{proof}
    By localizing the exact sequence, we see that $M_2$ and $M_1 \oplus M_3$ satisfy the assumptions of Proposition~\ref{prop: FFRT contributors in the same class}.
\end{proof}

\begin{corollary}\label{cor: FFRT contributors are supported at a prime}
    Let $(R, \m)$ be an F-finite local ring and $M$ be an indecomposable F-contributor satisfying FFRT. Then $\Ann(M)= \p$ for some $\p \in \Minh(M)$. 
\end{corollary}
\begin{proof}
    By passing to $\overline{R} = R/\Ann(M)$, we may assume that $\dim (R) = \dim (M)$.
    If $N$ satisfies FFRT with $\find(N) = \{N_1, \ldots, N_h\}$ and $\p \in \Ass(N)$,
    then $\Hh_\p^0 (N)$ also satisfies FFRT by indecomposable summands of $\Hh_\p^0 (N_i)$. In particular, among these summands there is an F-contributor of dimension $\dim R/\p$ (Theorem~\ref{thm: Yao main}). This contributor is an $R/\p$-module (\Cref{prop: associated primes and splitting}).

    Now suppose that  $\{\p_1, \ldots, \p_n\} = \Minh(R) \cap \Supp(M)$.
    By the first paragraph, for each $1 \leq i \leq n$, 
    there is an indecomposable F-contributor $X_i$ such that $\p_i = \Ann (X_i)$. 
    Let $r_i = \dim_{\kappa (\p_i)} (X_i)_{\p_i}$ and set $r = \lcm \{r_1, \ldots, r_n\}$.
    Then $M^{\oplus r}$ and $\bigoplus_{i = 1}^n X_i^{\oplus \length (M_{\p_i}) r/r_i}$ satisfy the assumptions of Proposition~\ref{prop: FFRT contributors in the same class}.
    Thus, $M \in \find (X_i)$ for some $i$ by the Krull--Remak--Schmidt property. Since $\dim (M) = \dim (R)$ we must have $\Ann (M) = \p_i$.
\end{proof}

\begin{corollary}\label{cor: finiteness of FFRT modules}
    Let $(R, \m)$ be an F-finite local ring and $\p$ be a prime ideal of $R$. Let
    $\cS_\p$ be the set of isomorphism classes of indecomposable F-contributors $X$ such that $X$ satisfies FFRT, $\p \in \Supp (X)$, and $\dim (X) = \dim (R/\p)$.
    Then $\cS_\p$ is finite and if $M, N \in \cS_\p$ then $N \in \find(M)$.
\end{corollary}

\begin{proof}
    Let $M, N \in \cS_\p$. By \Cref{cor: FFRT contributors are supported at a prime}, $\Ann (M)= \Ann(N)= \p$, so we will work over the ring $R/\p$. Since $\length_{R_\p} (M^{\oplus \length_{R_\p}N_\p}) = \length_{R_\p} (N^{\oplus \length_{R_\p}M_\p})$, and $\gsig{N}{N}$ is positive, we get that $\gsig{M}{N}$ is positive by \Cref{prop: FFRT contributors in the same class}. Since $M$ satisfies FFRT, the rest follows.
\end{proof}

\begin{corollary}\label{cor: FFRT rank formulas}
    Let $(R, \m)$ be an F-finite local ring, $M$ be a finitely generated $R$-module satisfying FFRT, and $\p \in \Minh(M)$. Then the rank formula towards every $X \in \find(\p, M)$ at $\p$ and the total rank formula at $\p$ hold for $M$. 
\end{corollary}
\begin{proof}
   Without loss of generality, $X$ is an F-contributor to $M$ at $\p$. Let $\{\p_1, \ldots, \p_n\} = \Minh(M)$ with $\p_1 = \p$. We 
   follow the proof of \Cref{cor: FFRT contributors are supported at a prime} and apply 
    Proposition~\ref{prop: FFRT contributors in the same class}
    to $\bigoplus_{i = 1}^n X_i^{\oplus \length (M_{\p_i}) r/r_i}$ and $M^{\oplus r}$
    with $X = X_1$ (possible by \Cref{cor: FFRT contributors are supported at a prime}). 
    Since $\gsig{X_i}{X_1} = 0$ for $i > 1$, by \Cref{cor: FFRT contributors are supported at a prime} it follows that 
    \[
    \gsig{M}{X_1} = \length (M_{\p_1})/\length (X_{\p_1}) \gsig{X_1}{X_1}. 
    \]    
Recall that $R(\Mat(\p,M)) \geq p^{-\alpha(R_\p)}$ by \Cref{cor: convergence radius is at least the power of dimension}. Since there is an F-contributor to $M$ which has $\p$ as an associated prime by \Cref{thm: Yao main}, equality holds in the last inequality and the matrix $\Mat(\p,M)$ has a basic positive recurrent class. Hence, the total rank formula now holds by \Cref{prop: total rank formula vs rank formula}.
\end{proof}

\subsection{Modules of essentially finite Frobenius representation type}\label{sse: essentially FFRT}
The property of having finitely many F-contributors is a natural weakening 
of the FFRT property.

\begin{definition}
    We say that an $R$-module $M$ is of essentially finite Frobenius representation type if a) it has an F-contributor and b) every F-contributor to $M$ has the FFRT property. 
\end{definition}

\begin{proposition}\label{prop: equivalent definition EFFRT}
    The following are equivalent:
    \begin{enumerate}
        \item $M$ is of essentially FFRT; \label{part: M is eFFRT}
        \item The set of indecomposable F-contributors to $M$ is finite and nonempty; \label{part: M finitely F-contributors}
        \item $R(\Mat(M)) = p^{-\delta_R(M)}$, $\Mat(M)$ has a basic positive recurrent class, and every basic positive recurrent class of $\Mat(M)$ is finite. \label{part: M has finitely F-contributors spectral}
    \end{enumerate}    
\end{proposition}
\begin{proof}
    By \Cref{cor: F-contributor matrix is irreducible},  
    the set of indecomposable F-contributors to $M$ is also
    the union of $\find(X)$ where $X$ ranges through F-contributors to $M$, hence (\ref{part: M finitely F-contributors}) $\Rightarrow$ (\ref{part: M is eFFRT}). 
    The converse holds by Corollary~\ref{cor: finiteness of FFRT modules}, since any F-contributor to $M$ will belong to $\cS_\p$ for some $\p \in \Minh(M)$. 
    (\ref{part: M has finitely F-contributors spectral}) 
    is a direct translation of (\ref{part: M is eFFRT}) via 
    \Cref{thm: F-contributor at a prime via matrices} and \Cref{cor: F-contributor matrix is irreducible}.
\end{proof}

\begin{example}
Hashimoto--Singh \cite[Theorem~3.1, Remark~3.3]{HashimotoSingh} provide an example of an invariant subring which is of essentially FFRT (by \cite[Theorem~5.1]{HashimotoSymonds}) but does not have FFRT. 
\end{example}

\begin{remark}
We expect that a ring that has a faithful finitely generated module that satisfies the FFRT property (e.g., a regular overring) should be essentially FFRT. 
Since FFRT and essentially FFRT coincide for a strongly F-regular ring (the matrix is irreducible), 
this property would give a different explanation for  \cite[Proposition~5.1]{HashimotoSingh} that shows that strongly F-regular invariant subrings have FFRT.
\end{remark}

Our interest in this property is that it allows one to reduce the convergence of $\Mat(M)^n$ to the FFRT case, where we may use Yao's work.

\begin{theorem}\label{thm: F-signature for eFFRT}
Let $(R, \m)$ be an F-finite local ring and $M$ be a finitely generated $R$-module of essentially finite Frobenius representation type. Then $\lim_{n \to \infty} p^{-n\delta_R(M)}\Mat(M)^n$ exists.
\end{theorem}
\begin{proof}
 By \Cref{th: existence of signatures in general}, it suffices to show that for any F-contributor $X$ to $M$ $$\lim_{e \to \infty} p^{-e\delta_R(N)}\Fmult{e}{N}{X}$$ exists for any $N \in \find(M)$. Any such $X$ has FFRT and $X$ is an F-contributor by \Cref{cor: Cohen--Macaulayless of contributors}. Therefore, by \Cref{prop: FFRT contributors have period one}, $X$ has period one and the limit exists by \Cref{cor: aperiodicity implies existence of limit}.
\end{proof}

\section{Convergence rates}\label{subsec: convergence rates} For a strongly F-regular ring $R$ and a finitely generated $R$-module $M$, we have further information about the convergence rate of $\Fmult{e}{M}{R}/p^{e\delta_R(R)}$ to the F-signature:
\[
\Fmult{e}{M}{R}= \gsig{M}{R}p^{e \delta_R(R)}+ O(p^{e(\delta_R(R)-1)}); \ \text{\cite[Theorem 4.9]{deStefani_Polstra_Yao_globalizing_F_invariants}}.
\]
The main result of this section, namely \Cref{th: convergence rates}, is used to extend this known result and yield a convergence rate for $\Fmult{e}{M}{N}/p^{e\delta_R(R)}$, where $N$ is not necessarily $R$, towards $\gsig{M}{N}$; see \Cref{th: convergence rate in the sfr case}. For $\gsig{M}{R}$, the existing approach to convergence rate utilizes freeness of $R$ as a module over itself, and describes $\Fmult{e}{R}{R}$ as the colength of the Frobenius splitting ideals $I_e$.
Since there are no known analogs of $I_e$ in the more general setting, we proceed by relating the convergence rate of a sequence to analytic properties of the corresponding generating function. Similar ideas in the context of Hilbert--Kunz theory are pursued in \cite{AlapanThesis}; see Rmk 2.3.2, loc. cit. 

Our approach is parallel to geometric ergodicity of Markov chains (\cite{vere_Jones_geometric_ergodicity_in_denumerable_markov_chains}), although
we do not restrict ourselves to irreducible matrices unlike \cite{vere_Jones_geometric_ergodicity_in_denumerable_markov_chains}. This level of generality is needed for applications, for example in \Cref{th: convergence rate in the sfr case}.

In the two-dimensional case, we prove similar convergence rates by a more direct approach in \Cref{th: rank formula in dimension two}.

In this section, for an $r \geq 0$, $D_r$ denotes $\{z \in \mathbb C \mid |z|<r\}.$ 

\begin{lemma}\label{le: rate of convergence vs meromorphic extension}
   Let $(a_n)_n$ be a sequence of complex numbers and $\omega$ be a positive integer such that for any integer $r$, where $0 \leq r \leq \omega -1$, $(a_{e\omega+r})_e$ converges to $l_r \in \mathbb C$. Then, for any $0 <\lambda< 1$, the following two conditions are equivalent:
    
    \begin{enumerate}
     \item For any $\lambda< \tilde \lambda<1$,
     \[a_{e\omega+r}= l_r+ O(\tilde \lambda^{e\omega}),\ \text{for all} \ 0 \leq r \leq \omega-1. \]

     \item The generating function $\sum_{n=0}^{\infty}a_nz^n$ admits a meromorphic extension to $D_{1/\lambda}$, such that the only possible poles of this extension are simple poles at the $\omega$-th roots of unity.
    \end{enumerate}
\end{lemma}
\begin{proof}
    (1) $\implies$ (2): 
    Write $a_{e\omega+r}= l_r+ b_{e\omega+r}$. Then on the open unit disc, we have the following equality of holomorphic functions
    \[ \sum_{n=0}^{\infty}a_nz^n= \sum_{r=0}^{\omega-1}\frac{l_rz^r}{1-z^\omega}+ \sum_{n=0}^{\infty}b_nz^n.\]
    Since $b_n= O(\tilde \lambda^n)$, $\sum_{n=0}^{\infty}b_nz^n$ is convergent on $D_{1/\tilde\lambda}$ for any $\tilde \lambda> \lambda$. This proves (2).

    \noindent (2) $\implies$ (1): The holomorphic function $(1-z^{\omega})(\sum_{n=0}^{\infty}a_nz^n)$ on the open unit disc extends to a holomorphic function on $D_{1/\lambda}$. On the open unit disc
    \[(1-z^{\omega})(\sum_{n=0}^{\infty}a_nz^n)= \sum_{r=0}^{\omega-1}  z^r\left( a_r+  \sum_{e=1}^{\infty}(a_{r+e\omega}- a_{r+(e-1)\omega})z^{e\omega} \right). \]
    For any $\tilde \lambda> \lambda$, the power series on the right converges on the closure of $D_{1/\tilde \lambda}$. Therefore there exists $C_{\tilde \lambda}>0$ such that, 
    \[|a_{r+e\omega}- a_{r+(e-1)\omega}| \leq C_{\tilde \lambda} \tilde \lambda^{e \omega}, \text{for all}\ 0 \leq r \leq \omega-1, e \geq 1.\]
    Therefore for any $r$ and $e' \geq e$, 
     \[|a_{r+e'\omega}- a_{r+e\omega}| \leq \sum_{j=e+1}^{e'} |a_{r+j\omega}- a_{r+(j-1)\omega}|\leq C_{\tilde \lambda}(\frac{1-\tilde \lambda^{(e'-e)\omega}}{1-\tilde \lambda^\omega})\tilde \lambda^{(e+1)\omega}.
    \]
    Taking limit as $e'$ approaches infinity, we get (1).
\end{proof}

Next, we discuss some results on zeroes and poles of formal power series associated to nonnegative matrices, which are used later in \Cref{th: convergence rates}. Let $\cS$ be a countable set and $T= (t_{ss'})_{ss'}$ be an $\cS \times \cS$ matrix with nonnegative coefficients. While most of the discussion is valid as long as the matrix powers $T^n$ are well defined, we further assume for simplicity that for any $s \in \cS$, $t_{ss'}$ is nonzero for only finitely many $s'$. We freely refer to the notation and terminologies set in \Cref{sse: terminologies and notation for infinite matrices}, \Cref{ssse: limit theorems for reducible matrices} and \Cref{sse: F decomposition matrix}. For every $s, s' \in \cS,\ \text{a nonempty} \ H \subseteq \cS$, define ${}_Ht^{(n)}_{ss'}$ recursively on $n$, so that
\[_Ht^{(0)}_{ss'}=0;\ _Ht^{(1)}_{ss'}= t_{ss'}, \ _Ht^{(n)}_{ss'}= \sum_{l \notin H} t_{sl} {}_{H}t^{(n-1)}_{ls'}, \ \text{when} \ n \geq 2. \]
When $H= \{l\}$, we simply write $_{H}t^{(n)}_{ss'}= {}_{l}t^{(n)}_{ss'}$. Moreover, for any $s,s', l \in \cS$, denote 
\[_{s'}t^{(n)}_{ss'}= f^{(n)}_{ss'},\  _{\{s',l\}}t^{(n)}_{ss'}= {}_lf^{(n)}_{ss'}. \]
When $T$ is a stochastic matrix, in terms of the associated Markov chain (\Cref{ssse: limit theorems for reducible matrices}), $_{H}t^{(n)}_{ss'}$ is the probability that the state $s$ arrives at state $s'$ at the $n$-th stage while avoiding the states in $H$ except possibly at the beginning and end.
Denote the formal power series $\sum_{n=0}^\infty{}_Ht^{(n)}_{ss'}z^n$ by ${}_HT_{ss'}(z)$ and $\sum_{n=0}^\infty t^{(n)}_{ss'}z^n$ by $T_{ss'}(z)$. The symbols $F_{ss'}(z)$, ${}_lF_{ss'}(z)$ have analogous interpretations.  We skip the verification of the following identities of formal power series. When $T$ is stochastic, the relationships among the coefficients of the power series corresponding to the identities also follow from probabilistic interpretations of the coefficients.

\begin{proposition}\label{pr: identities of generating functions} \upshape{(\cite[page 16]{vere_Jones_geometric_ergodicity_in_denumerable_markov_chains})} Let $T,\cS$ be as before and $s,s',l$ be three distinct elements of $\cS$. 
    \begin{enumerate}
        \item $T_{ss}(z)(1-F_{ss}(z))=1$.
        \item $T_{ss'}(z)= F_{ss'}(z)T_{s's'}(z)$.
        \item $F_{ss}(z)= {}_{l}F_{ss}(z)+ {}_{s}F_{sl}(z)F_{ls}(z)$.
        \item $F_{ss'}(z)= {}_{l}F_{ss'}(z)+ {}_{s'}T_{sl}(z){}_{l}F_{ls'}(z)$.
        \item $T_{ss'}(z)= {}_lT_{ss'}(z)+ F_{sl}(z)T_{ls'}(z)$; see page 26, (iii) loc. cit.
    \end{enumerate}
\end{proposition}

The following input from complex analysis is used crucially in our results on convergence rates of suitably normalized versions of the sequences $(\Fmult{e}{M}{N})_e$.

\begin{theorem}[Vivanti--Pringsheim]\label{th: Abel and vivanti pringsheim} Let $\sum_{n=0}^{\infty}a_nz^n$ be a power series with nonnegative real coefficients and $\rho \in (0, \infty)$ be the radius of convergence. Then $\sum_{n=0}^{\infty}a_nz^n$ must have a singularity at $z=\rho$, i.e., the function $T(z)= \sum_{n=0}^{\infty}a_nz^n$ cannot admit a holomorphic extension in a neighborhood of $z= \rho$. 
In other words, if $T(z)$ is holomorphic on $D_r$ and admits a holomorphic extension in a neighborhood of $z= r$, then $\rho>r$. See \cite[Chapter 8, Section 5]{Remmert_Rainhold_Theory_of_complex_functions}. 
\end{theorem}

The next result from complex analysis is crucially used later in \Cref{th: convergence rates}. We include a proof for a more algebraically minded reader.

\begin{lemma}\label{le: meromorphicity vs holomorphicity}
    Let $\omega$ be a positive integer, $(f_{n\omega})_{n \omega}, (t_{n\omega})_{n \omega}$ be two sequences of nonnegative real numbers such that $f_0=0$. Set $T(z)=\sum_{n=0}^{\infty}t_{n\omega}z^{n\omega}$ and $F(z)= \sum_{n=0}^{\infty}f_{n\omega}z^{n\omega}$. Assume that $T(z)(1-F(z))=1$ as formal power series and $T(z)$ has radius of convergence $\rho>0$. Then,
    \begin{enumerate}
         \item $\rho= \sup\{r \in \mathbb R | 1-F(z)\ \text{is holomorphic on}\ D_r\ \text{and positive on}\ [0, r) \}$.
     
        \item $T(z)$ admits a meromorphic extension on a neighborhood of $z= \rho$ if and only if $F(z)$ admits a holomorphic extension to a neighborhood of $z= \rho$. These two equivalent conditions are in turn equivalent to having a meromorphic (resp. holomorphic) extension of $T(z)$ (resp. $F(z)$) on $D_R$ for some $R>\rho$.

        \item Neither $T(z)$ nor $1-F(z)$ has any zero in $D_{\rho}$. Suppose that $\gcd\{n \mid f_{n\omega} \neq 0\}= 1$. If any of the equivalent conditions in (2) is satisfied, the zero-set of $1-F(z)$ or pole-set of $T(z)$ on the closure $\overline D_{\rho}$ is $S_{\rho, \omega}= \{\rho e^{2\pi ij/\omega} |\ 0 \leq j \leq \omega-1\}$.

        \item If any of the equivalent conditions in (2) is satisfied, there exists $R> \rho$ such that $T(z)$ is meromorphic and zero-free on $D_R$. Moreover,\\ $\sup \{R>\rho|\ T(z) \ \text{admits a meromorphic extension to}\ D_R\ \text{without any zeroes}\}$\\
        = $\sup \{R>\rho|\ F(z) \ \text{admits a holomorphic extension to}\ D_R\}$.
    \end{enumerate}
\end{lemma}

\begin{proof}
(1) Since $\rho>0$, $T(z)(1-F(z))=1$, and $1-F(0)=1$ is nonzero, $F(z)$ has a positive radius of convergence and $1-F(z)$ does not have any zero in a neighborhood of the origin. Let $r_0$ be the supremum. For any $z \in D_{r_0}$, $|F(z)|\leq F(|z|)<1$. Therefore $T(z)= 1/(1-F(z))$ is holomorphic on $D_{r_0}$, and hence $r_0\leq \rho$. Suppose, for a contradiction, that $\rho>r_0$. Then $F(z)$ is meromorphic around $z=r_0$. Since $\lim_{x \to r_0-}F(x)$ is at most one, $F$ cannot have a pole at $z=r_0$. By \Cref{th: Abel and vivanti pringsheim}, there exists $r' >r_0$ such that $F(z)$ is holomorphic on $D_{r'}$. Moreover $F(r_0)\leq 1$. Since $T(z)$ is holomorphic around $z= r_0$, $F(r_0)<1$. Therefore $r_0$ cannot be the supremum, contradicting our hypothesis.  

(2) The validity of $T(z)(1-F(z))=1$ on $D_\rho$ implies the equivalence of having meromorphic extensions of $T(z)$ and $F(z)$ around $z= \rho$. Since $\lim_{x \to \rho-}F(x) \leq 1$, any meromorphic extension of $F(z)$ around $z= \rho$ must be holomorphic around $z= \rho$. 
    If $F(z)$ admits a holomorphic extension around $z= \rho$, by \Cref{th: Abel and vivanti pringsheim}, $F(z)$ defines a holomorphic function on $D_R$, for some $R>\rho$. Then $T(z)= 1/(1-F(z))$ is meromorphic on $D_R$. Conversely, if $T(z)$ is meromorphic on $D_R$, for some $R>\rho$, $T(z)$ is meromorphic around $z=\rho$ and hence $F(z)$ admits a holomorphic extension on an open disc around the origin which is strictly bigger than $D_\rho$, by our argument presented before. 
    
    (3) $F(z)$ and $T(z)$ are holomorphic on $D_{\rho}$ and $T(z)(1-F(z))=1$. Thus $T(z)$ or $1-F(z)$ cannot have zeroes in $D_\rho$. 
    Suppose that $F(z)$ (resp. $T(z)$) admits a holomorphic (resp. meromorphic) extension to $D_R$, for some $R>\rho$. It suffices to prove the assertion on zeroes of $1-F(z)$ since $T(z)(1-F(z))=1$ holds on $D_R$. By \Cref{th: Abel and vivanti pringsheim} $z= \rho$ must be a pole of $T(z)$ and hence $F(\rho)=1$. 
    For some $\theta \in [0, 2\pi)$, suppose $F(\rho e^{i\theta})=1$. Since $F(\rho e^{i\theta})= F(\rho)$, the real part of $F(\rho)- F(\rho e^{i\theta})$, namely $\sum_{n=0}^{\infty}f_{n\omega}\rho^{n\omega}(1- \cos (n\omega \theta))$ must be zero. Therefore, for any $n$ such that $f_{n\omega}$ is nonzero, $\cos{(n\omega \theta)}=1$ and hence $e^{i n\omega \theta}=1$.
    Since $\gcd\{n \mid f_{n\omega} \neq 0\}= 1$,  $e^{i\theta \omega}=1$. Hence, $\rho e^{i\theta} \in S_{\rho, \omega}$. We finish by observing that $1= F(\rho)= F(\rho e^{2\pi i j/\omega})$ for any $j \in \mathbb Z$.

    (4) Assume that any of the equivalent conditions of (2) is satisfied. Since $T(z)$ does not have any zeroes on $\overline D_\rho$, there exists an $R>\rho$ such that $T(z)$ is meromorphic and zero-free on $D_R$. The rest follows from the validity of the equation $T(z)(1-F(z))=1$ for meromorphic extensions of $T(z)$ and $F(z)$.
\end{proof}

The following result is a suitable adaptation of a result from \cite{vere_Jones_geometric_ergodicity_in_denumerable_markov_chains}, which was originally stated for irreducible, aperiodic $\mathbb N \times \mathbb N$ matrices with nonnegative coefficients. We relax the aperiodic hypothesis. We skip the proof since the argument is similar to that in \cite{vere_Jones_geometric_ergodicity_in_denumerable_markov_chains}.

\begin{theorem}\textup{(\cite[Theorem IV]{vere_Jones_geometric_ergodicity_in_denumerable_markov_chains})}\label{th: geometric ergodicity result of Vere Jones} Let $\cS$ be a countable set and $T$ be a positive recurrent $\cS \times \cS$ irreducible matrix with nonnegative entries with convergence radius $\rho>0$ and period $\omega< \infty$. 
Suppose for some $s \in \cS$, the generating function $T_{ss}(z)=\sum_{n=0}^{\infty}t^{(n)}_{ss}z^n$ admits a meromorphic extension on an open neighborhood of $z= \rho$. Then there exists $r>\rho$, such that for any $\alpha, \beta \in \cS$, the generating function $T_{\alpha\beta}(z)= \sum_{n=0}^{\infty}t^{(n)}_{\alpha\beta}z^n$ admits a meromorphic extension on $D_r$, so that the only possible poles are  simple poles on $S_{\rho, \omega}= \{\rho e^{2i\pi j/\omega}| \ 0 \leq j \leq \omega-1\}$.  

Moreover, let $r'> \rho$ be such that on $D_{r'}$, for some $s$, $T_{ss}(z)$ admits a meromorphic extension with poles only on $S_{\rho, \omega}$ and no zeroes. Then one can take $r$ in the previous paragraph to be $r'$.
\end{theorem}

\begin{remark}\label{re: existence of zero free extension}
Recall that for any $s \in \cS$, $T_{ss}(z)(1-F_{ss}(z))= 1$ (\Cref{pr: identities of generating functions}). If $T_{ss}(z)$ admits a meromorphic extension to $D_r$, for some $r>\rho$, then there exists $r'>\rho$ such that $T_{ss}(z)$ is zero-free and meromorphic on $D_{r'}$; see \Cref{le: meromorphicity vs holomorphicity}. 
\end{remark}

Combining the above-mentioned results with \Cref{le: rate of convergence vs meromorphic extension}, we deduce the next theorem about convergence rates of the F-splitting signature functions towards respective F-splitting signatures with phases (\Cref{de: generalized signature with a phase}). The next lemma directly follows from \Cref{de: probability of entering a state for an indecomposable module}, and shows that \Cref{le: meromorphicity vs holomorphicity} is applicable to the following theorem.

\begin{lemma}\label{le: period in terms of first entrance probability}
Let $N$ be an indecomposable $R$-module with period $\omega<\infty$. Then $\gcd\{n \mid f^{(n\omega)}_{NN} \neq 0\}= 1$.
\end{lemma}

\begin{theorem}\label{th: convergence rates} Let $\p$ be a prime ideal of $R$, $M$ be a finitely generated $R$-module such that $\p \in \Ass(M)$, and $N \in \find(\p,M)$. 
Suppose that $\Fmult{e}{N}{N}$ is nonzero for some $e >0$. Denote the period of $N$ by $\omega$. All the F-splitting signatures with period in the assertions and proof are taken at $\p$.

Suppose that there exists $\lambda< p^{\alpha(R_\p)}$ such that, 
\[\Fmult{e\omega}{N}{N}= \gsig{N}{N}_{0, \omega}p^{e \omega\alpha(R_\p)}+ O(\lambda^{e\omega}).\]
Then the following assertions hold.
\begin{enumerate}
     \item  $T_{NN}(z)$ admits a meromorphic extension to $D_{1/\lambda}$.

     \item There exists an $l<p^{\alpha(R_\p)}$ such that $T_{NN}(z)$ admits a meromorphic extension to $D_{1/l}$ with only possible poles on $S_{1/p^{\alpha (R_\p)}, \omega} \coloneqq \{e^{2\pi i j/\omega}/p^{\alpha (R_\p)}| \ 0 \leq j\leq \omega-1\}$ and no zeroes in $D_{1/l}$. 

     For the rest of the assertions define $\lambda_N$ to be the infimum of all such $l$.
     
    \item Take $N_1$, $N_2$ in the connected class of $\Mat(\p, M)$ containing $N$. Let $r \in \mathbb N$ be such that $\Fmult{e}{N_1}{N_2}$ nonzero implies $e\equiv r$ modulo $\omega$. Then for any $p^{\alpha(R_\p)}>\lambda_0> \lambda_N$,
\[\Fmult{e\omega+r}{N_1}{N_2}= \gsig{N_1}{N_2}_{r, \omega}p^{(e\omega+r)\alpha(R_\p)}+ O(\lambda_0^{e\omega+r}).\]

\item Let $N'\in \find(\p,N)$ be in the connected class of $\Mat(\p,M)$ containing $N$ and $r$ be such that $\Fmult{e}{N'}{N}$ nonzero implies $e$ is congruent to $r$ modulo $\omega$. Then for any $p^{\alpha(R_\p)}>\lambda_0> \lambda_N$,
\[\Fmult{e\omega+r}{N'}{N}= \gsig{N'}{N}_{r,\omega}p^{(e\omega+r)\alpha(R_\p)}+ O(\lambda_0^{e\omega+r}).
\]

\item \label{part: convergence rate of the parent module to anything in a class} Suppose that there exists $\theta< p^{\alpha(R_\p)}$ such that for any $0 \leq r \leq \omega-1$
\[\Fmult{e\omega+r}{M}{N}= \gsig{M}{N}_{r,\omega}p^{(e\omega+r)\alpha(R_\p)}+ O(\theta^{e\omega+r}).
\]
Consider the infimum of all such $\theta< p^{\alpha(R_\p)}$ satisfying the last equation; define $\lambda_{M,N}$ to be the maximum of this infimum and $\lambda_N$. Then $\lambda_{M,N}< p^{\alpha(R_\p)}$. Moreover,
for any $M' \in \find(\p,M)$, any $N'$ in the connected class of $\Mat(\p,M)$ containing $N$, any $0 \leq r \leq \omega-1$, and $\lambda_{M,N}<\lambda_0<p^{\alpha(R_\p)}$,
\[\Fmult{e\omega+r}{M'}{N'}= \gsig{M'}{N'}_{r,\omega}p^{(e\omega+r)\alpha(R_\p)}+ O(\lambda_0^{e\omega+r}).
\]
\end{enumerate}
\end{theorem}

\begin{proof}
Recall that  $T_{NN}(z)(1-F_{NN}(z))=1$ by \Cref{pr: identities of generating functions}. In what follows, we apply \Cref{le: meromorphicity vs holomorphicity} several times to analyze poles of meromorphic extensions of $T_{NN}(z)$. Thanks to \Cref{le: period in terms of first entrance probability},  \Cref{le: meromorphicity vs holomorphicity} is applicable with the value of $\omega$ as denoted in \Cref{le: meromorphicity vs holomorphicity} being the period of $N$.

Let $\Lambda$ be the irreducible block of $\Mat(\p, M)$ corresponding to the connected class of $\Mat(\p,M)$ containing $N$. Suppose $\gsig{N}{N}_{0,\omega}=0$.
Then, by our assumption, the convergence radius of $\Lambda$ is at least $1/\lambda> 1/p^{\alpha(R_\p)}$. We claim that the convergence radius of $\Lambda$, denoted by $\rho$ for now, is $1/\lambda_N$. Indeed, since $T_{NN}(z)(1-F_{NN}(z))= 1$, $D_\rho$ is the largest open disc on which $T_{NN}$ is meromorphic, zero-free with only possible pole on $S_{1/p^{\alpha(R_\p)}, \omega}$; see  \Cref{le: meromorphicity vs holomorphicity}. 
Thus (1), (2), (3) and (4) follow from the definition of convergence radius. By \Cref{th: existence of signatures in general}, $\gsig{M}{N}_{r,\omega}=0$. Clearly $\lambda_{M,N}< p^{\alpha(R_\p)}$ and for any $\lambda_{M,N}<\lambda_0< p^{\alpha(R_\p)}$, $\Fmult{e\omega+r}{M}{N}= O(\lambda_0^{e\omega+r})$. 
A straightforward argument shows that for any $N'$ in the connected class of $\Mat(\p,M)$ containing $N$, $\gsig{M}{N'}_{r,\omega}$ is zero and $\Fmult{e\omega+r}{M}{N'}= O(\lambda_0^{e\omega+r})$ for any $r$. 
Since $M' \in \find(\p,M)$, for any $r$, $\Fmult{e\omega+r}{M'}{N'}= O(\lambda_0^{e\omega+r})$.

From now on, assume $\gsig{N}{N}_{0, \omega}$ is nonzero. (3) implies (4). (1) follows from \Cref{le: rate of convergence vs meromorphic extension}.
 By our assumption $\Lambda$ is an irreducible positive recurrent matrix with period $\omega$ and convergence radius $1/p^{\alpha(R_\p)}$. 

(2) $T_{NN}(z)$ admits a meromorphic extension to $D_{1/\lambda}$, with poles on $S_{1/p^{\alpha(R_\p)}, \omega}$ by \Cref{le: rate of convergence vs meromorphic extension}. Now, the existence of $l$ follows from \Cref{le: meromorphicity vs holomorphicity}, (4). 

(3) By \Cref{th: geometric ergodicity result of Vere Jones}, applied to the irreducible matrix $\Lambda$, we conclude that on $D_{1/\lambda_N}$, $T_{N_1N_2}(z)$ is meromorphic with simple poles on $S_{1/p^{\alpha(R_\p)}, \omega}$. Now (3) follows from \Cref{le: rate of convergence vs meromorphic extension}.

(5) Define $T_{MN'}(z) \coloneqq \sum_{n=0}^{\infty} \Fmult{n}{M}{N'}z^n$, and $T_{M'N'}(z)$ similarly. Without loss of generality, we assume $N \in \find(\p,M')$. Below we prove that $T_{M'N'}(z)$ admits a meromorphic extension to $D_{1/\lambda_{M,N}}$ with only possible poles on $S_{1/p^{\alpha(R_\p)}, \omega}$.
(5) then follows by using \Cref{le: rate of convergence vs meromorphic extension}. To that end, suppose $e_0$ is such that $\Fmult{e_0}{M}{M'}$ is positive. Write
\[F^{e_0+1}_*M= \underset{N_i \in \find(M)}{\oplus}N_i^{\oplus b_i}, \ F_*M'= \underset{N_i \in \find(M)}{\oplus}N_i^{\oplus b'_i}.\]
Since $M'$ is a summand of $F^{e_0}_*M$, $b'_i \leq b_i$ for all $N_i \in \find(M)$. Since
\[T_{M'N'}(z)= \Fmult{0}{M'}{N'}+ \sum_{N_i \in \find( M)}b'_izT_{N_iN'}(z),\]
and $b'_i \leq b_i$ for all $N_i$, it suffices to prove the following,

\noindent \textbf{Claim:} If $b_i>0$, $T_{N_iN'}(z)$ admits a meromorphic extension to $D_{1/\lambda_{M,N}}$ with only possible poles on $S_{1/p^{\alpha(R_\p)}, \omega}$.

For this claim, suppose $N_{i_0} \cong N$, where $b_{i_0}$ could be zero. Note,
\begin{equation*}
\begin{split}
  & T_{MN}(z) = \sum_{n=0}^{e_0}\Fmult{n}{M}{N}z^n+ \sum_{N_i \in \find(M), i \neq i_0}b_iz^{e_0+1}T_{N_iN}(z) + b_{i_0}z^{e_0+1}T_{NN}(z) \\
  & = \sum_{n=0}^{e_0}\Fmult{n}{M}{N}z^n+ \sum_{i \neq i_0}b_iz^{e_0+1}F_{N_iN}(z)T_{NN}(z) + b_{i_0}z^{e_0+1}T_{NN}(z), (\Cref{pr: identities of generating functions}, (2)).
\end{split}  
\end{equation*}
Both $T_{MN}(z)$ and $T_{NN}(z)$ are meromorphic on $D_{1/\lambda_{M,N}}$ with possible simple poles on $S_{1/p^{\alpha(R_\p)}, \omega}$ (\Cref{le: rate of convergence vs meromorphic extension}) and $T_{NN}(z)$ is nonvanishing on $D_{1/\lambda_{M,N}}$. Since $\gsig{N}{N}_{0, \omega}$ is positive, $T_{NN}(z)$ has simple poles at every point of $S_{1/p^{\alpha(R_\p)}, \omega}$; see \Cref{le: meromorphicity vs holomorphicity}, (3). Therefore, the above equation implies holomorphicity of $\sum_{b_i>0, i \neq i_0}b_iz^{e_0+1}F_{N_iN}(z)+ b_{i_0}z^{e_0+1}$ on $D_{1/\lambda_{M,N}}$.
Being a finite sum, $\sum_{b_i>0, i\neq i_0}b_iF_{N_iN}(z)$ cannot have a pole at the origin. So, it follows that $\sum_{b_i>0, i \neq i_0}b_iF_{N_iN}(z)$ must be absolutely convergent on $D_{1/\lambda_{M,N}}$.
Since the coefficients of $F_{N_iN}(z)$ are nonnegative, $F_{N_iN}(z)$ must be absolutely convergent on $D_{1/\lambda_{M,N}}$ and hence holomorphic on $D_{1/\lambda_{M,N}}$, whenever $b_i>0$ and $i \neq i_0$. When $b_i>0$ and $i \neq i_0$, since $T_{N_iN}= F_{N_iN}T_{NN}$, $T_{N_iN}$ is also meromorphic on $D_{1/\lambda_{M,N}}$ with only possible poles on $S_{1/p^{\alpha(R_\p)}, \omega}$. 

Now we assume $N' \ncong N$, and argue that $T_{N_iN'}$ satisfies the claim. By part (3) and \Cref{le: rate of convergence vs meromorphic extension}, the claim holds when $N_i$ is isomorphic to $N'$ or $N$. Fix $N_i$ such that $b_i>0$, but $N_i \ncong N, N'$. By \Cref{pr: identities of generating functions}, (5),
$T_{N_iN'}(z)= {}_NT_{N_iN'}(z)+ F_{N_iN}(z)T_{NN'}(z)$. 
By part (3) and the holomorphicity of $F_{N_iN}(z)$ on $D_{1/\lambda_{M,N}}$, $F_{N_iN}(z)T_{NN'}(z)$ is meromorphic on $D_{1/\lambda_{M,N}}$ with only possible poles on $S_{1/p^{\alpha(R_\p)}, \omega}$.
So for the claim, it suffices to prove that ${}_NT_{N_iN'}(z)$ is holomorphic on $D_{1/\lambda_{M,N}}$. To this end, recall from \Cref{pr: identities of generating functions}, (4), that:
$$F_{N_iN}(z)= {}_{N'}F_{N_iN}(z)+ {}_NT_{N_iN'}(z){}_{N'}F_{N'N}(z).$$
All the power series involved in the above identity have nonnegative coefficients. Since the left side is absolutely convergent on $D_{1/\lambda_{M,N}}$, so are ${}_{N'}F_{N_iN}(z)$ and ${}_NT_{N_iN'}(z){}_{N'}F_{N'N}(z)$.
Since $T_{N'N}= F_{N'N}T_{NN}$ (\Cref{pr: identities of generating functions}, (2)), and a pole of $T_{N'N}$ on $D_{1/\lambda_{M,N}}$ is a pole of $T_{NN}$ with the same order by part (3) and \Cref{le: rate of convergence vs meromorphic extension}, and $T_{NN}(z)$ is zero-free on $D_{1/\lambda_{M,N}}$, $F_{N'N}(z)$ is holomorphic on $D_{1/\lambda_{M,N}}$.  
Hence, the power series ${}_{N'}F_{N'N}(z)$ defines a holomorphic function on $D_{1/\lambda_{M,N}}$ which does not vanish on $D_{1/\lambda_{M,N}}\cap \mathbb R_{>0}$. Therefore, the meromorphic function ${}_NT_{N_iN'}(z)$ on $D_{1/\lambda_{M,N}}$ does not have any poles on $D_{1/\lambda_{M,N}} \cap \mathbb R_{>0}$. By Vivanti--Pringsheim's theorem (\Cref{th: Abel and vivanti pringsheim}), this means ${}_NT_{N_iN'}(z)$ is holomorphic on $D_{1/\lambda_{M,N}}$.
This finishes the proof of the claim and hence the theorem.
\end{proof}

\section{Strongly F-regular modules}\label{se: strongly F regular modules}

We now introduce and develop strongly F-regular modules --- a module-theoretic generalization of the notion of strongly F-regular rings (\Cref{de: F split SFR rings}) introduced by Hochster--Huneke in \cite{Hochster_Huneke_tight_closure_and_strong_F_regularity}. In particular, we will show that in certain cases, indecomposable F-contributors are strongly F-regular and the sets of strongly F-regular modules and F-contributors sharing an associated prime coincide (\Cref{prop: consequence of existence of one strongly F regular contributor}).
These are used later to establish the existence of F-splitting signatures in certain concrete cases, such as when the underlying ring is strongly F-regular (\Cref{cor: F contributors are SFR when there is a module-finite SFR extension}), or for invariant rings of finite groups (\Cref{cor: invariant ring}), and very generally in dimension up to two (\Cref{se: dimension two}). The following definition of strong F-regularity is meaningful even when the KRS condition does not hold. We still keep the KRS assumption because it is utilized in some applications.

\begin{definition}\label{def: sFR for modules}
    Let $R$ be a ring of positive characteristic $p > 0$. A finitely generated $R$-module $N$ is said to be strongly F-regular, if for any $f \in R$, which is not contained in any minimal prime of $N$, there exist a positive integer $e$ and $\phi \in \Hom_R(N, F^e_*N)$, such that the composition $N \xrightarrow[]{\phi}F^e_*N \xrightarrow[]{F^e_*f}F^e_*N$ is a split injection.
\end{definition}

The name is justified by the following result obtained by a straightforward argument.

\begin{proposition}\label{prop: strongly F regular rings vs modules}
    Let $R$ be an F-finite ring. Then, $R$ is strongly F-regular as a module over itself if and only if $R$ is a strongly F-regular ring.
\end{proposition}

It should be noted that even a strongly F-regular ring may admit positive-dimensional strongly F-regular modules different from the ring; see \Cref{cor: F contributors are SFR when there is a module-finite SFR extension}. Moreover, even the existence of full-dimensional strongly F-regular modules does not imply strong F-regularity of the ambient ring; see \Cref{eg: SFR modules over non SFR rings}. Every ring admits strongly F-regular modules:

\begin{example}
    Let $\m$ be a maximal ideal of $R$. Then $R/\m$ is strongly F-regular.
\end{example}

\subsection{Properties of strongly F-regular modules}

Now, we record some properties of strongly F-regular modules. The following well-known lemma is used in the sequel several times.

\begin{lemma}\label{le: rank sequence}
Let $R$ be an F-finite ring, which does not necessarily satisfy the KRS condition. Let $M$ and $N$ be two finitely generated $R$-modules, such that for every $\p \in \Min(M)$, $\length_{R_\p} (M_\p) = \length_{R_\p} (N_\p)$. Then, there exist $e \in \mathbb N$ and 
$\phi \in \Hom_R(F^e_*M, F^e_*N)$, such that for any $\p \in \Min(M)$, the induced map $\phi_\p$ is an isomorphism. 

If additionally, $\Ann(M) \subseteq \Ann(N)$ and $\Ann(M)$ is a radical ideal, then $e$ can be taken to be zero in the last assertion.
\end{lemma}
\begin{proof}
  Our assumption implies $\Min(M) \subseteq \Min(N)$. When, in addition, $\Ann(M) \subseteq \Ann(N)$ and $\Ann(M)$ is a radical ideal, consider $M, N$ as modules over the reduced ring $R/\Ann(M)$ and use the proof of \cite[Lemma~3.3]{Tucker}.
  
  In the general situation, take $S$ to be the complement of the union of minimal primes of $M$. Then the $S^{-1}R$-modules $S^{-1}M$ and $S^{-1}N$ have the same set of minimal primes. Take $e$ large enough so that $\Ann (F_*^e S^{-1}M)$ and $\Ann (F_*^e S^{-1}N)$ are radical; then they coincide. Hence, there is an $S^{-1}R$-linear map $\psi\colon F_*^e S^{-1}M \rightarrow F_*^e S^{-1}N$, such that $\psi_\p$ is an isomorphism for every $\p \in \Min(M)$. Take $\phi \in \Hom_R(F^e_*M,F^e_*N)$ such that $\psi= \phi/s$, for some $s \in S$. Then $\phi$ serves our purpose. 
\end{proof}

\begin{proposition}\label{prop: SFR modules have no embedded primes}
Let $R$ be an F-finite ring and $M$ be a finitely generated strongly F-regular $R$-module. Then
\begin{enumerate}
    \item \label{part: no embedded primes} Any associated prime of $M$ is minimal in $\Supp(M)$ and $\Ann (M)$ is radical.

    \item \label{part: component supported at a prime is SFR} For any ideal $I$, $\Hh^0_{I}(M)$ is strongly F-regular. In particular, $I\Hh^0_I(M)=0$.

    \item \label{part: local cohomology at a minimal prime of SFR modules} For any minimal prime $\p$ of $M$,  the annihilator of every nonzero element of $\Hh^0_\p(M)$ is $\p$. In particular, the only associated prime of $\Hh^0_\p(M)$ is $\p$.
\end{enumerate}

\end{proposition}

\begin{proof}
(\ref{part: no embedded primes}) First, $\Ann (M)$ is radical by \Cref{prop: associated primes and splitting}.
Suppose that $Q$ is an embedded associated prime of $M$ and let $f \in Q$ be not in any minimal prime of $M$. 
By definition, there exist $e \geq 1$ and an $R$-linear map $\phi\colon M \rightarrow F^e_*M$, such that the composition $M \xrightarrow[]{\phi}F^e_*M \xrightarrow[]{F^e_*f}F^e_*M$ is a split injection, 
so it remains a split injection after applying $\Hh^0_Q (\bullet)$.
By \Cref{prop: associated primes and splitting}, $f\Hh^0_Q(M) = 0$. 
Thus $\Hh^0_Q (M) = 0$, a contradiction.

(\ref{part: component supported at a prime is SFR})
Let $\p_1, \dots, \p_n$ be the minimal primes of $M$. 
By (\ref{part: no embedded primes}), $\Min (\Hh^0_I(M)) \subseteq \Min (M)$. For any $f$ not in any minimal prime of $\Hh^0_I(M)$, by re-indexing we may assume that there exists $m \leq n$ such that $f \notin \p_i$ for $1 \leq i \leq m$ and $f \in \p_i$ for $m < i \leq n$. 
Choose $g \in \cap_{i=1}^m \p_i \setminus \cup_{i=m+1}^n \p_i$. 
Since $g$ is in every minimal prime of $\Hh^0_I(M)$, $g$ is in the radical of the annihilator of $\Hh^0_I(M)$.
Since $\Hh^0_I(M)$ splits from $F^{e_0}_*\Hh^0_I(M)$ for some $e_0 \geq 1$, $g\Hh^0_I(M)=0$ by \Cref{prop: associated primes and splitting}. Note that $f+g$ is not in any minimal prime of $M$. Therefore, there exist $e \geq 1$ and $\phi \in \Hom_R(M, F^e_*M)$ such that the composition $M \xrightarrow[]{\phi}F^e_*M \xrightarrow[]{F^e_*(f+g)}F^e_*M$ is a split injection. Hence, the composition 
\[\Hh^0_I(M) \xrightarrow[]{\Hh^0_I(\phi)}F^e_*\Hh^0_I(M) \xrightarrow[]{F^e_*f= F^e_*(f+g)}F^e_*\Hh^0_I(M)\]
is a split injection. Hence, $\Hh^0_I(M)$ is strongly F-regular, so $I\Hh^0_I(M)=0$ by \Cref{prop: associated primes and splitting}.

(\ref{part: local cohomology at a minimal prime of SFR modules})
For a nonzero element $x \in \Hh^0_\p(M)$, $\p x=0$ by (\ref{part: component supported at a prime is SFR}). Had $\Ann(Rx)$ properly contained  $\p$, a minimal prime over $\Ann(Rx)$ would  properly contain $\p$ and would be an associated prime of $Rx$ and hence of $M$. This is impossible by (\ref{part: no embedded primes}).
\end{proof}

\begin{lemma}\label{lemma: main strongly F-regular contributor}
        Let $R$ be an F-finite ring, 
        $X$ be a finitely generated strongly F-regular  $R$-module, and $M$ be a finitely generated $R$-module. 
        Suppose that there exist $a, b \in \NN$ and a map
        $\phi \colon X^{\oplus a} \to M^{\oplus b}$
        which becomes an isomorphism after localizing at any $\p \in \Min(X)$.
        Then there exist $e \geq 1$ and a split injection $X \rightarrow F^e_*X^{\oplus a}$ such that its composition with $F^e_*\phi$ is also a split injection. 
        
        In particular, if $X$ is indecomposable,
        $X$ is a direct summand of both $F_*^{e} M$ and $F_*^e X$.
    \end{lemma}
    \begin{proof}
        By our assumption, there exists an element $f \in \Ann (\coker(\phi))$, 
        such that $f$ is not contained in any minimal prime of $X$. Due to strong F-regularity of $X$, there are $e_0 \geq 1$ and $\theta \colon X \rightarrow F^{e_0}_*X$ such that the composition $X \xrightarrow{\theta} F^{e_0}_* X \xrightarrow{F^{e_0}_*f}F^{e_0}_*X$ splits. 

        Because $X$ is strongly F-regular, there is $e \geq 1$, a split injection
        $\iota \colon X \rightarrow F^e_*X^{\oplus a}$, and a split projection
        $\pi\colon F^e_*X^{\oplus a} \rightarrow X$ with $\pi \circ \iota= \Id_X$. By Proposition~\ref{prop: SFR modules have no embedded primes}, $\phi$ is an injection. 
        Applying the functor $\Hom_R(F^e_* \bullet, X)$ to the sequence $0 \rightarrow X^{\oplus a} \rightarrow M^{\oplus b} \rightarrow \coker(\phi) \rightarrow 0$, we get the exact sequence
        \[
        \Hom_R(F^e_*M^{\oplus b}, X) \longrightarrow \Hom_R(F^e_*X^{\oplus a},X) \longrightarrow  \Ext^1_R(F^e_*\coker(\phi), X).
        \]
        The image of $\pi$ in $\Ext^1_R(F^e_*\coker(\phi), X)$ is annihilated by $f$. This means that $f \circ \pi$, i.e., the composition
        $F^e_*X^{\oplus a} \xrightarrow{\pi}  X \xrightarrow{f} X$ lifts to an $R$-linear map $\tilde \pi\colon  F^e_*M^{\oplus b} \rightarrow X$. Hence, the composition
        \begin{equation*}
        X \xrightarrow{\iota} F^e_*X^{\oplus a} \xrightarrow{F^e_*\phi}  F^e_*M^{\oplus b} \xrightarrow{\tilde{\pi}} X
        \end{equation*}
        is the multiplication by $f$ map on $X$. Since the composition
         \begin{equation*}
        X \xrightarrow[]{\theta} F^{e_0}_*X \xrightarrow{F^{e_0}_*\iota} F^{e+e_0}_*X^{\oplus a} \xrightarrow{F^{e+e_0}_*\phi}  F^{e+e_0}_*M^{\oplus b} \xrightarrow{F^{e_0}_*\tilde{\pi}} F^{e_0}_*X
        \end{equation*}
        is a split injection, the maps $X \rightarrow F^{e+e_0}_*X^{\oplus a}$ and $X \rightarrow F^{e+e_0}_*M^{\oplus b}$ obtained from the last diagram are also split injections, proving the first part of the lemma. When $X$ is indecomposable,
        the second part follows from the first by the KRS property.      
    \end{proof}

\begin{theorem}\label{th: further properties of SFR modules}
    Let $R$ be an F-finite ring, $X$ be a finitely generated strongly F-regular $R$-module, and $M$ be a nonzero finitely generated $R$-module. The following properties hold:

    \begin{enumerate}
         \item \label{part: decompose into local cohomology}
         $X \cong \oplus_{\p \in \Min(X)} \Hh^0_\p(X)$. If $X$ is indecomposable, $X$ has a unique associated prime, namely its annihilator. 

         \item \label{part: SFR modules under Frobenius pushforward and direct sum} $F^e_*X$ is strongly F-regular for any $e \in \mathbb N$. 
         
         \item \label{part: SFR direct sum}
         Let $Y$ be a finitely generated strongly F-regular module such that $\Min(X \oplus Y)= \Min(X) \cup \Min(Y)$. Then $X \oplus Y$ is strongly F-regular.

         \item \label{part: positive signature}  
         If $\Min (X) \subseteq \Min(M)$ and for every $\p \in \Min(X)$, $\alpha(R_\p)>0$, then  $(\Fmult{e}{M}{X})_e$ is nonzero for all large enough $e$. Moreover, $(\Fmult{e}{M}{X})_e$ is unbounded.

          \item \label{part: SFR in the connected class} A direct summand of $F^e_*X$ is strongly F-regular for any $e \in \mathbb N$.   

          \item \label{part: aperiodicity of SFR contributors} The period of $X$ is one. In particular, when $X$ is indecomposable, $(\Fmult{e}{M}{X}/p^{e\delta_R(X)})_e$ converges. If $X$ is indecomposable and an F-contributor, then the limit of the sequence is positive if and only if $X \in \find (M)$. 

          \item \label{part: alternate definition of SFR modules}

         For any $g \in R \setminus \cup_{\p \in \Min(X)} \p$, there exists a positive integer $e_X$ such that for any $e \geq e_X$, there is $\phi \in \Hom_R(X, F^e_*X)$, such that the following composition is a split injection:
         \[X \xrightarrow[]{\phi}F^e_*X \xrightarrow[]{F^e_*g}F^e_*X.\]
      
    \end{enumerate}
\end{theorem}

\begin{proof}
(\ref{part: decompose into local cohomology}) Let $Y = \bigoplus_{\p \in \Min (X)} \Hh^0_{\p}(X)$. By \Cref{prop: SFR modules have no embedded primes}, $\Ann(X)$ is a radical ideal and hence $\Ann(X)= \Ann(Y)$. Moreover, $\length_\p X_\p= \length_\p Y_\p$, for any minimal prime of $X$. Thus,
by \Cref{le: rank sequence}, there exists an $R$-linear map $\phi \colon X \to Y$, which becomes an isomorphism after localizing at any $\p \in \Min (X)$. Applying Lemma~\ref{lemma: main strongly F-regular contributor} to $\phi$, we conclude that $X$ is a direct summand of $F_*^e Y$ for some $e \geq 1$.
Therefore, by the KRS property, any indecomposable summand $X'$ of $X$ must be isomorphic to an indecomposable summand of $F^e_*\Hh^0_{\p}(X)$ for some $\p \in \Min(X)$. 

Now, let $\p_1, \ldots, \p_m$ be the minimal primes of $X$ and 
suppose that $x_i \in \Hh^0_{\p_i}(X)$ satisfy
$\sum_{i = 1}^m x_i = 0$. For any $i$, take $f_i \in \cap_{j \neq i} \p_j \setminus \p_i$. 
Then $f_i x_i = 0$ by \Cref{prop: SFR modules have no embedded primes}(\ref{part: component supported at a prime is SFR}), so $x_i = 0$ by \Cref{prop: SFR modules have no embedded primes}(\ref{part: local cohomology at a minimal prime of SFR modules}).
It follows that $\oplus_{\p \in \Min(X)}\Hh^0_\p(X) = Y \subseteq X$.  Since every indecomposable summand of $X$ is contained in $\Hh^0_\p(X)$ for some minimal prime $\p$ of $X$, the containment $Y \subseteq X$ cannot be proper.

If $X$ is indecomposable, $X$ must have exactly one minimal prime. Therefore, \Cref{prop: SFR modules have no embedded primes} implies that $\Ann(X)$ is a prime ideal and the only associated prime of $X$.

(\ref{part: SFR modules under Frobenius pushforward and direct sum}) and (\ref{part: SFR direct sum}) are straightforward.

(\ref{part: positive signature}) We prove a preparatory lemma.

\begin{lemma}\label{le: prep indecomposable summands of SFR}
    Let $Y$ be an indecomposable summand of $X$. Then
    $Y$ has period $1$. Moreover, $\Fmult{e}{M}{Y}$ is nonzero for all $e$ large enough. If additionally $\delta_R(Y)>0$,  then $(\Fmult{e}{M}{Y})_e$ is unbounded.
\end{lemma}

\begin{proof}[Proof of lemma:]  By \Cref{th: further properties of SFR modules}(\ref{part: decompose into local cohomology}), there exists $\p \in \Min(X)$ such that $Y$ is isomorphic to a direct summand of $\Hh^0_\p(X)$. By replacing $X$ by the strongly F-regular module $\Hh^0_\p(X)$, we assume that $X$ has exactly one minimal prime $\p$, which is its annihilator.

Choose positive integers $a,b,c$ such that the lengths of $X_\p^{\oplus a}, Y^{\oplus b}_\p$ and $F_{*}Y^{\oplus c}_\p$ are the same. Using \Cref{prop: SFR modules have no embedded primes} and \Cref{le: rank sequence}, we find $R$-linear maps
\[
\psi\colon X^{\oplus a} \rightarrow Y^{\oplus b}, \ \widetilde \psi\colon Y^{\oplus b} \rightarrow F_*Y^{\oplus c},\]
such that the induced maps $\psi_\p, \widetilde \psi_\p$ after localization at $\p$ are isomorphisms. By \Cref{lemma: main strongly F-regular contributor}, there exists an $e \geq 1$ and $\phi \in \Hom_R(X, F^e_*X^{\oplus a})$ such that the composition
\[X \xrightarrow[]{\phi}F^e_*X^{\oplus a} \xrightarrow[]{F^e_*\psi}F^e_*Y^{\oplus b}\xrightarrow[]{F^e_*\widetilde \psi}F^{e+1}_*Y^{\oplus c}\]
is a split injection. The composition of $Y \rightarrow X$ with the above chain gives split injections $Y \rightarrow  F^e_*Y^{\oplus b}$ and $Y \rightarrow  F^{e+1}_*Y^{\oplus c}$. Since $Y$ is indecomposable, the KRS property implies that $Y$ is a summand of both $F^e_*Y$ and $F^{e+1}_*Y$. Thus $Y$ has period one. Hence $\Fmult{e}{Y}{Y}>0$ for all $e$ large enough (\Cref{le: period one and splitting from large enough pushforward}).

Since $\p$ is the only minimal prime of $X$ and $\p \in \Min(M)$, there is an $e_1 \geq 1$ and a map $F^{e_1}_*X^{\oplus a'} \rightarrow F^{e_1}_*M^{\oplus b'}$ whose localization at $\p$ is an isomorphism by \Cref{le: rank sequence}. Since $F^{e_1}_*X$ is strongly F-regular by (\ref{part: SFR modules under Frobenius pushforward and direct sum}), by \Cref{lemma: main strongly F-regular contributor} $F^{e_1}_*X$ is a summand of $F^{e_0}_* M^{\oplus b'}$ for some $e_0 \geq 1$. By enlarging $e_0$, we can assume that $X$ is a summand of $F^{e_0}_* M^{\oplus b'}$ for some $e_0 \geq 1$. Composing with a split injection $Y \rightarrow X$ and by the KRS property, we conclude that $Y$ is a summand of $F^{e_0}_* M$. Since $Y$ has period one, $\Fmult{e}{M}{Y}>0$ for all $e$ large enough.

Assume $\delta(Y)>0$. Since $Y \in \find(M)$, it suffices to prove the unboundedness of $(\Fmult{e}{Y}{Y})_e$. We argue by contradiction. Suppose that the maximum value of this sequence is $\Fmult{e_0}{Y}{Y}$. 
Using the first part of this lemma, choose $e'$ such that $\Fmult{e}{Y}{Y}$ is positive for all $e \geq e'$ (\Cref{le: period one and splitting from large enough pushforward}).
Thus for $e \geq e'$, $\Fmult{e_0}{Y}{Y}\leq \Fmult{e_0+e}{Y}{Y}\leq \Fmult{e_0}{Y}{Y}$. For $e \geq e'$, write
\[F^{e_0+e}_*Y \cong Y^{\Fmult{e_0+e}{Y}{Y}} \oplus N_{e_0+e} \cong Y^{\Fmult{e_0}{Y}{Y}} \oplus N_{e_0+e}.\]
Since $0 < \alpha(R_\p)= \delta_R(Y)$, $(\grk{\p}{F^{e_0+e}_*Y})_e$ is unbounded. Hence, $N_{e_0+e}$ must be nonzero for some $e \geq e'$, say for $e= e_1$.
Since $N_{e_0+e_1}$ is a direct summand of $F_*^{e_0 + e_1}Y$, $\p$ is a minimal prime of $N_{e_0+e_1}$. Specializing the above argument to $M= N_{e_0+e_1}$, we conclude that $\Fmult{e}{N_{e_0+e_1}}{Y}$ is nonzero for infinitely many $e$.
Choose $e_2 \geq e'$ such that $\Fmult{e_2}{N_{e_0+e_1}}{Y}$ is nonzero. Thus
\[
\Fmult{e_0+e_1+e_2}{Y}{Y}>\Fmult{e_0}{Y}{Y},
\]
contradicting the boundedness assumption.
\end{proof}

We now prove (\ref{part: positive signature}). 
Suppose $X \cong \oplus_{i=1}^n X_i^{\oplus a_i}$, where $X_i$'s are indecomposable, mutually nonisomorphic, and $a_i>0$ for all $i$. By our assumption, $\delta(X_i)>0$ for all $i$. Fix a positive integer $m$. 
Using \Cref{le: prep indecomposable summands of SFR}, 
we may choose $e_i, f_i$ such that 
$X_i^{\oplus m a_i}$ is a summand of $F^{e_i}_* M$ 
and $\Fmult{e}{X_i}{X_i} \geq 1$ for all $e \geq f_i$.
It follows that for all $e \geq e_M \coloneqq \max_i \{e_i + f_i\}$ and for all $i$, $X_i^{\oplus m a_i}$ is a summand of $F^e_* M$.
By the KRS property, for $i \neq j$, $X_i$ cannot be a summand of $X_j^{\oplus b}$ for any $b \geq 1$. Thus, for $e \geq e_M$, $X^{\oplus m} \cong \oplus X_i^{\oplus ma_i}$ is a summand of $F^e_*M$, as required.

(\ref{part: SFR in the connected class}) Since $X$ has no embedded associated prime (\Cref{prop: SFR modules have no embedded primes}),
by (\ref{part: SFR modules under Frobenius pushforward and direct sum}) and (\ref{part: SFR direct sum}), it suffices to show that any indecomposable direct summand $Y$ of $X$ is strongly F-regular. By (\ref{part: decompose into local cohomology}) and \Cref{prop: SFR modules have no embedded primes}, we may also assume that $X = \Hh^0_\p(X)$ for some $\p$.
If $\delta(X)=\alpha(R_\p)=0$, $\p$ must be a maximal ideal. Since $\p Y=0$ by \Cref{prop: SFR modules have no embedded primes}, and $R/\p$ is a field, $Y$ must be isomorphic to a direct sum of copies of $R/\p$. By its indecomposability, $Y \cong R/\p$, and hence, $Y$ is strongly F-regular.

Therefore, we assume $\delta_R(X)= \delta_R(Y)= \alpha(R_\p)$ is positive from now on.
Using (\ref{part: positive signature}), choose $e_0$ such that $X$ is a direct summand of $F^{e_0}_*Y$. Let $s_Y, \pi_Y$ be a split injection and a split projection, realizing the splitting of $X$ from $F^{e_0}_*Y$.
For any $f \notin \p$, choose a positive integer $e_1$ and $\phi \in \Hom_R(X, F^{e_1}_*X)$ such that the composition $X \xrightarrow[]{\phi} F^{e_1}_*X \xrightarrow[]{F^{e_1}_*f}F^{e_1}_*X$ is a split injection. Let $\pi$ be a split projection corresponding to this split injection. Hence, the following composition is a split injection,
\[X \xrightarrow[]{\phi} F^{e_1}_*X \xrightarrow[]{F^{e_1}_*s_Y} F^{e_1+e_0}_*Y \xrightarrow[]{F_*^{e_1+e_0}f^{p^{e_0}}}F^{e_1+e_0}_*Y,\]
with a corresponding split projection $\pi \circ F^{e_1}_*\pi_Y$. Composing the above chain on the left by a split injection $j\colon Y \rightarrow X$, we conclude that the following resulting composition is a split injection
\[Y \xrightarrow[]{F^{e_1}_*s_Y \circ \phi \circ j} F^{e_1+e_0}_*Y \xrightarrow[]{F_*^{e_1+e_0}f^{p^{e_0}}= (F_*^{e_1+e_0}f)^{\circ p^{e_0}}}F^{e_1+e_0}_*Y.\]
From here, we deduce that the composition
\[Y \xrightarrow[]{F^{e_1}_*s_Y \circ \phi \circ j} F^{e_1+e_0}_*Y \xrightarrow[]{ F_*^{e_1+e_0}f}F^{e_1+e_0}_*Y,\]
is a split injection and thus $Y$ is strongly F-regular.

(\ref{part: aperiodicity of SFR contributors}) Write $X \cong \oplus_{i=1}^n X^{\oplus a_i}_i$, where $X_i$ are indecomposable and mutually nonisomorphic and $a_i>0$ for all $i$. If $\delta(X_i)>0$ for all $i$, part (\ref{part: positive signature}) implies that $X$ has period one. 
Now suppose $\delta(X_i)=0$ for some $i$. By relabeling, choose $1 \leq m\leq n $ such that $\delta(X_i)=0$ for $1 \leq i \leq m$ and $\delta(X_i)>0$ otherwise. 
Hence, for $1 \leq i \leq m$, $X_i \cong R/\m_i$ for maximal ideals $\m_i$, such that $F_* (R/\m_i) \cong R/\m_i$ as $R$-modules. Hence, for $1 \leq i \leq m$, $X^{\oplus a_i}_i$ is a summand of $F^e_*X^{\oplus a_i}_i$ and thus $F^e_*X$ for all $e \geq 1$.
By part (\ref{part: positive signature}) applied to the strongly F-regular module $X_i^{\oplus a_i}$ (part (\ref{part: SFR in the connected class})), for all $i>m$, $X_i^{\oplus a_i}$ is a summand of $F^e_*X$ for all $e$ large enough. 
Thus by KRS, $X$ is a summand of $F^e_*X$ for all $e$ large enough. This proves that $X$ has period one. 
The rest follows easily from \Cref{cor: aperiodicity implies existence of limit} and \Cref{thm: F-contributor at a prime via matrices} applied to a prime ideal $\p \in \Min(X)$, such that $\alpha(R_\p)= \delta_R(X)$. 

(\ref{part: alternate definition of SFR modules}) Since $X$ has period one, $X$ is a summand of $F^e_*X$ for all large enough $e$. Thus the desired conclusion follows. 
\end{proof}

\begin{remark}
In regards to \Cref{th: further properties of SFR modules}(\ref{part: SFR direct sum}), note that in general, a direct sum of two strongly F-regular modules need not be strongly F-regular. For example, take a strongly F-regular local ring $(R,\m, k)$ of positive dimension.
The residue field is obviously strongly F-regular, but, by \Cref{prop: SFR modules have no embedded primes}, $R \oplus k$ cannot be strongly F-regular.

In \Cref{th: further properties of SFR modules}(\ref{part: positive signature}), if $\alpha(R_\p)=0$ for some $\p \in \Min(X)$, $X$ might not be in $\find(M)$. For example, take $R$ to be a perfect field and $X= R^{\oplus 2}$. Then $X$ is not a summand of $F^e_*R$ for any $e \geq 1$.
\end{remark}

\begin{theorem}[Universality of strongly F-regular modules]\label{th: universality of SFR modules}
    Let $R$ be an F-finite ring and $X$ be an indecomposable strongly F-regular module. If the prime ideal $\Ann(X)$ is a minimal prime of a finitely generated module $M$, then $X$ is a direct summand of $F^e_*M$ for some $e \geq 1$.
\end{theorem}

\begin{proof}
Let $\p= \Ann(X)$. If  $\delta(X)= \alpha(R_\p)$ is zero, then $R/\p$ is a field and $X \cong R/\p$; see the proof of \Cref{th: further properties of SFR modules}(\ref{part: SFR in the connected class}). 
The desired result follows from \Cref{le: dimension zero contributor}.
If $\delta(X)>0$, \Cref{th: further properties of SFR modules}(\ref{part: positive signature}) suffices.
\end{proof}

The following corollary is immediate from above and \Cref{prop: SFR modules have no embedded primes}.  

\begin{corollary}\label{cor: additive generators of strongly F regular modules}
    Let $R$ be any F-finite ring and $N$ be an indecomposable strongly F-regular module with $\p \in \Ass(N)$. Then, the set of indecomposable strongly F-regular modules which have $\p$ as an associated prime is precisely $\find(N)$.
\end{corollary}

We record some consequences of the existence of a strongly F-regular module and a contributor sharing a minimal prime.  

\begin{proposition}\label{prop: consequence of existence of one strongly F regular contributor}
For a prime ideal $\p$ of $R$, suppose that there is an indecomposable strongly F-regular $R$-module $M$ such that $\p \in \Ass(M)$. Then 
\begin{enumerate}
    \item \label{part: existence of SFR modules gives one recurrent class module version} Let $X$ be a finitely generated indecomposable $R$-module such that $\p \in \Min(X)$ and $\sum_{e=0}^{\infty}\Fmult{e}{X}{X}/p^{e\alpha(R_\p)}$ is divergent. Then, $X \in \find(M)$ and $M \in \find(\p,X)$. In particular, $X$ is strongly F-regular and has period one. 

    \item \label{part: existence of SFR modules forces SFR contributors} Let $X$ be a finitely generated indecomposable F-contributor and $\p \in \Ass(X)$. Then, $\find(\p,X)= \find(M)$. In particular,  $M$ is an F-contributor.
\end{enumerate}
\end{proposition}

\begin{proof}
(\ref{part: existence of SFR modules gives one recurrent class module version})
Since $\p$ is a minimal prime of $X$, $M \in \find(\p, X)$, thanks to \Cref{th: universality of SFR modules}. By assumption 
$R(\Mat (\p, X)) = p^{-\alpha(R_\p)}$ and the connected class $\cC$ of $\Mat(\p,X)$ containing $X$ is $p^{-\alpha(R_\p)}$-recurrent. 
Thus, $\cC$ is final by \Cref{prop: properties of maximal recurrent classes for a prime}, so $M \in \cC$. Therefore, $\cC = \find(\p,M) = \find(M)$ by \Cref{prop: SFR modules have no embedded primes}. The rest follows from 
\Cref{th: further properties of SFR modules}.

(\ref{part: existence of SFR modules forces SFR contributors}) 
Since $X$ is an F-contributor, $\delta(X)= \alpha(R_\p)$ by \Cref{lem: radius of convergence for contributors}. Thus $\p \in \Min(X)$.
By part (\ref{part: existence of SFR modules gives one recurrent class module version}) and irreducibility of $\Mat(\p, X)$ (\Cref{thm: F-contributor at a prime via matrices}), $\find(\p,X)= \find(M)$.
Hence, $X \in \find(M)$. Thus, $X$ is strongly F-regular by \Cref{th: further properties of SFR modules}. Since $X$ is an F-contributor at $\p$, $M \in \find(\p, X)$ is also an F-contributor at $\p$ by \Cref{thm: F-contributor at a prime via matrices}.
Since $\p$ is the unique minimal prime of $M$, $M$ is an F-contributor.
\end{proof}

\subsection{Strongly F-regular rings}

We start by studying module-finite subrings of strongly F-regular rings via the following consequence of \Cref{prop: consequence of existence of one strongly F regular contributor}. 

\begin{corollary}
\label{cor: F contributors are SFR when there is a module-finite SFR extension}
Let $(R, \mathfrak m)$ be an F-finite local ring which admits a module-finite extension of local rings $\iota \colon (R, \m) \subseteq (S, \mathfrak n)$, such that $S$ is strongly F-regular. Let $M, N$ be finitely generated $R$-modules.
Then,
\begin{enumerate}
    \item \label{part: contributors and SFR match if there is a mod finite SFR extension} the set $\cS$ of indecomposable $R$-modules which are F-contributors of dimension  $\dim(R)$ coincides with the set of indecomposable strongly F-regular modules of dimension $\dim(R)$, and is the same as $\find (\iota_* S)$ (i.e., $R$-module summands). 

   \item \label{part: splitting signatures exist if there is a mod finite SFR extension} If $N$ is indecomposable, then the sequence $e \mapsto \Fmult{e}{M}{N}/p^{e\delta_R(R)}$ is convergent. 

    \item \label{part: positivity when there is a mod finite SFR extension} If $N$ is indecomposable, 
    the limit of the above sequence is positive if and only if $M$ is not a torsion module and $N \in \cS$.
\end{enumerate}
\end{corollary}

\begin{proof}
Recall that a strongly F-regular ring is a domain. Hence, any $R$-module of dimension $\dim(R)$ has the zero ideal as a minimal prime. We also note that $\delta_R(R)= \delta_S(S)$, e.g., by using that $[\kappa(\m): \kappa(\m)^p]= [\kappa(\mathfrak n): \kappa(\mathfrak n)^p]$.

(1) We prove containments between the sets. Let $\cS_1$
be the set of indecomposable strongly F-regular modules $L$ with $\dim (L) = \dim (R)$. Since $S$ is a strongly F-regular ring, it is strongly F-regular as an $R$-module. Therefore, $\find (\iota_* S) \subseteq \cS_1$ by \Cref{th: further properties of SFR modules}. 

Since $S$ is a strongly F-regular ring, 
$\liminf_{e} \Fmult{e}{S}{S}/p^{e\delta_S(S)}$ is positive, where $\Fmult{e}{S}{S}$ is computed by considering $S$ as an $S$-module; see \cite{AberbachLeuschke}.
Thus $\iota_* S$ is an F-contributor. 
Therefore, any indecomposable $R$-summand $L$ of $S$ is a strongly F-regular F-contributor by \Cref{th: existence of signatures in general}, applied at the prime ideal zero, and \Cref{cor: additive generators of strongly F regular modules}. 
Since $\find(\iota_* S) \subseteq \cS_1$, there is an indecomposable strongly F-regular F-contributor $L$ of dimension $\dim(R)$. 
By Corollary~\ref{cor: additive generators of strongly F regular modules}, $\cS_1 = \find (L)$, so $\cS_1 \subseteq \cS$.  
Furthermore,  $\cS \subseteq \find (L) \subseteq \find (\iota_* S)$ by \Cref{prop: consequence of existence of one strongly F regular contributor} and (\ref{part: contributors and SFR match if there is a mod finite SFR extension}) follows.

(2) and (3): The sequence in (2) converges to zero unless $N$ is an F-contributor of dimension $\dim(R)$ by \Cref{th: existence of signatures in general}.  By (1), indecomposable F-contributors of dimension $\dim(R)$ are strongly F-regular. In particular, these F-contributors have period one; hence the convergence follows, proving (2); see \Cref{th: further properties of SFR modules}.
For (3), note that  if $M$ is a torsion module, $\Fmult{e}{M}{N}=0$ for all $e \in \mathbb N$ and any $N \in \cS$. If $M$ is not a torsion module, the zero prime ideal is in $\Min(M)$; thus $\cS \subseteq \find(M)$ by \Cref{th: universality of SFR modules}. 
Since modules in $\cS $ are F-contributors of dimension $\dim(R)$, the positivity of the limit for any nontorsion $M$ and $N \in \cS$ follows by \Cref{th: further properties of SFR modules}(\ref{part: positive signature}) and \Cref{cor: positive ces signature -> F-contributor}.
\end{proof}

The above corollary implies Tucker's celebrated result confirming the existence of the F-signature. 

\begin{corollary}[Tucker, \cite{Tucker}]
Let $(S, \mathfrak n)$ be an F-finite local ring, where the category of finitely generated $S$-modules does not necessarily satisfy the KRS condition. Then the limit $\lim_e \Fmult{e}{ S}{S}/p^{e\delta_S(S)}$ defining the F-signature of $S$ exists.
\end{corollary}

\begin{proof}
    By \cite{AberbachLeuschke}, $\limsup_e\Fmult{e}{ S}{S}/p^{e\delta_S(S)}$ is positive if and only if $S$ is strongly F-regular. So, we assume that $S$ is strongly F-regular without loss of generality. Using \Cref{le: passing to completion}, we can assume that $S$ is complete. Then, convergence of $(\Fmult{e}{S}{S}/p^{e\delta_S(S)})_e$ follows from \Cref{cor: F contributors are SFR when there is a module-finite SFR extension}(\ref{part: splitting signatures exist if there is a mod finite SFR extension}).
\end{proof}

\begin{example}\label{eg: SFR modules over non SFR rings}
There are rings that are not strongly F-regular but admit module-finite strongly F-regular extensions; see \cite[Example 5.2]{HashimotoSingh}. These provide an example of a ring that is not strongly F-regular but admits strongly F-regular modules of positive dimension by \Cref{cor: F contributors are SFR when there is a module-finite SFR extension}.
\end{example}

The following result rephrases \Cref{cor: F contributors are SFR when there is a module-finite SFR extension} in terms of the F-decomposition matrices. 

\begin{corollary}\label{cor: unique maximal dimensional connected class}
    Let $R$ and $S$ be the same as in \Cref{cor: F contributors are SFR when there is a module-finite SFR extension}, and $M$ be a finitely generated $R$-module. If $M$ is not torsion, then $\Mat(M)$ has a unique $p^{-\delta_R(R)}$-recurrent connected class, namely $\find (\iota_* S)$. Otherwise, $\Mat(M)$ has no such connected class. 
    
    In particular, for any finitely generated module $M$, $\Mat(M)$ cannot have a null recurrent class of convergence radius $1/p^{\delta_R(R)}$.
\end{corollary}

\begin{proof}
If $M$ is a torsion module, then any connected class $\cC$ of $\Mat(M)$ must be contained in $\find (\p, M)$ for $\p \neq 0$, 
hence $R(\cC) > p^{-\delta_R(R)}$ by \Cref{cor: convergence radius is at least the power of dimension}. 
If $M$ is not a torsion module, by 
\Cref{cor: F contributors are SFR when there is a module-finite SFR extension}, (\ref{part: positivity when there is a mod finite SFR extension}), $\find(\iota_* S) \subseteq \find(M)$,
giving a $p^{-\delta_R(R)}$-recurrent class in $\Mat(M)$. 
On the other hand, if $\cC$ is a $p^{-\delta_R(R)}$-recurrent class in $\Mat(M)$, then $0 \in \Min (\cC)$ by \Cref{cor: convergence radius is at least the power of dimension}, so 
Proposition~\ref{prop: consequence of existence of one strongly F regular contributor} implies that 
$\cC \subseteq \find (\iota_* S)$.
\end{proof}

The above-mentioned results show existence of F-splitting signatures and describe the F-contributors for invariant rings of finite groups.

\begin{corollary}\label{cor: invariant ring}
    Let $(S, \mathfrak n)$ be a strongly F-regular F-finite local ring and $R$ be the invariant ring of the action of a finite group on $S$. Suppose that $R$ is Noetherian and $S$ is module-finite over $R$ (e.g., $p$ does not divide $|G|$). 
    Denote
    $$\cS_S = \{N| N \ \text{is an indecomposable}\ \text{$R$-summand of}\ F_{R,*}^eS\ \text{for some}\ e\geq 0\}.$$
    Then for any finitely generated $R$-modules $M,N$, where $N$ is indecomposable
    \begin{enumerate}
        \item The sequence $e \mapsto \Fmult{e}{M}{N}/p^{e\delta_R(R)}$ is convergent.

        \item The limit of the above sequence is positive if and only if $M$ is not a torsion module and $N \in \cS_S$.

        \item The set of indecomposable $R$-modules which are F-contributors of dimension $\dim(R)$ or, equivalently, strongly F-regular modules of dimension $\dim(R)$ is $\cS_S$.

        \item If $S$ has the FFRT property as an $S$-module, $\cS_S$ is a finite set.
    \end{enumerate}
\end{corollary}

\begin{proof}
A standard argument shows that $R$ is a local ring. So \Cref{cor: F contributors are SFR when there is a module-finite SFR extension} implies (1), (2) and (3). For (4) suppose $\find(S)$, formed by considering $S$ as an $S$-module, is $\{M_1, \ldots, M_n\}$. Then $\cS_S$ is the union of indecomposable $R$-summands of $M_i$ for $1 \leq i \leq n$; hence it is finite.
\end{proof}

\begin{remark}\label{re: Hashimoto Symonds}
    One special case of \Cref{cor: invariant ring} is when $S$ is a power series ring over an F-finite field $k$. Assume, additionally, that the action of the group is linear, i.e., it sends the variables to a $k$-linear combination of the variables and fixes every element of $k$. 
    Then for any two $R$-summands $M, N$ of $S$,  \cite{HashimotoSymonds} implies the convergence of $(\Fmult{e}{M}{N}/p^{e\delta_R(R)})_e$\footnote{\cite{HashimotoSymonds} shows convergence in a metric, which is stronger than the mere convergence.}, the positivity of the corresponding limit $\gsig{M}{N}$, and 
    computes $\gsig{M}{N}$ in terms of representation-theoretic information.
    In our language, \cite{HashimotoSymonds} implies that the modules in $\find(S)$ are F-contributors of period one. However, their representation-theoretic method does not identify $\find(S)$ with the set of F-contributors of dimension $\dim(R)$. 
\end{remark}

\subsubsection{Properties of splitting signatures}
We will now establish a rank formula, a total rank formula, and a geometric convergence rate over a strongly F-regular ring $R$.
In the proof of the rank formula below, we harness our formalism for rank formula using the F-decomposition matrix in \Cref{sse: rank formula} to exploit the known rank formula towards $R$.

\begin{theorem}\label{th: rank formula over sfr rings}
    Let $(R, \m)$ be a strongly F-regular local ring. For any finitely generated $R$-modules $M$ and $X$, where $X$ is indecomposable, we have 
    \[\lim_{e \to \infty} \frac{\Fmult{e}{M}{X}}{p^{e\delta_R(R)}}= \rank_R(M) \gsig{R}{X}.\]
\end{theorem}

\begin{proof}
The $R$-module $R$ is an F-contributor by \cite{AberbachLeuschke}. Since $R$ is strongly F-regular, by
\Cref{cor: F contributors are SFR when there is a module-finite SFR extension}(\ref{part: positivity when there is a mod finite SFR extension}), it suffices to assume that $X \in \find (R)$
and $M$ is not torsion. 
Since any such $X$ has period one (\Cref{th: further properties of SFR modules}), it suffices to prove the \ces rank formula of \Cref{de: rank formula}. Note that in the terminology of \Cref{de: rank formula}, we are concerned with the rank formula at the zero ideal.
When $X \cong R$, the desired rank formula follows from \cite[Theorem 4.11]{Tucker}. Therefore, our general result \Cref{pr: comparison of notions of rank formulas and consequences}(\ref{part: cesaro rank formula is a class property of the target}) implies the rank formula at the zero ideal towards any F-contributor in $\find((0),M)$, which is precisely $\find(R)$ by 
\Cref{cor: unique maximal dimensional connected class}.
\end{proof}

As a consequence, we derive the total rank formula (\Cref{sse: Total rank formula}) over strongly F-regular local rings.

\begin{corollary}\label{cor: total rank formula in the SFR case}
    Let $(R, \m)$ be a strongly F-regular local ring and $M$ be a finitely generated $R$-module. Then 
    \[
    \Big |\rank_R(M) - \sum_{X \in \find(R)} p^{-e\delta_R (R)}\Fmult{e}{M}{X}\rank_R(X)\Big | = o(1).
    \]
    In particular, $\rank_R(M)= \sum_{X \in \find(R)}\gsig{M}{X}\rank_R(X).$
\end{corollary}

\begin{proof}
    It suffices to assume that $M$ is not torsion; otherwise $\Fmult{e}{M}{X} = 0$ for any $X \in \find(R)$. Since the rank formula at the zero ideal holds for $M$, the total rank formula at the zero ideal also holds for $M$, i.e.,
    \[\rank_R(M)= \sum_{X \in \find(R)}\csig{M}{X}_{(0)}\rank_R(X);\,\, \text{see \Cref{de: rank quantization}, \Cref{prop: total rank formula vs rank formula}}.\]
    It remains to note that, since any $X \in \find(R)$ has period one, $\csig{M}{X}_{(0)}= \gsig{M}{X} = \lim_{e \to \infty} p^{-e\delta_R (R)}\Fmult{e}{M}{X}$. The first claim now follows from \Cref{re: interpretation of total rank formula}.
\end{proof}

Our general formalism for convergence rates developed in \Cref{subsec: convergence rates} allows us to deduce convergence rates of $(\Fmult{e}{M}{X}/p^{e\delta_R(R)})_e$ towards $\gsig{M}{X}$, for any $X$ in the connected class containing $R$ in $\Mat(R)$, from the known case of $X \cong R$.

\begin{theorem}\label{th: convergence rate in the sfr case}
    Let $(R, \m)$ be an F-finite local ring, not necessarily strongly F-regular, and $M$ be a finitely generated $R$-module. Then there is a nonnegative constant $\lambda_M<p^{\delta_R(R)}$ such that for any $M' \in \find(M)$ and $X$ in the connected class of $\Mat(R)$ containing $R$,
    \[\Fmult{e}{M'}{X}= \gsig{M'}{X}p^{e\delta_R(R)}+ O(\lambda_M^e).\]
\end{theorem}

\begin{proof}
    If $M$ is a torsion module, then one can take $\lambda_M= p^{\delta_R(M)}< p^{\delta_R(R)}$ (\Cref{cor: convergence radius is at least the power of dimension}). So, henceforth assume that $M$ is not a torsion module. First, assume that $R$ is not strongly F-regular. Then $\gsig{M}{R}=0$ and 
    \[\Fmult{e}{M}{R}= O(p^{e(\delta_R(R)-1)}),\,\,\, \text{see, for example, \cite{deStefani_Polstra_Yao_globalizing_F_invariants}}.\]
    We argue that for any $X$ in the connected class of $\Mat(R)$ containing $R$, 
    \[\Fmult{e}{M}{X}= O(p^{e(\delta_R(R)-1)}).\]
    To this end, suppose that for one such $X$, there exists a sequence $(e_i)_{i \in \mathbb N}$ of positive integers such that the sequence $(\Fmult{e_i}{M}{X}/p^{e_i(\delta_R(R)-1)})_{i}$ is unbounded. Take an $e'$ such that $\Fmult{e'}{X}{R}$ is nonzero. Then $(\Fmult{e_i+e'}{M}{R}/p^{(e_i+e')(\delta_R(R)-1)})_{i}$ is also unbounded, contradicting our earlier observation. Since $M'$ is a summand of some $F^e_*M$, a straightforward argument shows that
    $\Fmult{e}{M'}{X}= O(p^{e(\delta_R(R)-1)})$.
    
    From now on, we assume that $R$ is strongly F-regular. We know from \cite[Theorem 4.9]{deStefani_Polstra_Yao_globalizing_F_invariants}\footnote{Also follows from earlier works \cite{Tucker} and \cite{PolstraTucker}.} that
    \[\Fmult{e}{M}{R}= \gsig{M}{R}p^{e\delta_R(R)}+ O(p^{e(\delta_R(R)-1)}).\]
    The rest now follows from \Cref{th: convergence rates}, (\ref{part: convergence rate of the parent module to anything in a class}) applied to a prime ideal in $\Minh(R)$.
\end{proof}

\subsection{F-split rings}\label{subsec: F split rings}
A natural question arising from this line of thought is whether 
F-split rings $R$ with irreducible $\Mat(R)$ are strongly F-regular. In \Cref{prop: Mat is irreducible} we will provide some evidence, but it should also be noted that the irreducibility of $\Mat(R)$ does not imply strong F-regularity alone. Indeed, 
using \cite{MallorySayrafi} 
one can verify that the $D_6^0$-singularity $z^2 + x^2y + xy^3$ 
in characteristic $2$ has the following 
decomposition: $F_* R = M^{\oplus 2}$ and $F_* M = M^{\oplus 4}$, so by our definition $\Mat(R) = \begin{bmatrix}
    4
\end{bmatrix}$. 

\begin{lemma}\label{lem: basic properties of Mat}
Let $(R, \m)$ be an F-finite local ring. If $R$ is F-split, then the convergence radius  $R(\Mat(R)) \leq p^{-\sdim R - \alpha (R)}$.
\end{lemma}

\begin{proof}
Since it is known (\cite{BlickleSchwedeTucker} and \cite{AberbachEnescu}) that $\lim_{e \to \infty} \Fmult{e}{R}{R}/p^{e(\sdim R + \alpha (R))}$  is a positive real number,  
the radius of convergence of the connected class $\cC$
of $R$ in $\Mat(R)$ is $p^{-(\sdim (R) + \alpha (R))}$. 
But then $R(\Mat(R)) \leq R(\cC)$ by definition.
\end{proof}

The following lemma extends an observation in \cite{PolstraSmirnov}.

\begin{lemma}\label{lem: ratio is 1}
    Let $(R, \m)$ be an F-split  F-finite local ring
    and let $\cP$ denote the splitting prime of $R$.
    The following are equivalent:
    \begin{enumerate}
        \item $R/\cP$ is regular and $I_e = \cP + \m^{[p^e]}$ for all $e \geq 1$;
        \item $\Fmult{e}{R}{R} = p^{e (\sdim (R) + \alpha(R))}$ for all $e \geq 1$;
        \item $\Fratio (R) = 1$.
    \end{enumerate}
\end{lemma}
\begin{proof}
In this proof, for an F-finite local ring $S$, we use the notation $\Fmult{e}{S}{S}$ to denote the free rank of $F^e_*S$, even when $S$ does not have KRS.

    The implications $(1) \Rightarrow (2) \Rightarrow (3)$
    are immediate. We now show that $(3) \Rightarrow (2)$.
    By \cite[Proposition~3.6]{AberbachEnescu}, $\sdim (R_\cP) = 0$
    so by the proof of \cite[Corollary~3.4]{AberbachEnescu}
    $\Fmult{e}{R_\cP}{R_\cP} = p^{e (\sdim (R) + \alpha(R))}$.
    Hence, it suffices to show that $\Fmult{e}{R_\cP}{R_\cP} = \Fmult{e}{R}{R}$ for all $e \geq 1$. 
    We follow the approach of \cite[Theorem~3.5]{PolstraSmirnov}.

    By contradiction, suppose that $\Fmult{e_0}{R_\cP}{R_\cP} > \Fmult{e_0}{R}{R}$ for some $e_0 \geq 1$.
    This means that there is an indecomposable module $M$ 
    such that $M$ is a direct summand of $F_*^{e_0}R$ and $M_{\cP}$ has a free summand.
    Thus for $e \geq 1$
    we see that $M^{\oplus \Fmult{e}{R}{R}}$ is a direct summand of $F_*^{e_0 + e} R$. 
    Therefore, we may write 
    $F_*^{e_0 + e} R \cong R^{\oplus \Fmult{e+e_0}{R}{R}} \oplus M^{\oplus b_e} \oplus N_e$ where $b_e \geq \Fmult{e}{R}{R}$.
    In particular, we obtain after localizing that 
    $ \Fmult{e+e_0}{R_\cP}{R_\cP} \geq \Fmult{e+e_0}{R}{R} + \Fmult{e}{R}{R}.$
    It follows that $\Fratio (R_\cP) \geq 1 + p^{-e_0 (\sdim (R) + \alpha(R))}$, a contradiction. 

    We now prove the remaining implication. By \cite[Theorem~4.2]{BlickleSchwedeTucker},
    if $\Fratio (R) = 1$, then $R/\cP$ is regular.
    The conclusion now follows after noting that
    $\cP + \m^{[p^e]} \subseteq I_e$ and that the two ideals have equal colengths by $(2)$. 
\end{proof}

\begin{remark}
It is possible that $R/\cP$ is regular but $\Fratio (R) < 1$; see \cite[Remark 4.4]{Smirnov}. 
\end{remark}

\begin{proposition}\label{prop: Mat is irreducible}
    Let $(S, \m, k)$ be an F-finite F-split local ring of characteristic $p > 0$. 
    Suppose that $\Mat(S)$ is irreducible. 
    Then $R(\Mat(S)) = p^{-\sdim (S) - \alpha(S)}$
    and $\Mat(S)$ is aperiodic positive recurrent.
    Moreover, if either
    \begin{itemize}
        \item there is an indecomposable F-contributor to $S$, or
        \item $\Fratio (S) = 1$ (e.g., $\sdim (S) = 0$),
    \end{itemize}
    then $S$ is strongly F-regular. 
\end{proposition}
\begin{proof}
    If $\Mat(S)$ is irreducible, one easily obtains from \Cref{lem: basic properties of Mat} that 
    $R(\Mat(S)) = p^{-\sdim (S) - \alpha(S)}$. The F-splitness implies aperiodicity. The presence of an F-contributor forces that $R(\Mat(S)) = p^{-\dim(S) - \alpha(S)}$, so $\dim(S)= \sdim(S)$. Hence, $\fsig(S) = \Fratio (S) > 0$ and $S$ is strongly F-regular. 
    
    We now assume that $\Fratio (S) = 1$. By 
    \Cref{lem: ratio is 1}, we must then have 
    $\Fmult{e}{S}{S} = (p^{\sdim (S) + \alpha(S)})^e$
    for all $e \geq 1$. 
    Since $\Mat(S)$ is irreducible, this is only possible if $S$ is the only summand of $F_* S$: 
    if there were a different summand $M$
    then $\Fmult{e}{S}{S} \geq (p^{\sdim (S) + \alpha(S)})^e + \Fmult{e-1}{M}{S} > (p^{\sdim (S) + \alpha(S)})^e$ for $e \gg 0$, since $S \in \find(M)$. Hence $S$ is regular by \cite{Kunz1} and the conclusion follows.    
\end{proof}    

\begin{remark}
    Algebraically, $\Mat (S)$ is irreducible if and only if for all $e \geq 1$ any direct summand $M$ of $F_*^e S$ is not contained in $F_*^e \cP$.
    By combining with Lemma~\ref{lem: ratio is 1}
    one might be able to give a module-theoretic proof of the second case of Proposition~\ref{prop: Mat is irreducible}.
\end{remark}

\begin{remark}\label{rem: eigenvectors for F-split}
    When $R$ is F-split, the restriction of $\Mat(R)$ 
    to the connected class containing $R$ is positive $p^{-\sdim (R) - \alpha(R)}$-recurrent, hence it has distinguished positive left $v$ and right $u$ eigenvectors by Theorem~\ref{thm: vere-jones eigenvectors}. A natural form of these eigenvectors is
    \[
    u_M \coloneqq \lim_{e \to \infty} p^{-e\sdim (R) - e\alpha(R)} \Fmult{e}{M}{R} \text{ and }
    v_M \coloneqq \lim_{e \to \infty} p^{-e\sdim (R) - e\alpha(R)} \Fmult{e}{R}{M}.
    \]
    Since 
    $\Fmult{e+e'}{R}{R} = \sum_M \Fmult{e}{R}{M} \Fmult{e'}{M}{R}$, one can easily check that $u \cdot v \leq \Fratio (R)$. In fact, \Cref{thm: Vere-Jones limit exists} shows that $u \cdot v = \Fratio (R)$.
    \end{remark}

\section{Two-dimensional case}\label{se: dimension two}
In this section, we focus on F-contributors of dimension at most two. For any such F-contributor $X$, we prove the convergence of $(\Fmult{e}{M}{X}/p^{e\delta_R(X)})_e$, by proving that $X$ has period one (\Cref{cor: two dimensional F-contributors}). Moreover, we establish a rank formula and a convergence rate of $\Fmult{e}{M}{X}/p^{e\delta_R(X)}$ towards its limit (\Cref{th: rank formula in dimension two}). 

The technical heart of this section is \Cref{th: most splittings lift in dimension two}, which is essentially about proving strong F-regularity of such F-contributors. We now head towards \Cref{th: most splittings lift in dimension two} by proving some essential ingredients.

For an ideal $I$ in a noetherian ring $S$ and a finitely generated $S$-module $M$, such that $\length (M/IM)$ is finite, $\eh(I,M)$ denotes the Hilbert--Samuel multiplicity of $M$, i.e., 
\[\eh(I,M)=\lim_{n \to \infty} \dim(S)!\frac{\length(M/I^nM)}{n^{\dim(S)}}.\]
Let $\mu_S(M)$ denote the minimal number of generators of the $S$-module $M$.

\begin{lemma}\label{le: minimal number of generators in dimension one}
    Let $(R,\m)$ be a local ring  of dimension at most one and $M$ be a finitely generated $R$-module. Let $y$ be a system of parameters of $R$; $y$ is taken to be zero when $\dim(R)=0$. Then for any submodule $N \subseteq M$, 
    \[\mu_R(N) \leq \eh(yR,M)+ \length_R(0:_My) \leq \eh(yR,M)+ \length (\Hh^0_{\m}(M)).\]   
\end{lemma}

\begin{proof}
When $\dim(R)=0$, the assertions are straightforward. Assume $\dim(R)=1$. $\mu_R(N) = \length(N/\m N) \leq \length(N/yN)$. By Serre's formula for multiplicities, $\eh(yR,N)= \length(N/yN)- \length(0:_N y)$. Thus $\mu_R(N) \leq \eh(yR,N)+ \length(0:_N y) \leq \eh(yR,M)+ \length(0:_M y)$. Since $\Hh^0_\m(M)= \cup_{n\in \mathbb N}(0:_My^n)$, the rightmost inequality follows.
\end{proof}

The next lemma is straightforward.

\begin{lemma}\label{le: length composed with a functor is subadditive}
Let $S, S'$ be two rings and $G$ be a covariant or contravariant half-exact functor from the category of $S$-modules to the category of $S'$-modules. For any short exact sequence of $S$-modules
\[ 
0 \to M' \rightarrow M \rightarrow M'' \to 0
\] 
where $G(M'), G(M), G(M'')$ all have finite lengths, $\length(G(M)) \leq \length(G(M'))+ \length(G(M''))$. 
\end{lemma}

\begin{corollary}\label{cor: length of ext to pushforward of finite length}
    Let $(R, \m, k)$ be an F-finite local ring of characteristic $p > 0$
    and $M, N$ be finitely generated $R$-modules.
    If $N$ has finite length, then for any $i \in \mathbb N$, the minimal numbers of generators of $\Ext^i_R (M, F_*^e N)$ and $\Ext^i_R (F_*^e N,M)$ are at most $p^{e\alpha(R)}\length (N) \length (\Ext^i_R (M, k))$ and $p^{e\alpha(R)}\length (N) \length (\Ext^i_R (k,M))$ respectively.
\end{corollary}

\begin{proof}
    The modules concerned are of finite length, and the minimal number of generators of a finite length module is bounded above by its length. The rest follows from \Cref{le: length composed with a functor is subadditive}.
\end{proof}

\begin{lemma}\label{le: minimal number of generators of submodules of Ext groups}
    Let $(R, \mathfrak m)$ be a local ring, $M$ be a finitely generated Cohen--Macaulay module of dimension one or two. Let $N$ be a finitely generated $R$-module of dimension at most $\max\{\dim(M)-1,0\}$. Then there exists a constant $C$, depending only on $M$ and $N$, such that for any $e \in \mathbb N$, the minimal number of generators of any submodule of $\Ext^1_R(F^e_*N,M)$ is at most $Cp^{(\dim(M)-1+\alpha(R))e}$. 
\end{lemma}

\begin{proof}
When $\dim(M)=1$, $N$ has finite length. So the lemma follows from \Cref{cor: length of ext to pushforward of finite length}.

Suppose $\dim (M)=2$. If the Krull dimension of $N$ is zero,  $\Ext^1_R(F^e_*N,M)$ vanishes since $M$ has depth two. So we assume that $\dim(N)=1$. Since $M$ is Cohen--Macaulay, the minimal primes in $\Supp(M)$ are of dimension two. So, these minimal primes cannot be in $\Supp(N)$. Hence, we can choose an element $f \in \Ann_R (N)$, which is not in any minimal prime of $M$. 
Choose $y \in \mathfrak m$ to be regular on $M/fM$. 
By \cite[Lemma~1.2.4]{BrunsHerzog},
$\Ext^1_R (F_*^e N, M) \cong \Hom_R (F_*^e N, M/fM)$.
So $y$ is a regular element on the one-dimensional module $\Ext^1_R (F_*^e N, M)$. Let $R'= R/\Ann(\Ext^1_R (F_*^e N, M))$. Then $y$ is a system of parameters on $R'$.
By \Cref{le: minimal number of generators in dimension one}, the minimal number of generators of any $R'$ or equivalently $R$-submodule of $\Ext^1_R(F^e_*N,M)$ is at most
\begin{equation*}
\begin{split}
\eh(yR', \Ext^1_R(F^e_*N,M)) &= \sum_{\p \in \Minh(N)} \length_{R_\p}(\Ext^1_{R_\p}(F^e_*N_\p,M_\p))\eh(yR', R/\p).
\end{split}
\end{equation*}
For any $\p \in \Minh(N)$, $\length_{R_{\p}}(F^e_*N_\p)= \length_{R_\p}(N_\p)p^{e(1+\alpha(R))}$. Thus, by \Cref{le: length composed with a functor is subadditive}
\[\eh(yR', \Ext^1_R(F^e_*N,M))\leq \sum_{\p \in \Minh(N)} \length_{R_\p}(\Ext^1_{R_\p}(\kappa(\p),M_\p))\eh(yR', R/\p)\length_{R_\p}(N_\p)p^{e(1+\alpha(R))}.\]
\end{proof}

\begin{lemma}\label{le: bound on minimal number of generators with target twisted}
    Let $(R, \m, k)$ be a local ring, $M$ be a finitely generated Cohen--Macaulay module of dimension one or two. Let $N$ be a finitely generated $R$-module of dimension at most $\max\{\dim(M)-1,0\}$. Then there exists a constant $C$ such that for any $e$, the minimal number of generators of any submodule of $\Ext^1_R(M, F^e_*N)$ or $\Hom_R(M, F^e_*N)$ is at most $Cp^{e(\dim(M)-1+ \alpha(R))}$.     
\end{lemma}

\begin{proof}
If $\dim(N)= 0$, the lemma follows from \Cref{cor: length of ext to pushforward of finite length}. Thus we assume $\dim(M)=2, \dim(N)=1$, from now on.

Consider the long exact sequence of $\Ext^{\bullet}_R(M, \bullet)$ induced by the exact sequence
\[0 \rightarrow F^e_*\Hh^0_\m(N) \rightarrow F^e_*N \rightarrow F^e_*\frac {N}{\Hh^0_\m(N)} \rightarrow 0.\]
By \Cref{le: length composed with a functor is subadditive}, \Cref{cor: length of ext to pushforward of finite length}, there is a constant $C'$ such that the minimal number of generators of any submodule of $\Hom_R(M, F^e_*\Hh^0_\m(N))$ or $\Ext^1_R(M, F^e_*\Hh^0_\m(N))$ is at most $C'p^{e\alpha(R)}$, for any $e$. Thus, we can replace $N$ by $N/\Hh^0_\m(N)$, and assume that $\Hh^0_\m(N)=0$ from now on.

By our assumption $\m$ is not associated to $M$ or $N$. Thus, we can and do choose $y \in \m$ which is a nonzero divisor on both $M$ and $N$. In particular, the image of $y$ is a system of parameters in $R/\Ann(N)$. 
We consider $\Hom_R(M, F^e_*N)$ and $\Ext^1_R(M, F^e_*N)$ as modules over the one-dimensional ring $R/\Ann(N)$ and apply \Cref{le: minimal number of generators in dimension one}.
\begin{claim}
    There is a constant $C$ such that for any $n \in \mathbb N$, both $\length(0:_{\Hom_R(M, F^e_*N)}y^n)$ and $\length(0:_{\Ext^1_R(M, F^e_*N)}y^n)$ are bounded above by $C np^{e(1+\alpha(R))}.$
\end{claim}
\begin{proof}
For $\length(0:_{\Hom_R(M, F^e_*N)}y^n)$, note, 
\[
0:_{\Hom_R(M, F^e_*N)}y^n \cong 
\Hom_R (R/y^nR, \Hom_R (M, F_*^e N))
\cong \Hom_R(M/y^nM, F^e_*N). 
\]
But $\Hom_R(M/y^nM, F^e_*N) \cong 0$ as $y$ is a non-zero divisor on $N$.  

For $\length(0:_{\Ext^1_R(M, F^e_*N)}y^n)$, consider the derived sequence of Ext groups 
applied to $0 \rightarrow M \xrightarrow[]{y^n} M \rightarrow M/y^nM \rightarrow 0$. It shows, 
\begin{equation}\label{eq: bounding the torsion part quotient space}
\length(0:_{\Ext^1_R(M, F^e_*N)}y^n)\leq \length(\Ext^1_R(M/y^nM, F^e_*N)).    
\end{equation}
 By \cite[Lemma~1.2.4]{BrunsHerzog},
\[
\Ext^1_R(M/y^n M,F^e_*N) \cong 
\Hom_R(M/y^n M,  F^e_*\frac{N}{y^{np^e}N})
\cong
\Hom_R(M, F^e_*\frac {N}{y^{np^e}N}).
\]
Thus, by \Cref{le: length composed with a functor is subadditive},
\[
\begin{split}
\length(0:_{\Ext^1_R(M, F^e_*N)}y^n)
&\leq p^{e \alpha(R) }\length(\Hom_R(M, k)) \length(N/y^{np^e}N)
\\ &\leq n p^{e(1+\alpha(R))} \length(\Hom_R(M,k))\length(N/y N).    
\end{split}
\]
\end{proof}

We continue the proof of the lemma. Set $R'= R/\Ann(N)$. By associativity of multiplicities and
 \Cref{le: length composed with a functor is subadditive},
\[
  \begin{split}
\eh(yR', \Ext^1_R(M,F^e_*N)) &= \sum_{\p \in \Minh(N)}\length(\Ext^1_R(M,F^e_*N)_\p)\eh(yR', R/\p) \\ & \leq \sum_{\p \in \Minh(N)} p^{e(1+\alpha(R))}\length(\Ext^1_{R_\p}(M_\p,R_\p/\p R_\p))\eh(yR', R/\p)\length(N_\p).
\end{split}  
\]
A similar argument shows that
$\eh(yR', \Hom_R(M,F^e_*N))= O(p^{e(1+\alpha(R))})$. Using these multiplicity bounds, the claim, and \Cref{le: minimal number of generators in dimension one}, we conclude the lemma.
\end{proof}

The proof of the following lemma is straightforward.

\begin{lemma}\label{le: ascend and descend of splitting}
    Let $(R, \mathfrak m)$ be a local ring, $X$ be a finitely generated $R$-module, and $\phi\in \Hom_R(M, N)$ be an $R$-linear map between two finitely generated modules.
    \begin{enumerate}
        \item Let $\begin{tikzcd}
	{M} & {} & {X^{\oplus m}}
	\arrow["{\pi=(\pi_1, \cdots, \pi_{m})}", from=1-1, to=1-3]
\end{tikzcd}$ be an $R$-linear split projection. Let $\Lambda \subseteq \{\pi_1, \ldots, \pi_{m}\}$ be such that each $ \pi_i \in \Lambda$ lifts along $\Hom_R(N,X) \xrightarrow[]{-\circ \phi} \Hom_R(M,X)$ to some $\tilde \pi_i$. Then, the map
$\begin{tikzcd}
	{N} & {} & {X^{\oplus |\Lambda|}}
	\arrow["{(\tilde \pi_i)_{\pi_i \in \Lambda}}", from=1-1, to=1-3]
\end{tikzcd}$
is also a split projection.

\item Let $\begin{tikzcd}
	{X^{\oplus n}} & {} & {N}
	\arrow["{s=(s_1, \cdots, s_{n})}", from=1-1, to=1-3]
\end{tikzcd}$ be an $R$-linear split injection. Let $\Gamma \subseteq \{s_1, \ldots, s_{n}\}$ be such that each $s_i \in \Gamma$ lifts along $\Hom_R(X, M) \xrightarrow[]{\phi \circ-} \Hom_R(X, N)$ to some $\tilde s_i$. Then $\begin{tikzcd}
	{X^{\oplus |\Gamma|}} & {} & {M}
	\arrow["{(\tilde s_i)_{s_i \in \Gamma}}", from=1-1, to=1-3]
\end{tikzcd}$ is a split injection.
    \end{enumerate}
\end{lemma}

We record a criterion for lifting split injective or surjective maps, which is used in \Cref{th: most splittings lift in dimension two}. 

\begin{proposition}\label{prop: minimal number of generators and lifting}
Let $(R,\m)$ be a local ring and $X$ be a finitely generated indecomposable $R$-module. Let $E$ be an $R$-module.
\begin{enumerate}
   \item Let $M' \xrightarrow[]{\Pi= (\pi_1, \ldots, \pi_n)} X^{\oplus n}$ be a split projection and $\partial\colon \Hom_R(M',X) \rightarrow E$ be an $R$-linear map. If the submodule of $E$ generated by $\partial(\pi_1), \ldots, \partial(\pi_n)$ is generated by $r\leq n$ elements, then there is an invertible $n \times n$ matrix $T$ over $R$ such that $T\Pi$ is a split projection $M' \rightarrow X^{\oplus n}$ all but at most $r$ of whose components lie in $\ker(\partial)$.
   \item Let $X^{\oplus n} \xrightarrow[]{S= (s_1, \ldots, s_n)} M''$ be a split injection and $\partial\colon \Hom_R(X,M'') \rightarrow E$ be an $R$-linear map. If the submodule of $E$ generated by $\partial(s_1), \ldots, \partial(s_n)$ is generated by $r\leq n$ elements, then there is an invertible $n \times n$ matrix $T$ over $R$ such that $ST$ is a split injection $X^{\oplus n} \rightarrow M''$ all but at most $r$ of whose components lie in $\ker(\partial)$.
\end{enumerate}
\end{proposition}
\begin{proof}
Since the proofs of (1) and (2) are similar, we present the argument for (1) only. Using Nakayama's lemma, we assume without loss of generality that the submodule of $E$ generated by $\partial(\pi_i)$'s is  generated by $\partial(\pi_1), \ldots, \partial(\pi_r)$. So, for any $1 \leq i \leq n-r$, inside $E$, $\partial(\pi_{r+i}) = \sum_{j=1}^{r}a_{r+i,j}\partial(\pi_j)$, for some choices of $a_{r+i,j} \in R$. 
For all $1\leq i\leq n-r$, define $\pi'_i= \pi_{r+i}-  \sum_{j=1}^{r}a_{r+i,j}\pi_j,$
so that $\partial(\pi'_i) = 0$. 

Now, consider the lower triangular $n \times n$-matrix,
\[
T=\begin{bNiceArray}{ccc|ccc c c c c}
\Block{3-3}{\text{Id}_{r \times r}} & & & &&\Block{3-3}{\text{O}_{r \times (n-r)}}& & & &\\
&&&&&&&&&\\
&&&&&&&&&\\
\hline
-a_{r+1,1}&\ldots&-a_{r+1,r}&&&\Block{3-3}{\text{Id}_{n-r \times n-r}}& &&&  \\

\vdots&\vdots&\vdots&&&&&&&\\

-a_{r+j,1}&\ldots& -a_{r+j,r}&&&&&&&\\
\vdots&\vdots&\vdots&&&&&&&
\end{bNiceArray}
\begin{array}{l}
\\
\\
\\
\\
\\
\leftarrow\ r+j-\text{th row}\\
\end{array}
\]
Considering $\Pi$ as the column vector of morphisms $(\pi_1, \ldots, \pi_n)^{\tr}$, we have 
\[T\Pi= (\pi_1, \ldots, \pi_r, \pi'_1, \ldots, \pi'_{n-r})^{\tr}.\]
Consider a split injection $S' \colon X^{\oplus n} \rightarrow M'$ corresponding to $\Pi$. We think of $S'$ as a row vector of morphisms. Since $T$ is invertible, we get another row vector of morphisms $S'T^{-1}\colon X^{\oplus n} \rightarrow M'$. Since 
$T\Pi S' T^{-1}= \Id_{X^{\oplus n}}$,
$T\Pi= (\pi_1, \ldots, \pi_r, \pi'_1, \ldots, \pi'_{n-r})^{\tr}$ is a split projection whose last $n-r$ components lie in $\ker \partial$.
\end{proof}

\begin{proposition}\label{th: most splittings lift in dimension two}
Let $(R, \mathfrak m)$ be a local ring, and $X$ be an indecomposable F-contributor of dimension one or two\footnote{All the assertions hold when $\dim(X)=0$ using the argument in \Cref{le: dimension zero contributor}.}. Let $\phi\colon M' \rightarrow M''$ be an $R$-linear map between two finitely generated $R$-modules such that $\dim(M') = \dim(M'') = \dim(X)$, 
and both $\ker(\phi)$ and $\coker(\phi)$ have dimension at most $\dim(X)-1$. For any $e \in \mathbb{N}$, denote $a'_e \coloneqq \Fmult{e}{M'}{X}$ and $a''_e \coloneqq \Fmult{e}{M''}{X}$.

\begin{enumerate}
    \item Assume that $X \in \find(M')$. Then there exists a constant $C_1$ such that, for any $e\geq 1$, there exists an $R$-linear split projection
    \[
    F^e_*M' \xrightarrow[]{\pi=(\pi_1, \ldots, \pi_{a'_e})} X^{\oplus a'_e},
    \]
    where all but at most $C_1p^{(\dim(X)-1+ \alpha(R))e}$ many morphisms among $\{\pi_1, \ldots, \pi_{a'_e}\}$ lift along the map $\Hom_R(F^e_*M'',X) \xrightarrow{- \,\circ\, F^e_*\phi} \Hom_R(F^e_*M',X)$. In particular,
    \[\Fmult{e}{M''}{X} \geq \Fmult{e}{M'}{X} - C_1p^{(\dim(X)-1+ \alpha(R))e}.\]

    \item Assume that $X \in \find(M'')$. Then there exists a constant $C_2$ such that, for any $e \geq 1$, there exists an $R$-linear split injection
    \[
        X^{\oplus a''_e} \xrightarrow{s=(s_1, \ldots, s_{a''_e})} F^e_*M'',
    \]
    where all but at most $C_2p^{(\dim(X)-1+ \alpha(R))e}$ many morphisms among $\{s_1, \ldots, s_{a''_e}\}$ lift along the map $\Hom_R(X, F^e_*M') \xrightarrow{F^e_*\phi \,\circ\, -} \Hom_R(X, F^e_*M'')$. In particular,
    \[\Fmult{e}{M'}{X} \geq \Fmult{e}{M''}{X} - C_2p^{(\dim(X)-1+ \alpha(R))e}.\]
\end{enumerate}
\end{proposition}

\begin{proof}
In (1) and (2), the comparisons of $\Fmult{e}{M'}{X}$ and $\Fmult{e}{M''}{X}$ follow from the respective lifting assertions thanks to \Cref{le: ascend and descend of splitting}. Thus we only argue for the lifting assertions.

Write $K= \ker(\phi)$, $I= \operatorname{Im}(\phi)$ and $Q= \coker(\phi)$, and factor $\phi= \iota \circ \rho$ with $\rho\colon M' \twoheadrightarrow I$ and $\iota\colon I \hookrightarrow M''$. Since $F^e_*$ is exact, for every $e \geq 1$ we have short exact sequences
\[0 \rightarrow F^e_*K \rightarrow F^e_*M' \xrightarrow{F^e_*\rho} F^e_*I \rightarrow 0, \qquad 0 \rightarrow F^e_*I \xrightarrow{F^e_*\iota} F^e_*M'' \rightarrow F^e_*Q \rightarrow 0.\]
Recall that $X$ is Cohen--Macaulay (\Cref{cor: Cohen--Macaulayless of contributors}) and $\dim(K), \dim(Q) \leq \dim(X)-1$. By \Cref{le: minimal number of generators of submodules of Ext groups} and \Cref{le: bound on minimal number of generators with target twisted}, there is a constant $C$ such that for every $e \geq 1$, any submodule of $\Ext^1_R(F^e_*Q, X)$, $\Hom_R(X, F^e_*Q)$ or $\Ext^1_R(X, F^e_*K)$ is generated by at most $Cp^{e(\dim(X)-1+\alpha(R))}$ elements.

(1) Since every associated prime of $X$ has dimension $\dim(X)$ while $\dim(F^e_*K) < \dim(X)$, we have $\Hom_R(F^e_*K, X)=0$. Hence $-\circ F^e_*\rho\colon \Hom_R(F^e_*I, X) \rightarrow \Hom_R(F^e_*M', X)$ is an isomorphism, and the long exact sequence for $\Hom_R(-,X)$ of the second short exact sequence gives an $R$-linear map
\[\partial\colon \Hom_R(F^e_*M', X) \cong \Hom_R(F^e_*I, X) \rightarrow \Ext^1_R(F^e_*Q, X)\]
whose kernel consists exactly of the maps that lift along $- \circ F^e_*\phi$. Applying \Cref{prop: minimal number of generators and lifting}(1) to a split projection $F^e_*M' \rightarrow X^{\oplus a'_e}$ and $\partial$ gives (1) with $C_1= C$.

(2) Let $s\colon X^{\oplus a''_e} \rightarrow F^e_*M''$ be a split injection. The exact sequence $\Hom_R(X, F^e_*I) \rightarrow \Hom_R(X, F^e_*M'') \xrightarrow{\partial_1} \Hom_R(X, F^e_*Q)$ and \Cref{prop: minimal number of generators and lifting}(2) give an invertible matrix $T_1$ such that $sT_1$ is a split injection all but at most $C p^{e(\dim(X)-1+\alpha(R))}$ of whose components are of the form $F^e_*\iota \circ t_i$ with $t_i \in \Hom_R(X, F^e_*I)$. The tuple $t= (t_i)_i$ is a split injection into $F^e_*I$: if $\sigma$ is a left inverse of $F^e_*\iota \circ t$, then $\sigma \circ F^e_*\iota$ is a left inverse of $t$. Next, the exact sequence $\Hom_R(X, F^e_*M') \xrightarrow{F^e_*\rho \circ -} \Hom_R(X, F^e_*I) \xrightarrow{\partial_2} \Ext^1_R(X, F^e_*K)$ and \Cref{prop: minimal number of generators and lifting}(2) give an invertible matrix $T_2$ such that $tT_2$ is a split injection all but at most $C p^{e(\dim(X)-1+\alpha(R))}$ of whose components lift along $F^e_*\rho \circ -$. Since $F^e_*\iota \circ (tT_2)= (F^e_*\iota \circ t)T_2$ is again a split injection into $F^e_*M''$, replacing the corresponding components of $sT_1$ by those of $(F^e_*\iota \circ t)T_2$ gives a split injection $X^{\oplus a''_e} \rightarrow F^e_*M''$ all but at most $2C p^{e(\dim(X)-1+\alpha(R))}$ of whose components lift along $F^e_*\phi \circ -$. This proves (2) with $C_2= 2C$.
\end{proof}

Using the lifting assertions above, we prove strong F-regularity of F-contributors of dimension at most two. This strong F-regularity, thanks to the theory developed in \Cref{se: strongly F regular modules}, implies the existence of F-splitting signatures towards such an F-contributor.

\begin{theorem}\label{cor: two dimensional F-contributors}
 Let $(R, \mathfrak m)$ be a local ring, $X$ be an indecomposable F-contributor of dimension at most two, and $M$ be a finitely generated $R$-module.
 Then 
 \begin{enumerate}
     \item $X$ is strongly F-regular.
     \item $\Ann(X)$ is a prime ideal.
        
     \item  If $\Ann(X)$ is a minimal prime of $M$, then there exists an $e_0 \in \mathbb N$, such that for all $e \geq e_0$, $X$ is a summand of $F^e_*M$.
        
     \item The F-splitting signature of $M$ towards $X$ at $\Ann(X)$ exists. It is positive if and only if $X \in \find (\Ann(X), M)$.
 \end{enumerate}
 \end{theorem}

\begin{proof}
When $\dim(X)=0$, since the annihilator of $X$ is radical (cf. \Cref{prop: associated primes and splitting}), it must be $\m$. Since $X$ is indecomposable, $X \cong R/\m$. Hence, (1), (2) follow from straightforward arguments and (3), (4) follow from \Cref{le: dimension zero contributor}. So, assume $\dim(X)>0$ from now on. 

(1) Let $f \in R$ be an element which is not in any minimal prime of $X$. Since $X$ has no embedded primes, the multiplication by $f$ map is injective and its cokernel is of dimension at most $\dim(X)-1$.
Since $X$ is an F-contributor, there exists a positive constant $D$ such that for infinitely many $e$, $\Fmult{e}{X}{X} \geq Dp^{e(\dim(X)+\alpha(R))}$. Therefore, by \Cref{th: most splittings lift in dimension two}(2), there exist an $e$ and a split injection $s\colon X \rightarrow F^e_*X$ and another map $s'\colon X \rightarrow F^e_*X$ such that the following diagram commutes
\[\begin{tikzcd}
	& X \\
	{F^e_*X} & {F^e_*X}
	\arrow["{s'}"', dashed, from=1-2, to=2-1]
	\arrow["s", from=1-2, to=2-2]
	\arrow["{F^e_*f}"', from=2-1, to=2-2]
\end{tikzcd}\]
Therefore, $X$ is strongly F-regular.

Now, (2) follows by \Cref{th: further properties of SFR modules}(\ref{part: decompose into local cohomology}) and (3) follows from \Cref{th: further properties of SFR modules}(\ref{part: positive signature}). Since $\delta_R(X)= \alpha(R_{\Ann(X)})$, (4) follows from \Cref{th: further properties of SFR modules}(\ref{part: aperiodicity of SFR contributors}).

\end{proof}

Now, we establish the rank formula towards F-contributors of dimension at most two (see \Cref{sse: rank formula}) and convergence rate (see \Cref{subsec: convergence rates}) towards the respective F-splitting signatures.

\begin{theorem}\label{th: rank formula in dimension two}
    Let $(R, \m)$ be a local ring and $X$ be an indecomposable  F-contributor of dimension at most two. Denote the prime ideal $\Ann(X)$ by $\p$. For any finitely generated $R$-module $M$ of dimension $\dim(X)$, 
    \[\Fmult{e}{M}{X}= \frac{\grk{\p}{M}}{\grk{\p}{X}}\gsig{X}{X}_{\p}p^{e(\dim(X)+ \alpha(R))}+ O(p^{e(\dim(X)-1+\alpha(R))}).\]
In particular, if $\grk{\p}{M}$ is nonzero, then $X \in \find(M)$.
\end{theorem}

\begin{proof}
Suppose that the desired equation for $\Fmult{e}{M}{X}$ holds. Since by \Cref{cor: two dimensional F-contributors} $\gsig{X}{X}_\p$ is positive, whenever $\grk{\p}{M}$ is nonzero, $X \in \find(M)$.

Now we argue for the desired equation. The desired assertion is nontrivial only when $\p \in \Ass(M)$. We therefore assume $\p \in \Ass(M)$. Since $\dim(M)= \dim(X)$, $\p \in \Min(M)$. Since $X$ is strongly F-regular, $X \in \find(M)$ by \Cref{th: universality of SFR modules}.  When $\dim(X)=0$, $X \cong R/\m$; see the beginning of the proof of \Cref{cor: two dimensional F-contributors}. Hence the result follows from \Cref{le: dimension zero contributor}. So, we assume $\dim(X)>0$ from now on.

Let $\Minh(M)= \{\p=\p_1, \p_2, \ldots, \p_n\}$ and define
\[
L \coloneqq \bigoplus_{j=2}^{n} (R/\p_j)^{\oplus c_j}, \qquad c_j \coloneqq \length_{R_{\p_j}}\big((M^{\oplus \grk{\p}{X}})_{\p_j}\big),
\]
and $\widetilde M' \coloneqq X^{\oplus \grk{\p}{M}} \oplus L$, $\widetilde M'' \coloneqq M^{\oplus \grk{\p}{X}}$. Since $\p_j$ are distinct minimal primes of $M$ of the same dimension and $\Ann(X)= \p$, we have $\Min(\widetilde M')= \{\p_1, \ldots, \p_n\}$, $L_{\p}=0$ and $X_{\p_j}=0$ for $j \geq 2$. Hence $\length(\widetilde M'_{\p_j})= \length(\widetilde M''_{\p_j})$ for every $j$: both sides equal $\grk{\p}{M}\grk{\p}{X}$ for $j=1$, and $c_j$ for $j \geq 2$. By \Cref{le: rank sequence}, there are $e_1 \in \mathbb N$ and an $R$-linear map
\[\widetilde\phi\colon F^{e_1}_*\widetilde M' \rightarrow F^{e_1}_*\widetilde M''\]
such that $\widetilde\phi_{\p_j}$ is an isomorphism for $1 \leq j \leq n$. Therefore, neither $\ker(\widetilde\phi)$ nor $\coker(\widetilde\phi)$ is supported at a prime of $\Minh(M)$, which are the only primes of dimension $\dim(X)$ in $\Supp(\widetilde M') \cup \Supp(\widetilde M'')$. So both have dimension at most $\dim(X)-1$. Moreover, $X \notin \find(L)$ since $\p \notin \Supp(L)$. Hence, by KRS, for all $e \geq 1$,
\[\Fmult{e}{F^{e_1}_*\widetilde M'}{X}= \grk{\p}{M}\Fmult{e+e_1}{X}{X} \quad \text{and} \quad \Fmult{e}{F^{e_1}_*\widetilde M''}{X}= \grk{\p}{X}\Fmult{e+e_1}{M}{X}.\]

Applying \Cref{th: most splittings lift in dimension two} to $\widetilde \phi$, we conclude that
\begin{equation}\label{eq: rank formula intermediate equation}
\grk{\p}{X}\Fmult{e}{M}{X}= \grk{\p}{M}\Fmult{e}{X}{X}+ O(p^{e(\dim(X)-1+ \alpha(R))}).
\end{equation}

Specialize the last equation to $M=F_*X$. Since $\grk{\p}{F_*X}= p^{\dim(X)+ \alpha(R)}\grk{\p}{X}$,
\[\Fmult{e+1}{X}{X}= p^{\dim(X)+ \alpha(R)}\Fmult{e}{X}{X}+ O(p^{e(\dim(X)-1+\alpha(R))}).\]
From here, using a straightforward recursion argument, we get
\[\Fmult{e}{X}{X}= p^{e(\dim(X)+ \alpha(R))}\gsig{X}{X}_\p+ O(p^{e(\dim(X)-1+\alpha(R))}).\]
Plugging this back into \Cref{eq: rank formula intermediate equation}, we conclude the desired equation. 

\end{proof}

We employ the rank formula above to prove a total rank formula:
\begin{corollary}\label{cor: total rank formula in dimension upto two}
Let $(R, \m)$ be a local ring, $M$ be a finitely generated $R$-module of dimension at most two and $\p \in \Minh(M)$. Suppose that there is an indecomposable F-contributor $X$ which has $\p$ as an associated prime. Then:
\begin{enumerate}
    \item the total rank formula at $\p$ holds for $M$;
    \item $\Mat(\p,M)$ has convergence radius $p^{-\alpha(R_\p)}$ and exactly one basic recurrent class, namely the connected class containing $X$, which is positive recurrent. In particular, $\Mat(\p,M)$ has no null $p^{-\alpha(R_\p)}$-recurrent class.
\end{enumerate}
\end{corollary}

\begin{proof}
Since $X$ is Cohen--Macaulay by \Cref{cor: Cohen--Macaulayless of contributors}, $\p \in \Minh(X)$, so $\dim(X) = \dim(R/\p) = \dim(M) \leq 2$. By \Cref{cor: two dimensional F-contributors}, $\Ann(X) = \p$ and $X \in \find(\p,M)$, and the \ces rank formula for $M$ towards $X$ at $\p$ holds by \Cref{th: rank formula in dimension two}. Now (1) follows from \Cref{prop: total rank formula vs rank formula}, and (2) from \Cref{pr: comparison of notions of rank formulas and consequences}(\ref{part: Cesaro rank formula implies a unique recurrent class}).
\end{proof}

\section{Questions and Conjectures}\label{se: questions and conjectures}

We assume in this section that $(R, \m)$ is local satisfying KRS.

\begin{conjecture}[Universality of F-contributors]\label{conj: universality conjecture} Let $X$ be an indecomposable finitely generated $R$-module which is an F-contributor. Then $X$ is strongly F-regular. In particular, $\Ann (X)$ is a prime ideal (\Cref{th: further properties of SFR modules}(\ref{part: decompose into local cohomology})) and for any finitely generated $R$-module $M$ such that $\Ann(X)$ is a minimal prime in the support of $M$, there exists a positive integer $e_0$, such that $X$ is an $R$-summand of $F^{e}_*M$, for all $e \geq e_0$.    
\end{conjecture}

If the `Universality Conjecture' holds, every F-contributor $X$ has period one (\Cref{th: further properties of SFR modules}(\ref{part: aperiodicity of SFR contributors})); hence, the limit defining F-splitting signatures $\lim_e \Fmult{e}{M}{X}/p^{e(\dim(X)+ \alpha(R))}$ exists (\Cref{cor: aperiodicity implies existence of limit}). We have confirmed the universality conjecture for F-contributors of dimension $\dim(R)$ over invariant rings of strongly F-regular rings (\Cref{cor: F contributors are SFR when there is a module-finite SFR extension}, \Cref{cor: invariant ring}) and in the dimension up to two case (\Cref{cor: two dimensional F-contributors}).

The next question is motivated by the observation that asymptotic invariants such as Hilbert--Kunz multiplicity, F-signature satisfy a rank formula (\cite{Tucker}) over a domain. 

\begin{conjecture}[Rank formula]\label{conj: rank formula} 
    Let $X$ be an indecomposable $R$-module which is an F-contributor and $\p$ be a prime ideal of dimension $\dim(X)$ in the support of $X$. Then for any finitely generated module $M$ of dimension $\dim(X)$, 
    \[\csig{M}{X}_{\p}= \frac{\grk{\p}{M}}{\grk{\p}{X}}\csig{X}{X}_{\p}.\]
Specializing to the case where $\p$ is a minimal prime of $R$, the previous part posits that the map $\bullet \mapsto \csig{\bullet}{X}_{\p}$ is additive over short exact sequences of finitely generated modules; see \Cref{prop: characterize rank formula}, \Cref{re: rank formula does not imply additivity}.  
\end{conjecture}

 \Cref{conj: rank formula} implies parts of \Cref{conj: universality conjecture}. It is not hard to show that if the rank formula holds, then for an F-contributor $X$ $\Ann(X)$ is a prime ideal (\Cref{pr: rank formula implies ann is prime}). Then rank formula implies that for any finitely generated module $M$ with $\Ann(X) \in \Min(M)$, $X$ is a summand of $F^e_*M$, for some $e$.

For evidence towards the rank formula, see \Cref{th: rank formula in dimension two}, \Cref{th: rank formula over sfr rings}.

\begin{conjecture}[Total rank formula]\label{conj: total rank formula}
Let $M$ be a finitely generated $R$-module and $\p$ be a minimal prime of dimension $\dim(M)$ in $\Supp(M)$. Suppose that $\find(M)$ contains an F-contributor which has $\p$ as its associated prime. Then the total rank formula at $\p$ holds for $M$. 
\end{conjecture}

Whenever the rank formula for $M$ towards an F-contributor in $\find(M)$, which has $\p$ as an associated prime, holds, the total rank formula also holds. See \Cref{prop: total rank formula vs rank formula} for further connections between the rank formula and the total rank formula.

The rest of the questions concern the nonnegative matrix $\Mat(M)$ or the dynamical system it induces. These reflect the way we expect $\Mat(M)$ to be distinguished among all countably sized matrices with nonnegative entries:

\begin{conjecture}[Geometric ergodicity of F-decomposition matrices]\label{conj: dynamical properties of the decomposition matrix} Let $M$ be a finitely generated $R$-module.
 For any $X \in \find(M)$ of period $\omega <\infty$, there is a nonnegative constant $\lambda_{M,X}< p^{\dim(M)+ \alpha(R)}$, such that
    \[\Fmult{t+ e\omega}{M}{X}= p^{(t+ e\omega)(\dim(M)+ \alpha(R))}\lim_n \frac{\Fmult{t+ n\omega}{M}{X}}{p^{(t+ n\omega)(\dim(M)+ \alpha(R))}}+ O(\lambda_{M,X}^{(t+ e\omega)}).\]    
\end{conjecture}

We have established geometric ergodicity when $R$ is strongly F-regular, $X \in \find(R)$, and $\dim(M)= \dim(R)$ in \Cref{th: convergence rate in the sfr case}, for F-contributors $X$ when $\dim(M) \leq 2$ in \Cref{th: rank formula in dimension two}, or when $M$ has FFRT (\Cref{cor: Yao}). We wonder, if in general $\lambda_{M,X}= p^{\dim(M)+ \alpha(R)-1}$ works in the Geometric Ergodicity Conjecture.

The validity of Geometric Ergodicity Conjecture implies that $\Mat(M)$ cannot have a $1/p^{\dim(M)+ \alpha(R)}$-null recurrent class. From a different direction, the validity of the total rank formula for $M$ towards every $\p \in \Minh(M)$ implies nonexistence of $1/p^{\dim(M)+ \alpha(R)}$-null recurrent classes for $\Mat(M)$; see \Cref{prop: consequences of the total rank formula}(\ref{part: total rank -> no null basic}). We recall that the total rank formula at $\p$ is in fact equivalent to nondissipativeness or Tweedie recurrence of the Markov chain induced by $\Mat(\p,M)$; see \Cref{sse: Total rank formula}, \Cref{re: probabilistic interpretation of rank formula and total rank formula}.

\begin{conjecture}[Nonexistence of null recurrence]\label{conj: no null recurrence}
Let $M$ be a finitely generated $R$-module. Then $\Mat(M)$ does not have any $1/p^{\dim(M)+ \alpha(R)}$-null recurrent classes.    
\end{conjecture}

\Cref{cor: unique maximal dimensional connected class} includes a class of examples where \Cref{conj: no null recurrence} holds, although we do not know geometric ergodicity or the total rank formula yet. We also have the following simple observation.

\begin{lemma}
    Suppose that $R$ is an F-finite local ring of finite Cohen--Macaulay type. Then for any full-dimensional finitely generated $R$-module $M$, the matrix $\Mat(M)$ does not have $p^{-\delta_R(R)}$-null recurrent classes. 
\end{lemma}
\begin{proof}
    By Corollary~\ref{cor: properties of maximal recurrent classes in the whole matrix} a $p^{-\delta(M)}$-recurrent class $\cC$ is final. Since any $N \in \cC$ is Cohen--Macaulay, $\Mat(N)$ is an irreducible finite matrix, thus $\cC$ is positive recurrent.
\end{proof}

The next question links the notion of essentially FFRT rings introduced in \Cref{sse: essentially FFRT} to the existence of Frobenius limit --- a notion introduced in \cite{HashimotoSymonds} in the graded setting and adapted to the local setting as follows. Let $\Theta (R)$ 
be the additive Grothendieck group of finitely generated $R$-modules with $\mathbb{R}$ coefficients, i.e., the $\mathbb{R}$-vector space with a basis corresponding to isomorphism classes of
indecomposable $R$-modules $\{N_i\}_{i \in I}$. 
A finitely generated $R$-module $M$ is said to have a \textit{Frobenius limit} 
$\FL(M) \in \Theta(R)$ if elements $p^{-e \delta_R(R)}[F_*^e M]$
converge to $\FL(M)$ in the norm 
\[
|| \sum_{i \in I}c_i N_i|| = \sum_{i \in I} |c_i| \length (N_i/\m N_i).
\]

\begin{question}\label{qe: essentially FFRT vs Frobenius limit}
Let $M$ be a finitely generated $R$-module of dimension $\dim(R)$. Is the existence of $\FL(M)$ equivalent to $M$ possessing the essentially FFRT property?    
\end{question}

Clearly, any FFRT module has a Frobenius limit. However, this is a strictly weaker property, e.g., modules, Frobenius limits exist when the module is the invariant ring of a linear finite group action on a power series ring (\cite{HashimotoSymonds}), but such rings need not have the FFRT property by \cite[Theorem~3.1]{HashimotoSingh}.

All the examples we are aware of give an affirmative answer to the following question, although, we do not know the answer in general.

\begin{question}
When $R(\Mat(M))$ is finite, is it always a power of $p$?    
\end{question}


\bibliographystyle{alphaurl}
\bibliography{PF}

\end{document}